\documentclass[a4paper,11pt,final]{article}

\usepackage{graphicx}
\usepackage{amssymb,amsthm,amsmath}
\usepackage{enumerate}
\usepackage{a4wide}
\usepackage{hyperref}

\numberwithin{equation}{section}
\allowdisplaybreaks[1]

\theoremstyle{plain}
\newtheorem{proposition}{Proposition}[section]
\newtheorem{theorem}[proposition]{Theorem}
\newtheorem{lemma}[proposition]{Lemma}

\theoremstyle{definition}
\newtheorem{definition}[proposition]{Definition}

\theoremstyle{remark}

\renewenvironment{proof}{\smallskip\noindent\emph{\textbf{Proof.}}\hspace{1pt}}%
{\hspace{-5pt}{\nobreak\quad\nobreak\hfill\nobreak$\square$\vspace{8pt}%
\par}\smallskip\goodbreak}

\newenvironment{proofof}[1]{\smallskip\noindent\emph{\textbf{Proof of #1.}}%
\hspace{1pt}}{\hspace{-5pt}{\nobreak\quad\nobreak\hfill\nobreak%
$\square$\vspace{8pt}\par}\smallskip\goodbreak}

\renewcommand{\leq}{\leqslant}
\renewcommand{\geq}{\geqslant}

\newcommand{\C}[1]{\mathbf{C}^{#1}}
\newcommand{\Cc}[1]{\mathbf{C}_c^{#1}}
\newcommand{\modulo}[1]{{\left|#1\right|}}
\newcommand{\norma}[1]{{\left\|#1\right\|}}
\newcommand{\reali}{{\mathbb{R}}}
\newcommand{\naturali}{{\mathbb{N}}}
\newcommand{\rpic}{{\mathbb{R}^+}}

\newcommand{\BV}{\mathbf{BV}}

\renewcommand{\epsilon}{\varepsilon}
\renewcommand{\phi}{\varphi}
\renewcommand{\theta}{\vartheta}
\renewcommand{\L}[1]{{\mathbf{L}^#1}}

\newcommand{\W}[2]{{\mathbf{W}^{#1,#2}}}

\newcommand{\Lloc}[1]{{\mathbf{L}_{loc}^{#1}}}
\newcommand{\tv}{\mathinner{\rm TV}}
\newcommand{\caratt}[1]{{\chi_{\strut#1}}}

\renewcommand{\div}{\mathinner{\rm div}}

\newcommand{\spt}{\mathop{\rm spt}}
\newcommand{\sgn}{\mathop{\rm sgn}}

\renewcommand{\d}[1]{\mathinner{\mathrm{d}{#1}}}

\begin{document}

\title{A Class of Non-Local Models for Pedestrian Traffic}

\author{Rinaldo M.~Colombo$^1$, Mauro Garavello$^2$, Magali
  L\'ecureux-Mercier$^3$}

\footnotetext[1]{Dipartimento di Matematica, Universit\`a degli studi
  di Brescia, Italia. \texttt{rinaldo@ing.unibs.it}}

\footnotetext[2]{Di.S.T.A., Universit\`a del Piemonte Orientale,
  Italia, \texttt{mauro.garavello@mfn.unipmn.it}}

\footnotetext[3]{Universit\'e d'Orl\'eans, UFR Sciences, B\^atiment de
  math\'ematiques - Rue de Chartres B.P. 6759 - 45067 Orl\'eans cedex
  2, France, \texttt{magali.lecureux-mercier@univ-orleans.fr}}

\maketitle

\begin{abstract}
  We present a new class of macroscopic models for pedestrian
  flows. Each individual is assumed to move towards a fixed target,
  deviating from the best path according to the instantaneous crowd
  distribution. The resulting equation is a conservation law with a
  nonlocal flux. Each equation in this class generates a Lipschitz
  semigroup of solutions and is stable with respect to the functions
  and parameters defining it. Moreover, key qualitative properties
  such as the boundedness of the crowd density are proved. Specific
  models are presented and their qualitative properties are shown
  through numerical integrations.

  \medskip

  \noindent\textit{2000~Mathematics Subject Classification:} 35L65,
  90B20.

  \medskip

  \noindent\textit{Keywords:} Crowd Dynamics, Macroscopic Pedestrian
  Model, Non-Local Conservation Laws.
\end{abstract}

\section{Introduction}
\label{sec:I}

\par From a macroscopic point of view, a moving crowd is described by
its density $\rho = \rho(t,x)$, so that for any subset $A$ of the
plane, the quantity $\int_A \rho(t,x) \d{x}$ is the total number of
individuals in $A$ at time $t$. In standard situations, the number of
individuals is constant, so that conservation laws of the type
$\partial_t \rho + \div_x (\rho\, {\mathbf{v}} ) = 0$ are the natural
tool for the description of crowd dynamics. A key issue is the choice
of the speed $\mathbf{v}$, which should describe not only the target
of the pedestrians and the modulus of their speed, but also their
attitude to adapt their path choice to the crowd density they estimate
to find along this path.

Our starting point is the following Cauchy problem for the
conservation law
\begin{equation}
  \label{eq:General}
  \left\{
    \begin{array}{l}
      \partial_t \rho
      +
      \div
      \left(
        \rho\, v (\rho) \, \left(\nu (x) + \mathcal{I} (\rho)\right)
      \right)
      = 0\,,
      \\
      \rho(0,x) = \rho_0 (x)\,.
    \end{array}
  \right.
\end{equation}
The scalar function $\rho \mapsto v(\rho)$ describes the modulus of
the pedestrians' speed, independently from geometrical considerations.
In other words, an individual at time $t$ and position $x \in
\reali^N$ moves at the speed $v\left(\rho (t,x)\right)$ that depends
on the density $\rho (t,x)$ evaluated at the same time $t$ and
position $x$. Given that the density is $\rho$, the vector $\nu (x) +
\mathcal{I} (\rho)$ describes the direction that the individual
located at $x$ follows and has norm (approximately) $1$. More
precisely, the individual at position $x$ and time $t$ is assumed to
move in the direction $\nu(x) + \left( \mathcal{I} \left( \rho(t)
  \right) \right) (x)$.

In situations like the evacuation of a closed space $\Omega \subset
\reali^N$, it is natural to assume that the first choice of each
pedestrian is to follow a path optimal with respect to the visible
geometry, for instance the geodesic. As soon as walls or obstacles are
relevant, it is necessary to take into consideration the discomfort
felt by pedestrians walking along walls or too near to corners, see
for instance~\cite{HelbingJohansson, HoogendoornBovy2002} and the
references therein.

The vector $\left( \mathcal{I} \left( \rho (t) \right) \right) (x)$
describes the deviation from the direction $\nu(x)$ due to the density
distribution $\rho (t)$ at time $t$. Hence, the operator $\mathcal{I}$
is in general \emph{nonlocal}, so that $\left(\mathcal{I}\left(\rho
    (t)\right)\right) (x)$ depends on all the values of the density
$\rho (t)$ at time $t$ in a neighborhood of $x$.  More formally, it
depends on all the function $\rho (t) \in \L1 (\reali^N;[0,R])$ and
not only on the value $\rho (t,x) \in [0,R]$. The case in which
$\mathcal{I} = 0$ is equivalent to assume that the paths followed by
the individuals are chosen \emph{a priori}, independently from the
dynamics of the crowd.

Here we present two specific choices that fit in~\eqref{eq:General}. A
first criterion assumes that each individual aims at avoiding high
crowd densities. Fix a mollifier $\eta$. Then, the convolution
$(\rho*\eta)$ is an average of the crowd density around $x$. This
leads to the natural choice
\begin{equation}
  \label{eq:IGood}
  \mathcal{I} (\rho)
  =
  - \epsilon \,
  \frac{\nabla (\rho*\eta)}{\sqrt{1+\norma{\nabla(\rho*\eta)}^2}} \,,
\end{equation}
related to~\cite{BressanColomboPedoni}, which states that individuals
deviate from the optimal path trying to avoid entering regions with
higher densities. Through numerical integrations, below we provide
examples of solutions to~\eqref{eq:General}--\eqref{eq:IGood}. They
show the interesting phenomenon of \emph{pattern formation}. In the
case of a crowd walking along a corridor, coherently with the
experimental observation described in the literature, see for
instance~\cite{HelbingEtAl2002, HelbingJohansson2007,
  HoogendoornBovy2003, PiccoliTosin2009}, the solution
to~\eqref{eq:General}--\eqref{eq:IGood} self-organizes into lanes. The
width of these lanes depends on the size of the support of the
averaging kernel $\eta$. This feature is stable with respect to strong
variations in the initial datum and also in the geometry. Indeed, we
have lane formation also in the case of the evacuation of a room, when
the crowd density sharply increases in front of the
door. Section~\ref{subs:LaneFormation} is devoted to this property.

From the analytical point of view, we note that the convolution term
in~\eqref{eq:IGood} seems not sufficient to regularize
solutions. Indeed, the present analytical framework is devised to
consider solutions in $\L1 \cap \BV$. Both in the case of a crush in
front of an exit (Section~\ref{subs:Evacuation}) and in the specific
example in Section~\ref{subs:Rise}, numerical simulations highlight
that the space gradient of $\rho$ increases dramatically.

According to~\eqref{eq:IGood}, pedestrians evaluate the crowd density
all around their position. When restrictions on the angle of vision
are relevant, the following choice is reasonable:
\begin{equation}
  \label{eq:IPT}
  \mathcal{I} (\rho)
  =
  -\epsilon\,
  \frac{\nabla\,\int_{\reali^N} \rho(y) \,
    \eta(x-y)  \,
    \phi\!\left((y-x) \cdot g(x)\right)
    \d{y}}
  {\sqrt{1+\norma{\nabla\,\int_{\reali^N} \rho(y) \,
        \eta(x-y)  \,
        \phi\!\left((y-x) \cdot g(x)\right)
        \d{y}}^2}} \,.
\end{equation}
Here, $\eta$ is a fixed mollifier as above and the smooth function
$\phi$ weights the deviation from the preferred direction $g (x)$.

We note that the constructions in~\cite{Hughes1, Hughes2}
and~\cite{DiFrancesco} fit in the present setting. Indeed, there the
following choices were considered:
\begin{displaymath}
  \left\{
    \begin{array}{l}
      \partial_t \rho-\div \left(\rho \, f^2(\rho) \, \nabla\phi\right)=0
      \\
      \modulo{\nabla\phi}= 1/f(\rho)
    \end{array}
  \right.
  \qquad \mbox{ and } \qquad
  \left\{
    \begin{array}{l}
      \partial_t \rho-\div \left(\rho \, f^2(\rho) \, \nabla\phi\right)=0
      \\
      -\epsilon \Delta \phi
      + \modulo{\nabla \phi}^2
      = 1/\left(f(\rho)+\epsilon\right)^2 \,.
    \end{array}
  \right.
\end{displaymath}
The former admits an immediate interpretation: the direction $\nabla
\phi$ of the speed $\mathbf{v}$ is chosen along the solutions to the
eikonal equations, i.e.~all pedestrians follow the shortest path,
weighing at every instant the length of paths with the amount of
people that are moving along it. In the former case, $\nu (x)=0$ and
$\mathcal{I} (\rho)$ is the gradient of the solution to the eikonal
equation with $1/f (\rho)$ in the right hand side, while in the latter
case $\nu (x)=0$ and $\mathcal{I} (\rho)$ is the gradient of the
solution to an elliptic partial differential equation.

The model introduced in~\cite{CristianiPiccoliTosin, PiccoliTosin}
relies on this measure valued conservation law:
\begin{displaymath}
  \partial_t \mu +\div (\mu \, v) = 0
  \quad \mbox{ where } \quad
  v
  =
  \nu(x)
  +
  \int_{\reali^2} f(\modulo{x-y}) \,
  \phi\left( (y-x) \cdot \nu(x) \right) \frac{x-y}{\norma{x-y}}\d\mu_t(y)\,.
\end{displaymath}
Here the unknown is a map $\mu \colon [0,T] \to \mathcal{M}
(\reali^N;\reali)$ assigning at every time $t$ a positive measure $\mu
(t)$ which substitutes the crowd density, in the sense that the amount
of people that at time $t$ are in $A$ is $\left(\mu(t)\right)
(A)$. Contrary to this model, the present framework is the space of
$\L1$ densities and physical \emph{a priori} $\L\infty$ bounds on the
solutions to~\eqref{eq:General} are rigorously proved, preventing any
focusing effect as well as the rise of any Dirac delta. Moreover,
below we prove global in time existence of solutions, their continuous
dependence from the initial data and their stability with respect to
variations in the speed law $\mathbf{v}$.

In~\cite{CosciaCanavesio}, the geometric part $\nu$ of the speed is
chosen \emph{a priori}, while its modulus depend on the density as
well as on the gradient of the density:
\begin{displaymath}
  \partial_t\rho
  +
  \div \left(\rho \, \phi(\rho,\nabla \rho) \, \nu(x)\right)
  =
  0 \,.
\end{displaymath}
On the contrary, the model introduced in~\cite{ColomboHertyMercier}
postulates a nonlocal dependence of the speed from the density:
\begin{displaymath}
  \partial_t \rho
  +
  \div \left(\rho \, v (\rho*\eta) \, \nu (x)\right)
  =
  0 \,,
\end{displaymath}
which amounts to assume that pedestrians choose their behavior
according to evaluations of an \emph{average} of the density around
their position, rather than according to the density at their place.
The following second order model was presented in~\cite{BellomoDogbe,
  Dogbe_Cauchy} and does not fit in~\eqref{eq:General}:
\begin{displaymath}
  \left\{
    \begin{array}{l}
      \partial_t \rho +\div (\rho v) = 0
      \\
      \partial_t v+ (v\cdot \nabla ) v = F(\rho,\nabla\rho,  v) \,,
    \end{array}
  \right.
\end{displaymath}
we refer to the review~\cite{BellomoDogbe_review} for further
details. To underline the variety of analytical techniques with which
crowd dynamics has been tackled, we recall the further approaches:
optimal transport in~\cite{ButtazzoJimenezOudet}, the mean field limit
in~\cite{Dogbe_MeanField}, the functional analytic one
in~\cite{MauryChupin, MauryChupinVenel} and the nonclassical shocks
used in~\cite{Chalons, ColomboRosini2005, ColomboRosini}.

Throughout this work, we set all statements in all $\reali^N$. In
Section~\ref{sec:Des} we show how the present framework is able to
take into consideration the presence of various constraints. In the
case of pedestrian dynamics, for instance, this amounts to prove that
no individual passes through the walls of a given room, provided the
initial datum is assigned inside it. Concerning the dimension, our
main applications are referred to the case of crowd dynamics,
i.e.~$N=2$. Nevertheless, from the analytical point of view,
considering the case of a general $N$ does not add any
difficulty. Furthermore, we believe that the present setting can be
reasonably applied also to, say, fishes and birds moving in
$\reali^3$, for example in predator-prey like situations as described
in~\cite{ColomboMercier}.

The next section is devoted to the analytical properties
of~\eqref{eq:General}: well posedness and stability. The general
theory is particularized to specific examples in
Section~\ref{sec:Des}, where the presence of obstacle (walls) is
considered. Sample numerical integrations are provided in
Section~\ref{sec:NR}. The final Section~\ref{sec:TD} collects the
analytical proofs. We defer to Appendix~\ref{app:G} further remarks
related to the geometry of the physical domain.

\section{Analytical Results}
\label{sec:AR}

In the following, $N\in \naturali \setminus \{0\}$ is the (fixed)
space dimension. We denote $\reali^+ = \left[0, +\infty\right[\,$; the
open ball in $\reali^N$ centered at $x$ and with radius $r>0$ is
$B(x,r)$ and we let $W_N = \int_0^{\pi/2} (\cos \theta)^N\,
\d{\theta}$. As usual, we denote $S^{N-1} = \left\{x\in \reali^N
  \colon \norma{x}=1\right\}$.

The density $\rho$ can be defined as the fraction of space occupied by
pedestrians, so that $\rho$ turns out to be a nondimensional scalar in
$[0,1]$. Otherwise, it can be useful to think at $\rho$ as measured
in, say, $\mbox{individuals}/m^2$ and varying in $[0,R]$, with $R>0$
being a given maximal density, for example $R = 8\;
\mbox{individuals}/m^2$.

Our first step in the study of~\eqref{eq:General} is the formal
definition of solution.
\begin{definition}
  \label{def:WS}
  Fix a positive $T$ and an initial datum $\rho_0 \in \L1
  (\reali^N;[0,R])$. A function $\rho \in \C0\left([0,T];\L1
    (\reali^N;\reali)\right)$ is a \emph{weak entropy solution}
  to~\eqref{eq:General} if it is a Kru\v zkov solution to the Cauchy
  problem for the scalar conservation law
  \begin{equation}
    \label{eq:Kruzkov}
    \left\{
      \begin{array}{l}
        \partial_t \rho + \div \left( \rho\, v (\rho) \, w(t,x) \right) =0
        \\
        \rho(0,x) = \rho_0 (x)
      \end{array}
    \right.
  \end{equation}
  where $w(t,x) = \nu (x) + \left(\mathcal{I} \left( \rho(t) \right)
  \right) (x)$.
\end{definition}

\noindent In other words, recalling~\cite[Definition~1]{Kruzkov}, for
all $k \in \reali$, for all $\phi \in \Cc\infty(\left]-\infty,T
\right] \times \reali^N; \reali^+)$,
\begin{eqnarray*}
  \!\!\!\!\!\!\!\!
  & &
  \int_0^T \!\!\! \int_{\reali^N}\!
  \left[
    \modulo{\rho-k} \, \partial_t\phi
    +
    \left[
      \left(\rho \, v(\rho) - k \, v(k)\right) \, w(t,x) \cdot \nabla\phi
      -
      k\,v(k)\div w(t,x) \, \phi
    \right]
    \sgn(\rho-k)
  \right] \d{x}\d{t}
  \\
  \!\!\!\!\!\!\!\!
  & &
  + \int_{\reali^N} \modulo{\rho_0 (x)-k} \, \phi (0,x) \d{x}
  \geq 0 .
\end{eqnarray*}

\noindent On the functions defining the general
model~\eqref{eq:General}, we introduce the following hypotheses:
\begin{description}
\item[(v)] $v \in \C2\left(\reali;\reali\right)$ is non increasing, $v
  (0) = V$ and $v (R)=0$ for fixed $V,R>0$.
\item[($\boldsymbol{\nu}$)] $\nu \in (\C2 \cap \W1\infty) (\reali^N;
  \reali^N)$ is such that $\div \nu \in (\W11 \cap \W1\infty)
  (\reali^N; \reali)$.
\item[(I)] $\mathcal{I} \in \C0\left(\L1(\reali^N;[0,R]);
    \C2(\reali^N;\reali^N)\right)$ satisfies the following estimates:
  \begin{enumerate}[\bf({I}.1)]
  \item There exists an increasing $C_I \in \Lloc\infty(\rpic; \rpic)$
    such that, for all $r \in \L1(\reali^N;[0,R])$,
    \begin{eqnarray*}
      \norma{\mathcal{I}(r)}_{\W1\infty}
      \leq
      C_I(\norma{r}_{\L1})
      \quad \mbox{ and } \quad
      \norma{\mathcal{\div I}(r)}_{\L1}
      \leq
      C_I(\norma{r}_{\L1})\,.
    \end{eqnarray*}
  \item There exists an increasing $C_I \in \Lloc\infty(\rpic; \rpic)$
    such that, for all $r \in \L1(\reali^N;[0,R])$,
    \begin{displaymath}
      \norma{\nabla \div  \mathcal{I}(r)}_{\L1}
      \leq
      C_I(\norma{r}_{\L1})\,.
    \end{displaymath}
  \item There exists a constant $K_I$ such that for all $r_1, r_2\in
    \L1(\reali^N; [0,R])$,
    \begin{eqnarray*}
      \norma{\mathcal{I}(r_1) - \mathcal{I}(r_2)}_{\L\infty}
      & \leq &
      K_I \cdot \norma{r_1-r_2}_{\L1} \,,
      \\
      \norma{\mathcal{I}(r_1) - \mathcal{I}(r_2)}_{\L1}
      +
      \norma{\div(\mathcal{I}(r_1 ) - \mathcal{I}(r_2))}_{\L1}
      &\leq &
      K_I \cdot \norma{r_1-r_2}_{\L1} \,.
    \end{eqnarray*}
  \end{enumerate}
\end{description}

\noindent Furthermore, throughout we denote by $q$ the map $q \colon [0,R] \mapsto \reali$ defined by $q (\rho) = \rho \, v (\rho)$.

As a first justification of these conditions, note that they make
Definition~\ref{def:WS} acceptable.

\begin{lemma}
  \label{lem:Kruzkov}
  Fix a positive $T$.  Let \textbf{(v)}, \textbf{($\boldsymbol{\nu}$)}
  and~\textbf{(I.1)} hold. Choose an arbitrary density $r \in \C0
  \left([0,T]; \L1 (\reali^N;[0,R])\right)$.  Then, the Cauchy problem
  \begin{equation}
    \label{eq:K}
    \left\{
      \begin{array}{l}
        \partial_t \rho + \div \left( \rho\, v (\rho) \, w(t,x) \right) =0
        \\
        \rho(0,x) = \rho_0 (x)
      \end{array}
    \right.
    \qquad \mbox{ with } \qquad
    w(t,x) = \nu
    (x) + \left(\mathcal{I} \left(r(t)\right)\right) (x)
  \end{equation}
  satisfies the assumptions of Kru\v zkov
  Theorem~\cite[Theorem~5]{Kruzkov}.
\end{lemma}

This lemma is proved in Section~\ref{sec:TD} below. In
Section~\ref{sec:Des} we show that the above assumption~\textbf{(I)}
allows to comprehend physically reasonable cases.

The next result is devoted to the proof of existence and uniqueness
for~\eqref{eq:General}. It is obtained through Banach Fixed Point
Theorem.

\begin{theorem}
  \label{thm:existence}
  Let \textbf{(v)}, \textbf{($\boldsymbol{\nu}$)} and~\textbf{(I)}
  hold. Choose any $\rho_0 \in (\L1 \cap \BV) (\reali^N;[0,R])$. Then,
  there exists a unique weak entropy solution $\rho \in
  \C0 \left(\rpic; \L1(\reali^N; [0,R]) \right)$ to~(\ref{eq:General}). Moreover,
  $\rho$ satisfies the bounds
  \begin{eqnarray*}
    \norma{\rho(t)}_{\L1}
    & = &
    \norma{\rho_0}_{\L1}\,,  \mbox{ for a.e. } t\in \rpic\,,
    \\
    \tv(\rho(t))
    & \leq &
    \tv(\rho_0) \, e^{k t}
    + te^{k t} NW_N \norma{q}_{\L\infty([0,R])}
    \left(
      \norma{\nabla\div \nu}_{\L1}
      +
      C_I(\norma{\rho_0}_{\L1})
    \right)\,,
  \end{eqnarray*}
  where $k= (2N+1) \norma{q'}_{\L\infty([0,R])} \left( \norma{\nabla
      \nu}_{\L\infty} + C_I(\norma{\rho_0}_{\L1}) \right)$.
\end{theorem}

Using the techniques in~\cite{ColomboHertyMercier,
  ColomboMercierRosini}, we now obtain the continuous dependence of
the solution to~\eqref{eq:General} from the initial datum and its
stability with respect to $v$, $\nu$ and $\mathcal{I}$ in the natural
norms.

\begin{theorem}
  \label{thm:stability}
  Let \textbf{(v)}, \textbf{($\boldsymbol{\nu}$)} and~\textbf{(I)} be
  satisfied by both systems
  \begin{displaymath}\label{eq:2eq}
    \left\{
      \begin{array}{@{}l@{}}
        \partial_t \rho
        +
        \div
        \left[
          \rho\, v_1 (\rho) \, \left(\nu_1 (x) + \mathcal{I}_1 (\rho)\right)
        \right]
        = 0
        \\
        \rho(0,x) = \rho_{0,1} (x)
      \end{array}
    \right.
    \quad \mbox{ and } \quad
    \left\{
      \begin{array}{@{}l@{}}
        \partial_t \rho
        +
        \div
        \left[
          \rho\, v_2 (\rho) \, \left(\nu_2 (x) + \mathcal{I}_2 (\rho)\right)
        \right]
        = 0
        \\
        \rho(0,x) = \rho_{0,2} (x)
      \end{array}
    \right.
  \end{displaymath}
  with $\rho_{0,1},\rho_{0,2} \in (\L1 \cap \BV)
  (\reali^N;[0,R])$. Then, the two solutions $\rho_1$ and $\rho_2$
  satisfy
  \begin{eqnarray*}
    \norma{\rho_1 (t) - \rho_2 (t)}_{\L1}
    & \leq &
    C (t) \big(
    \norma{\rho_{0,1}-\rho_{0,2}}_{\L1}
    +
    \norma{q_1 - q_2}_{\W1\infty}
    \\
    & &
    \qquad\quad
    +
    \norma{\nu_1 - \nu_2}_{\L\infty}  +
    \norma{\div(\nu_1 - \nu_2)}_{\L1}
    +
    d\left(\mathcal{I}_1,\mathcal{I}_2\right)
    \big)
  \end{eqnarray*}
  where $\displaystyle d (\mathcal{I}_1, \mathcal{I}_2) = \sup \left\{
    \norma{\mathcal{I}_1 (\rho) -\mathcal{I}_2(\rho)}_{\L\infty}+
    \norma{\div \mathcal{I}_1 (\rho) -\div\mathcal{I}_2(\rho)}_{\L1}
    \colon \rho \in \L1 (\reali^N;[0,R]) \right\}$, the map $C \in \C0
  (\reali^+;\reali^+)$ vanishes at $t=0$ and depends on
  $\tv(\rho_{0,1})$, $\norma{\rho_{0,1}}_{\L1}$,
  $\norma{\nu_1}_{\L\infty}$, $\norma{\div \nu_1}_{\W11}$,
  $\norma{q_1}_{\W1\infty}$, $\norma{q_2}_{\W1\infty}$.
\end{theorem}

Thanks to these stability results, several control problems can be
proved to admit a solution through a direct application of
Weierstra\ss~Theorem. A possible standard application could be the
minimization of the evacuation time from a given room. Describing the
actions of a controller able to determine the initial pedestrians'
distribution and/or their preferred paths, one is lead to an optimal
control problem with the initial datum $\rho_0$ and the vector field
$\nu$ as control parameters, for instance. Without any loss of
generality, the compact sets on which the optimization is made can be
$\left\{ \rho \in (\L1\cap \BV) (\reali^N; [0,1]) \colon \tv (\rho) <
  M\right\}$ for $\rho_0$ and $\left\{\nu \in \C2 \left(\reali^N;
    \overline{B (0,1)} \right) \colon \norma{\nabla^3\nu} \leq M \right\}$
for $\nu$, where $M>0$ is arbitrary.

A different problem solved by the same analytical techniques is that
of the dynamic parameter estimation. Once real data are available, one
is left with the problem of determining the various parameters
entering $\nu$, $v$ or $\mathcal{I}$. Theorem~\ref{thm:stability}
ensures the existence of the parameters that allows a best agreement
between the solutions to~\eqref{eq:General} and the data.

\section{The Models}
\label{sec:Des}

This section is devoted to the study of specific cases
of~\eqref{eq:General}. Aiming at real applications, it is necessary to
take into consideration the various constraints that are present on
the movement of pedestrians. Therefore, we introduce the subset
$\Omega$ of $\reali^N$ which characterizes the region reachable to any
individual. The boundary $\partial\Omega$ consists of \emph{walls}
that can not be crossed by any individual. On the set $\Omega$ we
require that
\begin{description}
\item[($\boldsymbol{\Omega}$)] $\Omega \subseteq \reali^N$ is the
  closure of a non empty connected open set. If $\partial\Omega$ is
  not empty, there exists a positive $r_\Omega$ such that the function
  \begin{displaymath}
    \begin{array}{ccccc}
      d_{\partial\Omega} &\colon & B (\partial\Omega,r_\Omega) \cap \Omega
      & \to & \reali^+
      \\
      & & x
      & \mapsto & \inf \left\{d (x,w) \colon w \in \partial\Omega\right\}
    \end{array}
  \end{displaymath}
  is of class $\C2\left(B (\partial\Omega,r_\Omega) \cap \Omega;
    \reali^+\right)$.
\end{description}
Note that we do not require $\Omega$ to be bounded. This assumptions
allows us to introduce the inward normal
\begin{equation}
  \label{eq:n}
  n (x) = \nabla d_{\partial\Omega} (x)
\end{equation}
on all the strip $B (\partial\Omega,r_\Omega) \cap \Omega$. Moreover,
$\norma{n (x)} = 1$.

For the present models to be acceptable, it is mandatory that no
individual enters any wall, provided the initial datum is supported
inside the physically admissible space. Analytically, this is
described by the following \emph{invariance} property:
\begin{description}
\item{\textbf{(P)}} The model~\eqref{eq:General} is invariant with
  respect to $\Omega$ if
  \begin{displaymath}
    \spt \rho_0 \subset \Omega
    \qquad\qquad \Longrightarrow  \qquad\qquad
    \spt \rho(t) \subset \Omega
    \mbox{ for all } t\geq 0\,.
  \end{displaymath}
\end{description}

\noindent Below, we show that theorems~\ref{thm:existence}
and~\ref{thm:stability} can be applied to reasonable crowd dynamics
models, in the sense that the presence of a physically admissible set
$\Omega$ is also considered and its invariance in the sense
of~\textbf{(P)} is proved. More precisely, we show below the following
sufficient condition for invariance.

\begin{proposition}
  \label{prop:invariance}
  Let $\nu$, $\mathcal{I}$ and $\Omega$
  satisfy~\textbf{($\boldsymbol{\nu}$)}, \textbf{($\boldsymbol{I}$)}
  and~\textbf{(\/$\boldsymbol{\Omega}$)}. If for all $x \in
  \partial\Omega$ and $\rho \in \L1 (\reali^N;[0,R])$ with $\spt\rho
  \subseteq \Omega$
  \begin{equation}
    \label{eq:Nagumo}
    \left(
      \nu (x)
      +
      \left(\mathcal{I} (\rho)\right) (x)
    \right)
    \cdot n (x)
    \geq 0
  \end{equation}
  then, property~\textbf{(P)} holds.
\end{proposition}

\subsection{The Model~\eqref{eq:General}--\eqref{eq:IGood}}
\label{sub:Good}

The starting point for~\eqref{eq:General}--\eqref{eq:IGood} is
provided by the degenerate parabolic model
\begin{displaymath}
  \partial_t \rho
  +
  \div\left(
    \rho\, v (\rho) \bigg(
    \nu (x)
    -
    \epsilon \frac{\nabla \psi(\rho)}{\sqrt{1+\norma{\nabla\psi(\rho)}^2}}
    \bigg)
  \right)
  =
  0
\end{displaymath}
introduced in~\cite{BressanColomboPedoni}, which fits
in~\eqref{eq:General} with $\mathcal{I} (\rho) = -\epsilon \,
\nabla\psi(\rho) \, \left/ \sqrt{1+\norma{\nabla\psi(\rho)}^2}
\right.$, motivated by the desire of each individual to avoid entering
regions occupied by a high crowd density. Here, $\psi$ is a suitable
weight function. Assuming that each pedestrian reacts to evaluations
of \emph{averages} of the density, we obtain
\begin{displaymath}
  \partial_t \rho
  +
  \div\left(
    \rho\, v (\rho) \bigg(
    \nu (x)
    -
    \epsilon \frac{\nabla (\psi(\rho)*\eta)}{\sqrt{1+\norma{\nabla(\psi(\rho)*\eta)}^2}}
    \bigg)
  \right)
  =
  0\,.
\end{displaymath}
Here, we avoid the introduction of $\psi$ to limit the analytical
technicalities, obtaining~\eqref{eq:General}--\eqref{eq:IGood}. We
assume throughout that $\epsilon>0$ is fixed and that the mollifier
$\eta$ satisfies
\begin{description}
\item[($\boldsymbol{\eta}$)] $\eta \in \Cc3 (\reali^N;\reali^+)$ with
  $\int_{\reali^N}\eta (x)\d{x}=1$.
\end{description}
\noindent This convolution kernel has a key role: the individual at
$x$ deviates from the optimal path considering the crowd present in
the region $x-\spt\eta$, when no walls are present. The value $\eta
(\xi)$ is the relevance that the individual at $x$ gives to the
density $\rho(x-\xi)$ located at $x-\xi$.

The next result shows that the present model fits in the framework
described in Section~\ref{sec:AR} when applied on all of $\reali^N$.

\begin{lemma}
  \label{lem:Good}
  Fix $\epsilon>0$ and let $\eta$
  satisfy~\textbf{($\boldsymbol{\eta}$)}. Then, the operator
  $\mathcal{I}$ in~\eqref{eq:IGood} satisfies~\textbf{(I)}.
\end{lemma}

When the region $\Omega$ reachable by the crowd is restricted by the
walls $\partial\Omega$, we intend the convolution restricted to
$\Omega$
\begin{equation}
  \label{eq:convolution}
  (\rho * \eta) (x)
  =
  \int_\Omega \rho (y) \, \eta (y-x) \, \d{y}
\end{equation}
which coincide with the previous definition in the case $\spt
\rho\subset \Omega$.  A better choice is described
in~\eqref{eq:OmegaX}.  The vector $\nu$ is here chosen as a sum
\begin{equation}
  \label{eq:nu}
  \nu = g + \delta \,.
\end{equation}
The former vector $g$ is tangent to the \emph{``optimal''} path,
depending only on the geometry of the environment and coherent with
it. Hence we assume that
\begin{description}
\item[(g)] $g \in \C2(\reali^N;S^{N-1})$ satisfies $\nabla g \in
  \L\infty (\reali^N; \reali^{N\times N})$, $\div g \in (\W11 \cap
  \W1\infty) (\reali^N; \reali)$ and the invariance condition $g (x)
  \cdot n (x) \geq 0$ holds for all $x \in \partial\Omega$.
\end{description}
\noindent The latter vector $\delta$ describes the discomfort felt
when passing too near to walls or obstacles. Below, we choose
\begin{equation}
  \label{eq:delta}
  \delta (x)
  =
  \lambda \, \alpha (x) \, n (x) \,.
\end{equation}
Here, $\lambda \in \reali^+$ is a suitable constant and $n (x)$ is the
inward normal~\eqref{eq:n}.  The role of the function $\alpha$ is to
confine this discomfort to the region near the walls, i.e.~we require
\begin{description}
\item[($\boldsymbol{\alpha}$)] $\alpha \in \C2 (\reali^N;[0,1])$ is
  such that $\alpha (x) = 0$ whenever $B (x,r_{\Omega}) \subseteq
  \Omega$ and $\alpha (x) = 1$ for $x \in \partial\Omega$.
\end{description}
A better choice for the discomfort is discussed in
Appendix~\ref{app:G}.

The present setting~\eqref{eq:General}--\eqref{eq:IGood} can be
effectively applied also in presence of walls, as shown by the next
result.

\begin{proposition}
  \label{prop:GoodP}
  Let $\epsilon>0$, \textbf{(v)}, \textbf{(\/$\boldsymbol{\Omega}$)},
  \textbf{($\boldsymbol{\eta}$)}, \textbf{(g)}
  and~\textbf{($\boldsymbol{\alpha}$)} hold. Define $\nu$
  by~\eqref{eq:nu}, $\delta$ by~\eqref{eq:delta} and $\mathcal{I}$
  by~\eqref{eq:IGood}.  Assume moreover that either $\partial\Omega$
  is compact, or $\alpha \in (\W21 \cap \W2\infty)$ and
  $d_{\partial\Omega} \in (\W31 \cap \W3\infty)$.  Then,
  \eqref{eq:General}--\eqref{eq:IGood} satisfies the assumptions of
  Theorem~\ref{thm:existence} and
  Theorem~\ref{thm:stability}. Furthermore, if $\lambda \geq
  \epsilon\, R \, \norma{\nabla\eta}_{\L1}$, then
  property~\textbf{(P)} holds.
\end{proposition}

\subsection{The Model~\eqref{eq:General}--\eqref{eq:IPT}}
\label{sub:GoodPT}

In the framework of~\eqref{eq:General}, we now extend the
model~\cite[(2.1)--(2.4)]{CristianiPiccoliTosin} to take into account
the effects of crowd density on the speed modulus, i.e.~of $v
(\rho)$. Moreover, we interpret the nonlocal term as a weighted
average of the density, where the amount of crowd in the direction of
the optimal path $g$ is given more importance.

Using the same notation as above, we thus consider the
model~\eqref{eq:General} with
\begin{equation}
  \label{eq:IGoodPT}
  \mathcal{I} (\rho)
  =
  -
  \epsilon \,
  \frac{
    \nabla \!\!
    \int_{\reali^N} \rho(y) \, \eta(x-y)  \,
    \phi \!\left((y-x) \cdot g(x)\right)\d{y}
  }{\sqrt{1+
      \norma{
        \nabla \!\!
        \int_{\reali^N} \rho(y) \, \eta(x-y)  \,
        \phi \!\left((y-x) \cdot g(x)\right)\d{y}}^2}
  } \,.
\end{equation}
Here, $\epsilon$, $\eta$ and $v$ are as in \S~\ref{sub:Good} and, in
particular, $\nu$ is as in~\eqref{eq:nu}.

The (almost) isotropic convolution~\eqref{eq:convolution} is here
weighted by $\phi$, whose argument is essentially the angle between
the preferred path $g$ and $y-x$, the point $x$ being the position of
the individual an the point $y$ being the location where the density
is evaluated. This term takes into account the preference of each
individual to deviate little from the preferred path defined by
$g$. For example, let $\phi\in \C\infty(\reali, [0,1])$ be such that
$\phi \equiv 0$ on $\left]-\infty, 0\right]$ and $\phi \equiv 1$ on
$\left[\theta, +\infty\right[$, where $\theta>0$ is a given
parameter. Then, adding the function $\phi\left((y-x)\cdot
  g(x)\right)$ into the nonlocal term means that the individual at $x$
reacts to the average density evaluated in the prefer ed direction
$g(x)$.

The denominator is a regularized normalization. Its presence is
necessary from the modeling point of view, for coherence with the
presence of $v (\rho)$. From the analytical point of view, this
normalization makes various expressions slightly more complicate, but
all estimates remain doable.

We are thus lead to the equation
\begin{equation}
  \label{eq:GoodPT}
  \partial_t \rho
  +
  \div \left(
    \rho\, v (\rho)
    \left(
      \nu (x)
      -
      \epsilon \,
      \frac{
        \nabla \!\!
        \int_{\reali^N} \rho(y) \, \eta(x-y)  \,
        \phi \!\left((y-x) \cdot g(x)\right)\d{y}
      }{\sqrt{1+
          \norma{
            \nabla \!\!
            \int_{\reali^N} \rho(y) \, \eta(x-y)  \,
            \phi \!\left((y-x) \cdot g(x)\right)\d{y}}^2}
      }
    \right)
  \right)
  =
  0
\end{equation}
and we verify that it fits in the analytical framework provided in
Section~\ref{sec:AR}.

\begin{lemma}
  \label{lem:PT}
  Fix $\epsilon>0$, $\eta \in \Cc3 (\reali^N;\reali^+)$ with
  $\int_{\reali^N} \eta (x) \d{z} =1$, $g \in \W3\infty(\reali;
  [0,1])$ and $\nu \in \W3\infty(\reali^N,\reali)$. Then, the operator
  $\mathcal{I}$ defined by~\eqref{eq:IPT} satisfies~\textbf{(I)}.
\end{lemma}

Passing now to the case in which the crowd's movement is constrained
by the walls $\partial\Omega$,as in the preceding section, we intend
all integrals in~\eqref{eq:IGoodPT}--\eqref{eq:GoodPT} restricted to
$\Omega$. In particular, we consider now
\begin{equation}
  \label{eq:IGoodPTOmega}
  \mathcal{I} (\rho)
  =
  -
  \epsilon \,
  \frac{
    \nabla \!\!
    \int_{\Omega} \rho(y) \, \eta(x-y)  \,
    \phi \!\left((y-x) \cdot g(x)\right)\d{y}
  }{\sqrt{1+
      \norma{
        \nabla \!\!
        \int_{\Omega} \rho(y) \, \eta(x-y)  \,
        \phi \!\left((y-x) \cdot g(x)\right)\d{y}}^2}
  } \,,
\end{equation}
with $v$ as in \textbf{(v)} and $\nu$ as in (\ref{eq:nu}).  The
applicability of theorems~\ref{thm:existence} and~\ref{thm:stability}
to~\eqref{eq:General}--\eqref{eq:IGoodPTOmega} and the validity of
property~\textbf{(P)} is ensured by the following proposition.

\begin{proposition}
  \label{prop:GoodPT}
  Let $\epsilon>0$, \textbf{(v)}, \textbf{(\/$\boldsymbol{\Omega}$)},
  \textbf{($\boldsymbol{\eta}$)} and~\textbf{(g)} hold, with moreover
  $g \in (\C3 \cap \W3\infty) (\reali^N;S^{N-1})$. Let $\phi \in (\C3
  \cap \W3\infty) (\reali;\reali)$. Define $\nu$ by~\eqref{eq:nu},
  $\delta$ by~\eqref{eq:delta} and $\mathcal{I}$
  by~\eqref{eq:IGoodPT}.  Then, \eqref{eq:General}--\eqref{eq:IGoodPT}
  satisfies the assumptions of Theorem~\ref{thm:existence} and
  Theorem~\ref{thm:stability}.

  Moreover, assume that $\phi' \geq 0$ and call $\ell = \mbox{{\rm
      diam}} (\spt \eta)$. Then,
  \begin{displaymath}
    \phi'
    \geq
    0
    \qquad \mbox{and} \qquad
    \lambda
    \geq
    R \, \norma{\eta}_{\W11} \,
    \norma{\phi}_{\W1\infty} \,
    \left( 1 + \ell \, \norma{\nabla g}_{\L\infty}\right)
  \end{displaymath}
  imply that property~\textbf{(P)} holds.
\end{proposition}

\subsection{The Model~\cite[(2.1)--(2.4)]{CristianiPiccoliTosin}}
\label{sub:PT}

The model in~\cite{CristianiPiccoliTosin, PiccoliTosin}, although set
therein in the space $\mathcal{M} (\reali^N;\reali^+)$ of positive
Radon measures, can be seen as a particular case of~\eqref{eq:General}
setting
\begin{equation}
  \label{eq:PT}
  \begin{array}{rcl}
    v (\rho) & = & 1
    \\
    \nu (x) & = & v_{\mbox{\tiny des}} (x)
  \end{array}
  \quad \mbox{ and } \quad
  \mathcal{I} (\rho)
  =
  \epsilon
  \int_{\reali^N}\rho(y) \,
  \nabla\eta(x-y) \; \phi\!\left((y-x) \cdot \nu(x)\right)\d{y} \,,
\end{equation}
$v_{\mbox{\tiny des}}$ being the \emph{desired} speed,
see~\cite[formula~(2.4)]{CristianiPiccoliTosin}. When $\eta$ is
radially symmetric, we
recover~\cite[formula~(2.6)]{CristianiPiccoliTosin} with $\eta (x) =
\tilde\eta (\norma{x})$ and $\tilde\eta'=f$. We leave
to~\cite{CristianiPiccoliTosin} the motivations of this model.

\begin{proposition}
  \label{prop:PT}
  Let $\epsilon>0$, \textbf{(v)}, \textbf{($\boldsymbol{\nu}$)} and
  \textbf{($\boldsymbol{\eta}$)} hold. Assume that $\phi \in (\C2\cap
  \W2\infty) (\reali;\reali)$.  Then,
  \eqref{eq:General}--\eqref{eq:PT} satisfies the assumptions of
  Theorem~\ref{thm:existence} and Theorem~\ref{thm:stability}.
\end{proposition}

In particular, above we prove that if the initial datum $\rho_0$ is in
$\L1 (\reali^N;[0,R])$, then the corresponding solution satisfies the
same bounds.  This ensures that neither focusing to any Dirac delta
takes place, nor values of the density above $R$ can be expected. (For
the sake of completeness, we note that the conditions $\eta \geq 0$
and $\int_{\reali^N} \eta\d{x}=1$ is in the case
of~\eqref{eq:General}--\eqref{eq:PT} neither necessary, nor meaningful
and can be replaced by $\eta \in (\Cc3 \cap \W2\infty \cap \W31)
(\reali^N;\reali)$, see the proof in Section~\ref{sec:TD} for more
details.).

\section{Qualitative Properties}
\label{sec:NR}

This section is devoted to sample numerical integrations
of~\eqref{eq:General}. In all the examples below, we choose as vector
field $\nu = \nu (x)$ the geodesic one, computed solving numerically
the eikonal equation. This leads to the formation of congested queues
near to the door jambs. From the modeling point of view, a more
refined choice would consist in choosing $\nu$ so that most
pedestrians are directed towards the central part of the exit. This
choice increases the difficulties neither of the analytical treatment
nor of the numerical integration but imposes the introduction of
several further parameters.

The algorithm used is the classical Lax-Friedrichs method with
dimensional splitting. As usual, a fixed grid $(x_i,y_j)$ for $i_0,
\ldots, n_x$ and $j=1, \ldots, n_y$ is introduced and the density
$\rho$ is approximated through the values $\rho_{ij}$ on this grid. At
every time step, the convolution in vector $\mathcal{I} (\rho)$ is
then computed through products of the type $A_{ih} \rho_{hk} B_{kj}$,
for two fixed matrices $A$ and $B$ depending only on the kernel
$\eta$.

All the examples below are set in $\reali^2$, due to obvious
visualization problems in higher dimensions. As is well known, the
analytical techniques are essentially independent from the dimension
as also the numerical algorithm. The time of integration obviously
increases with the dimension.

\subsection{Lane Formation}
\label{subs:LaneFormation}

A widely detected pattern formed in the context of crowd dynamics is
that of \emph{lane formation}, see for instance~\cite{HelbingEtAl2002,
  HelbingJohansson2007, HoogendoornBovy2003, PiccoliTosin2009}. This
feature has been often related to the specific qualities of each
individual, i.e.~it has usually been explained from a microscopic
point of view. Here, in a purely macroscopic setting, we show that the
solutions to~\eqref{eq:General}--\eqref{eq:IGood} also display this
pattern formation phenomenon, with pedestrian self organizing along
lanes.

\begin{figure}[htpb]
  \centering
  \includegraphics[width=0.32\textwidth
]{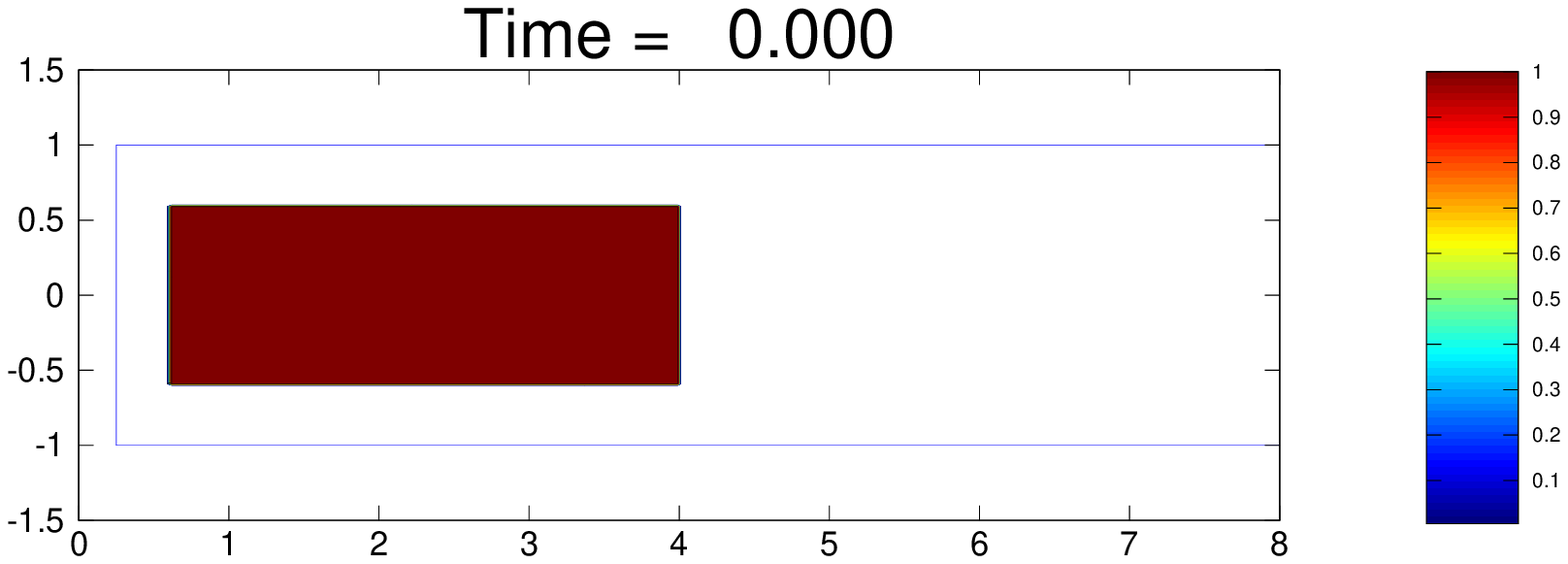}\hfil%
  \includegraphics[width=0.32\textwidth
]{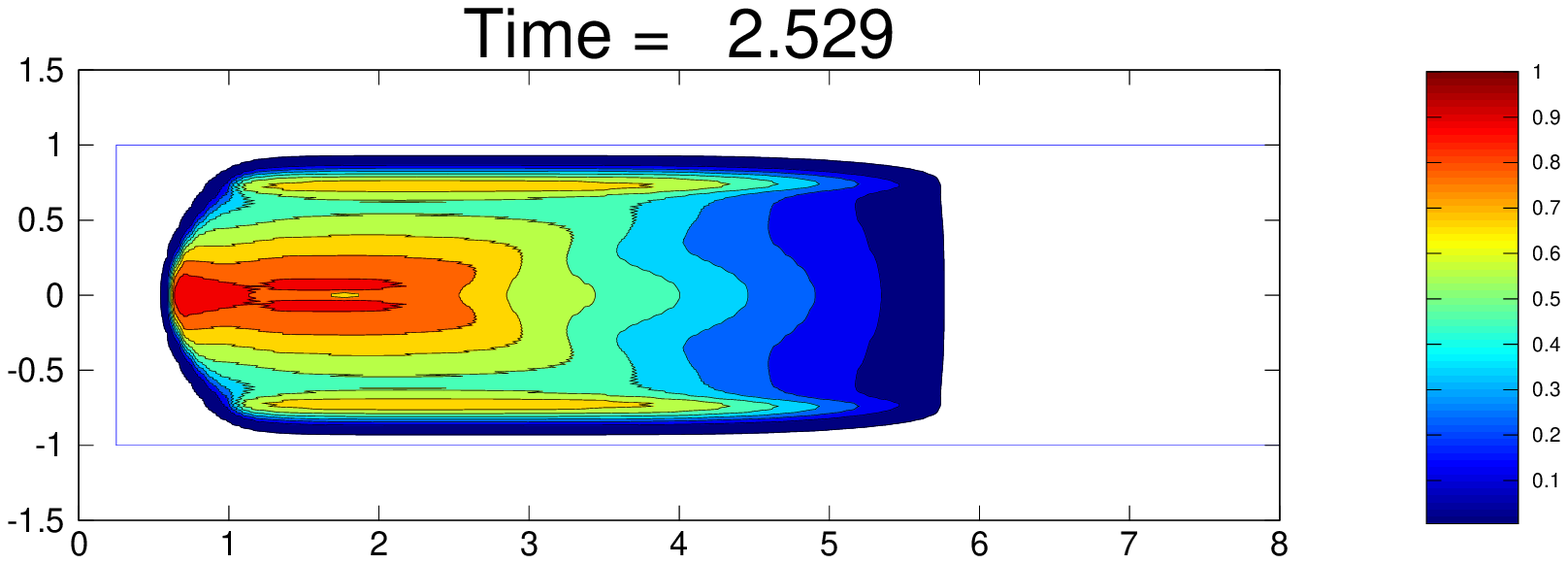}\hfil%
  \includegraphics[width=0.32\textwidth
]{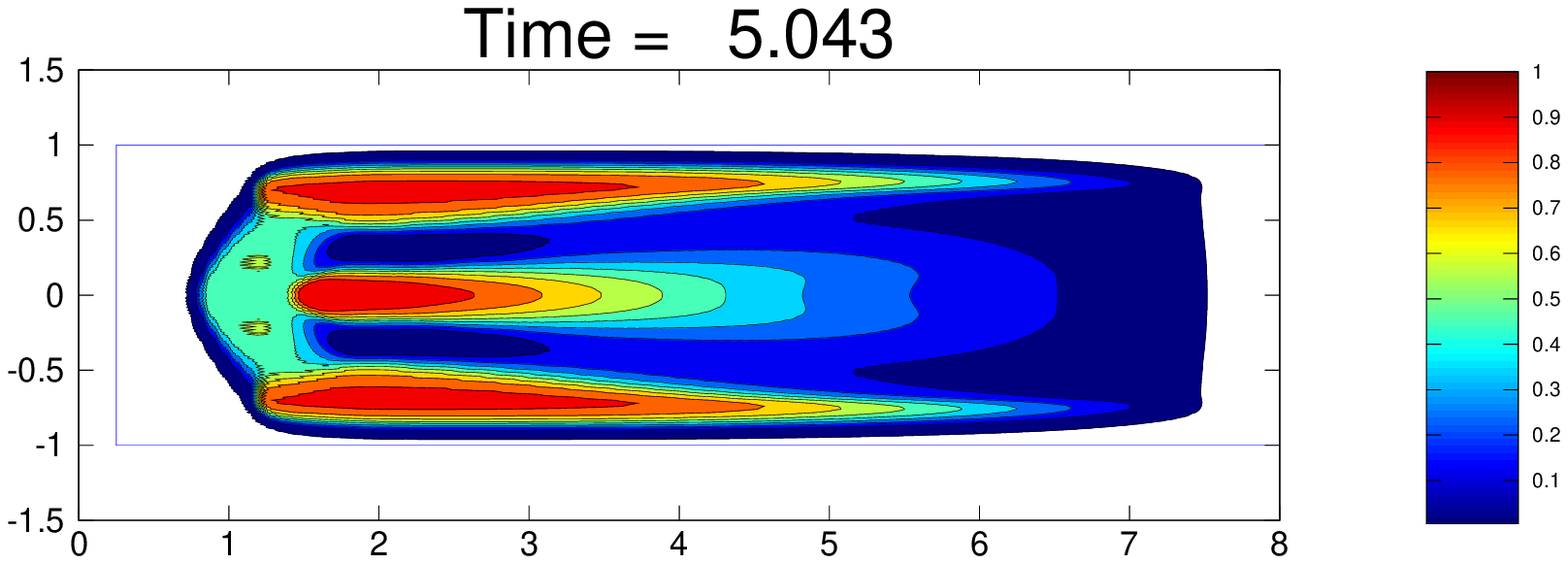}\\
  \includegraphics[width=0.32\textwidth
]{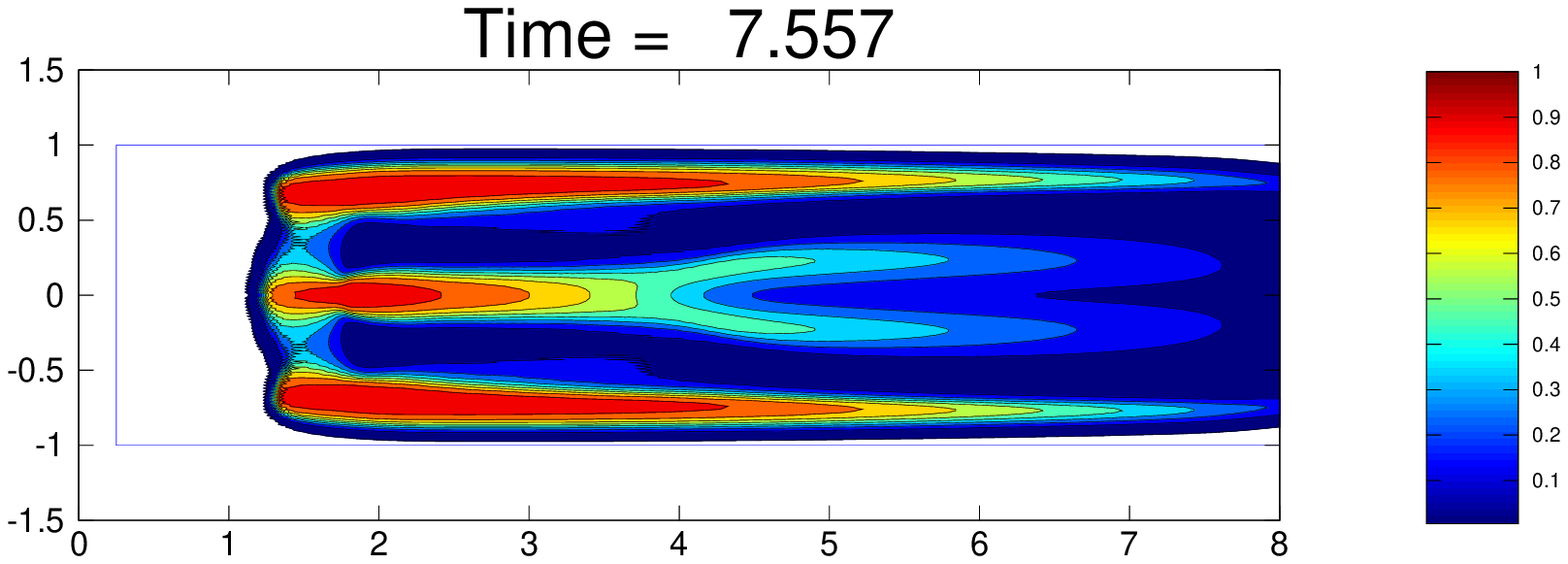}\hfil%
  \includegraphics[width=0.32\textwidth
]{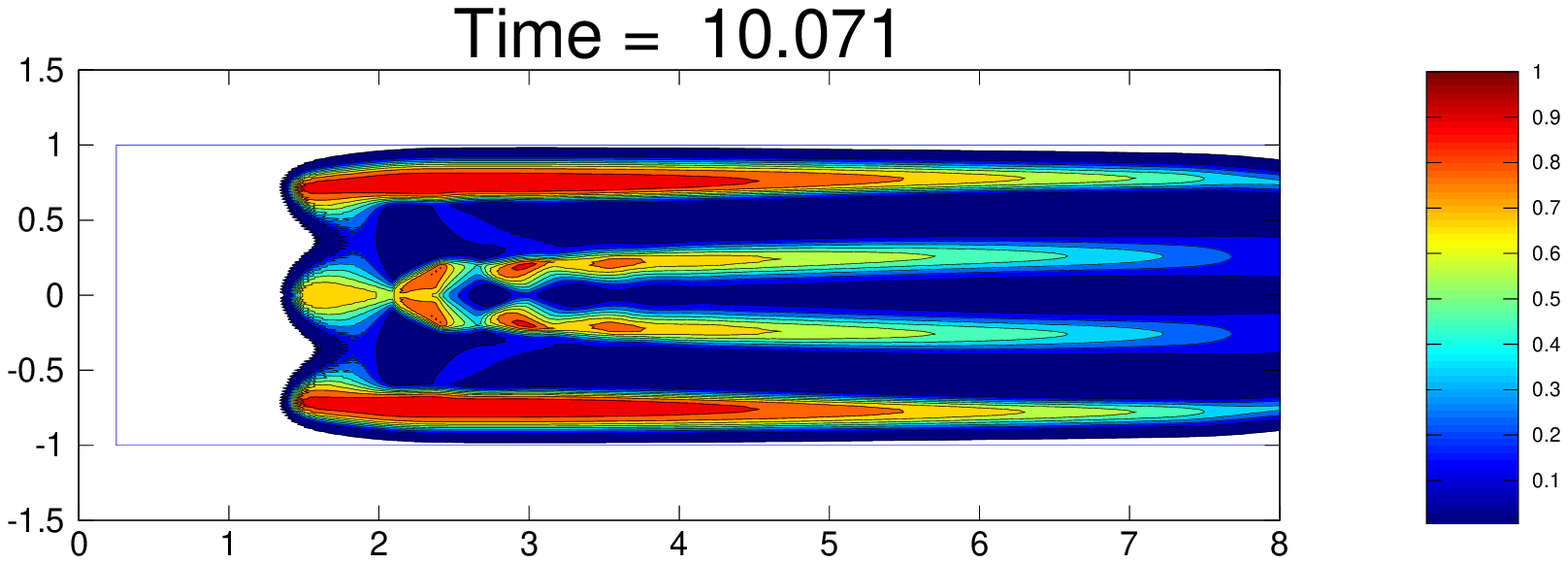}\hfil%
  \includegraphics[width=0.32\textwidth
]{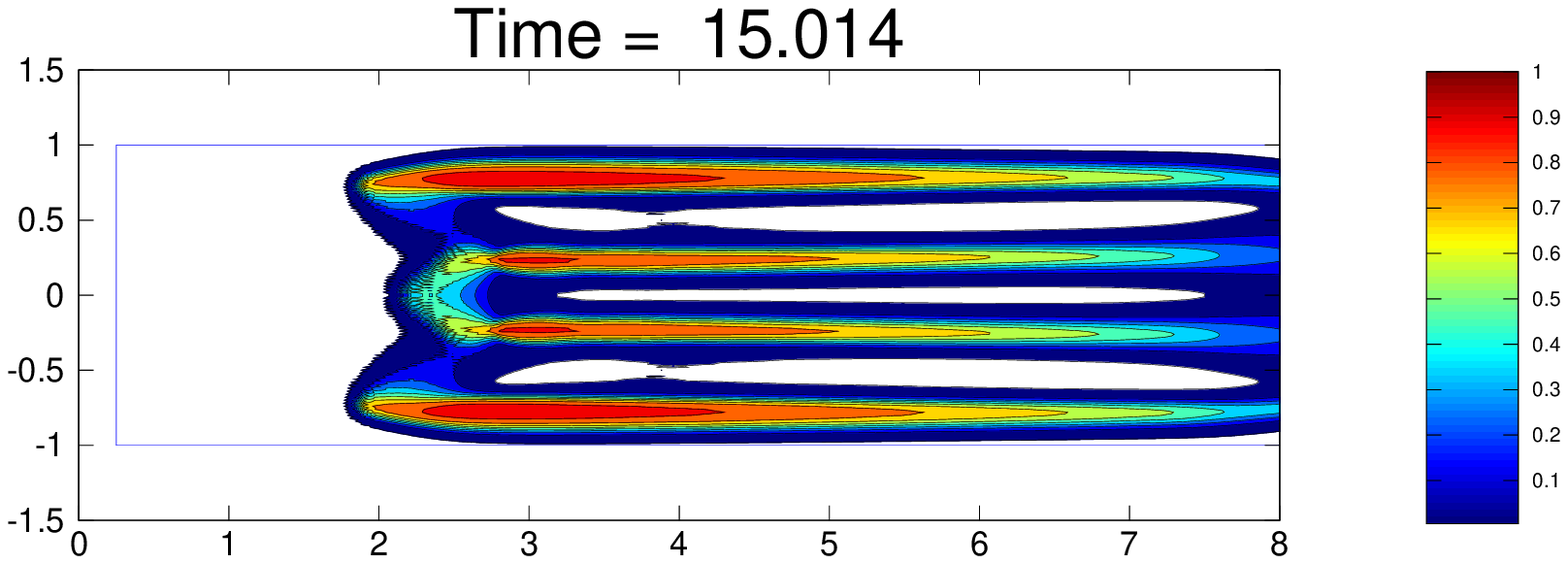}%
  \caption{Solution
    to~\eqref{eq:General}--\eqref{eq:IGood}--\eqref{eq:nu}--\eqref{eq:lanes}
    at times $t=0$, $2.529$, $5.043$, $7.557$, $10.071,\,
    15.014$. First 3 lanes are formed, then the middle lane bifurcates
    forming the fourth lane.}
  \label{fig:lanes}
\end{figure}

Consider~\eqref{eq:General}--\eqref{eq:IGood} with
\begin{equation}
  \label{eq:lanes}
  \begin{array}{@{}rcl@{\quad}rcl@{}@{\quad}rcl@{}}
    \nu (x)
    & = &
    \left[
      \begin{array}{@{\,}c@{\,}}
        1\\0
      \end{array}
    \right]
    + \delta (x)\,,
    &
    \eta (x)
    & = &
    \left[1-\left(\frac{x_1}{r}\right)^2\right]^3
    \left[1-\left(\frac{x_2}{r}\right)^2\right]^3 \,
    \caratt{[-r,r]^2} (x)\,,
    &
    r & = & \frac45\,,
    \\[15pt]
    v (\rho)
    & = &
    \frac12 \, (1-\rho)\,,
    &
    \rho_0 (x)
    & = &
    \caratt{[3/5, 4] \times [-3/5, 3/5]} (x)\,,
    &
    \epsilon
    & = &
    \frac25\,,
  \end{array}
\end{equation}
where $\delta = \delta(x)$ describes the discomfort due to walls: it
is a vector normal to the walls, pointing inward, with intensity $3/2$
along the walls, decreasing linearly to $0$ at a distance $3/10$ from
the walls.  As Figure~\ref{fig:lanes} shows, the initially uniform
crowd distribution evolves into a patterned configuration, first with
4 lanes and then with 5 lanes. The number of lanes depends on the size
of the support of the convolution kernel $\eta$. Indeed, keeping all
functions and parameters fixed, but not the parameter $r$, which
determines the size of $\spt \eta$, we obtain patterns differing in
the number of lanes, see Figure~\ref{fig:DifferentLanes}.
\begin{figure}[htpb]
  \centering
  \includegraphics[width=0.32\textwidth
]{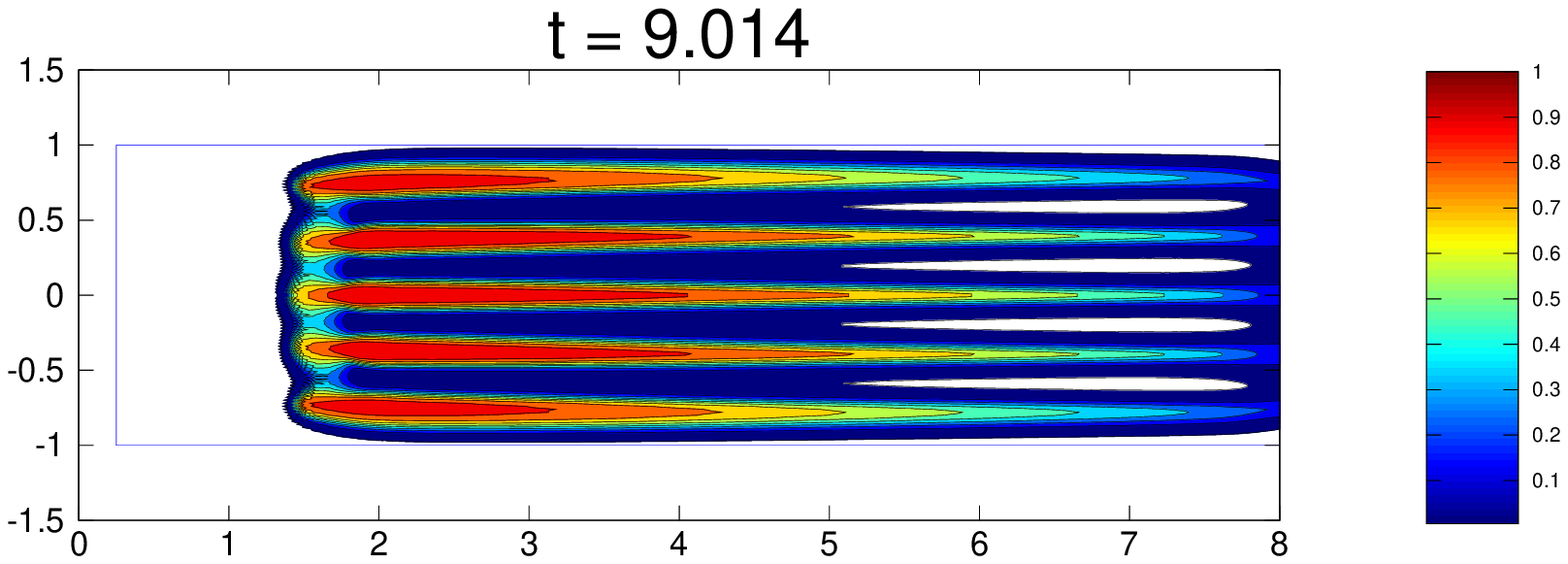}%
  \includegraphics[width=0.32\textwidth
]{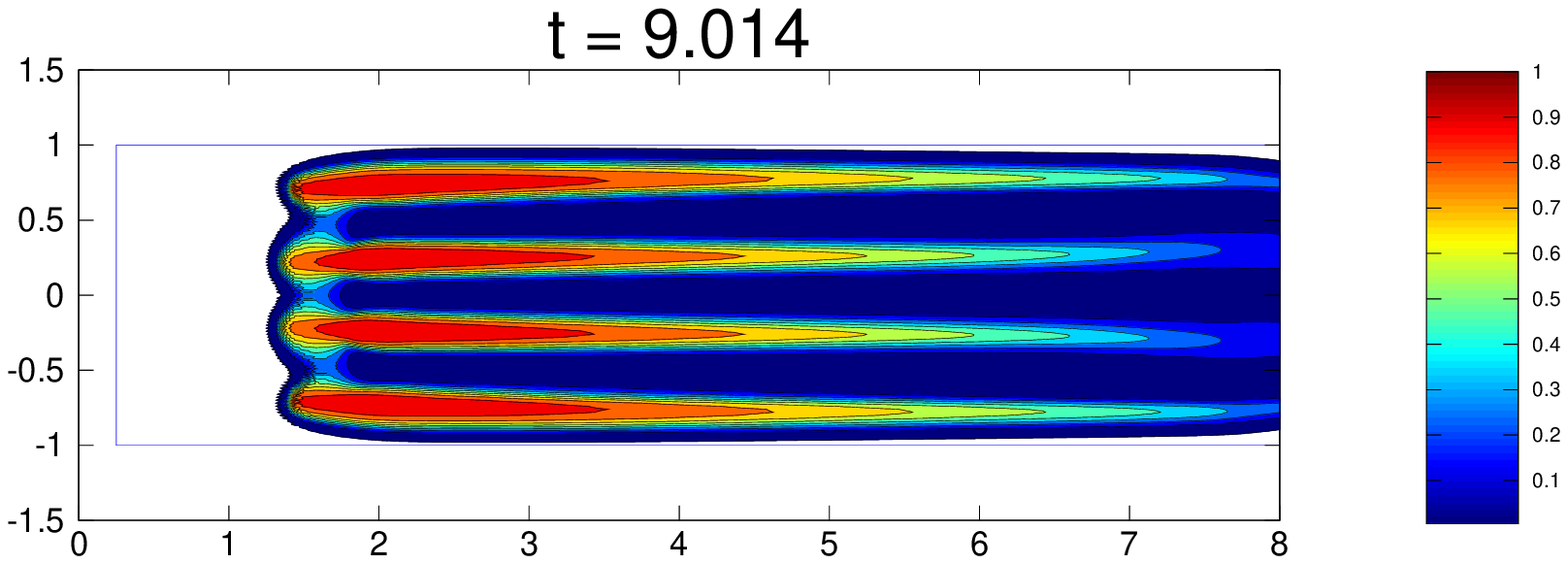}%
  \includegraphics[width=0.32\textwidth
]{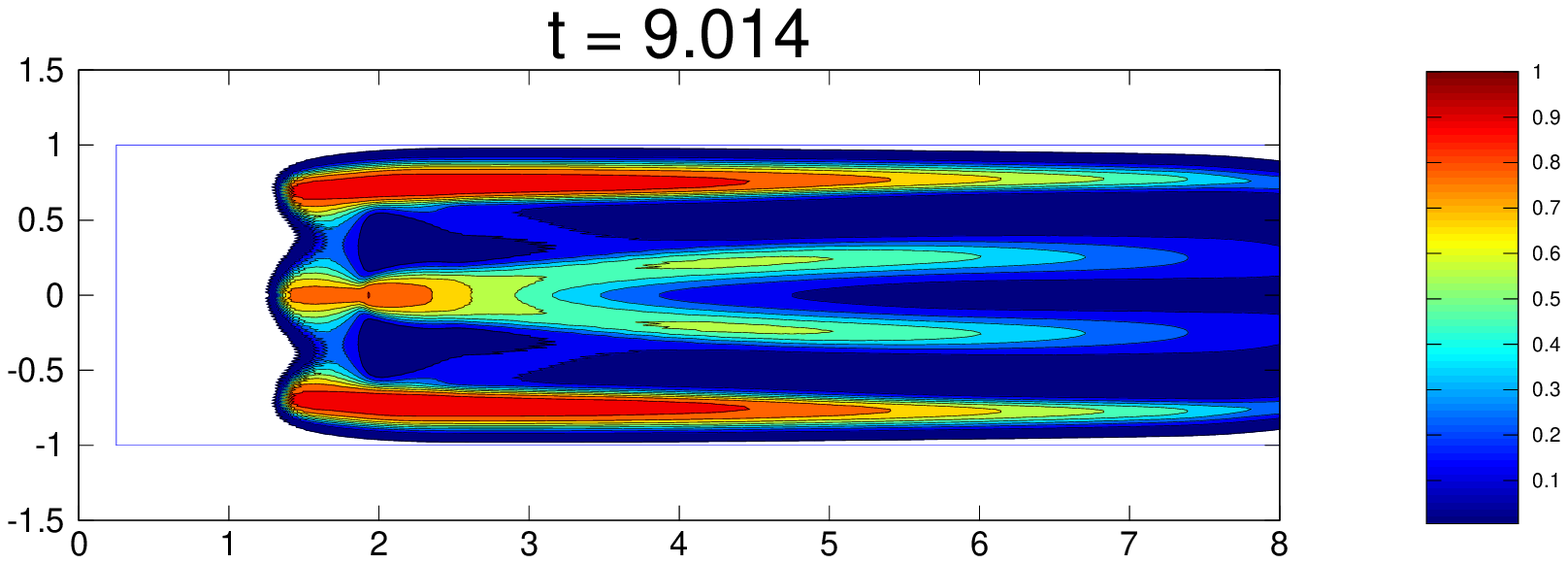}\\
  \includegraphics[width=0.32\textwidth
]{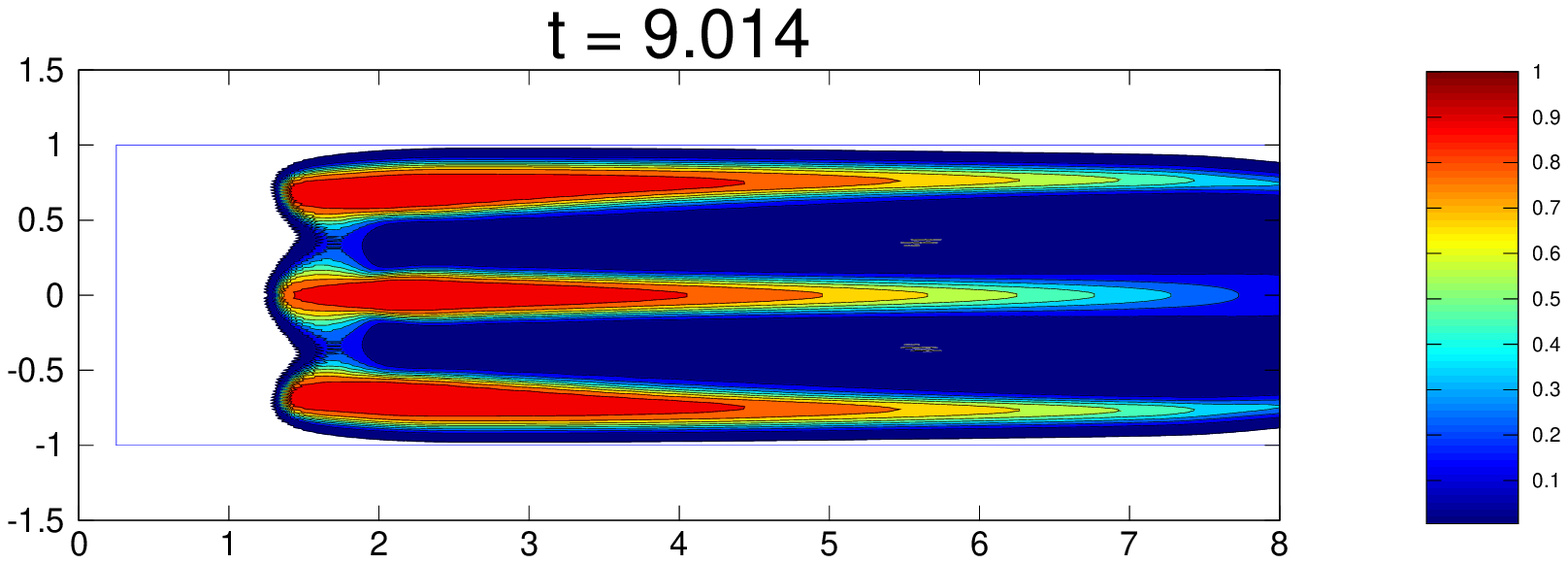}%
  \includegraphics[width=0.32\textwidth
]{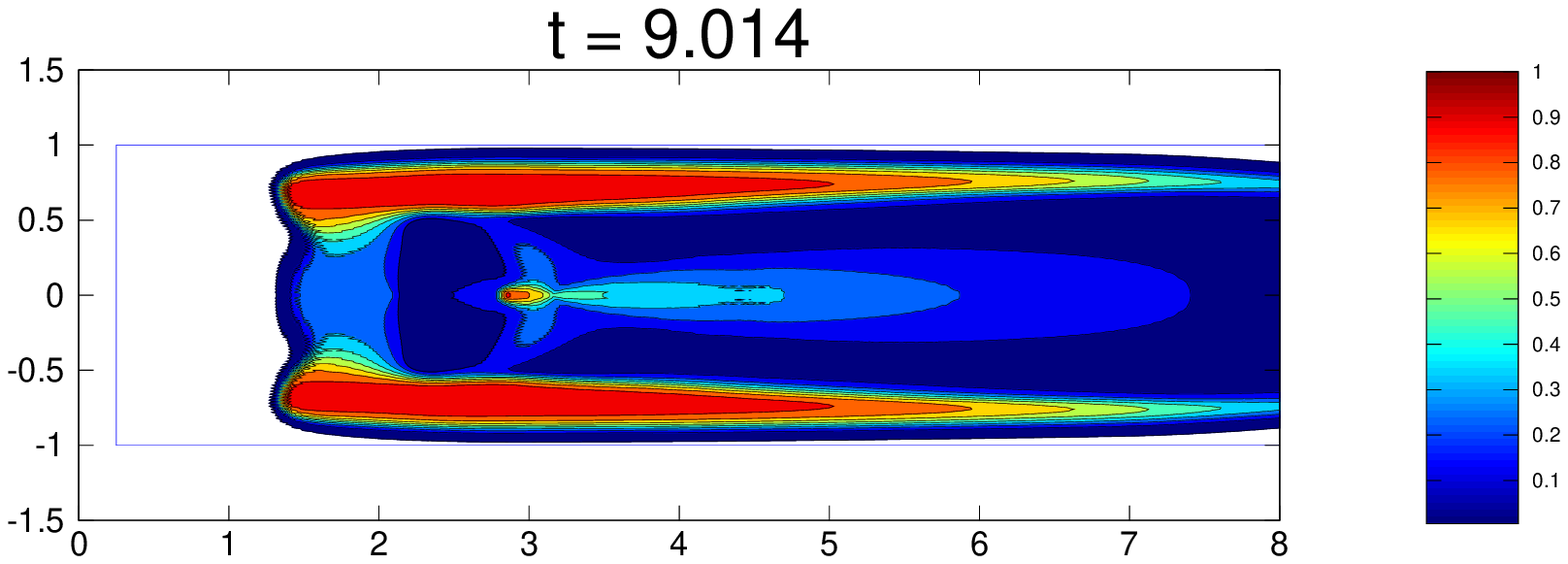}%
  \includegraphics[width=0.32\textwidth
]{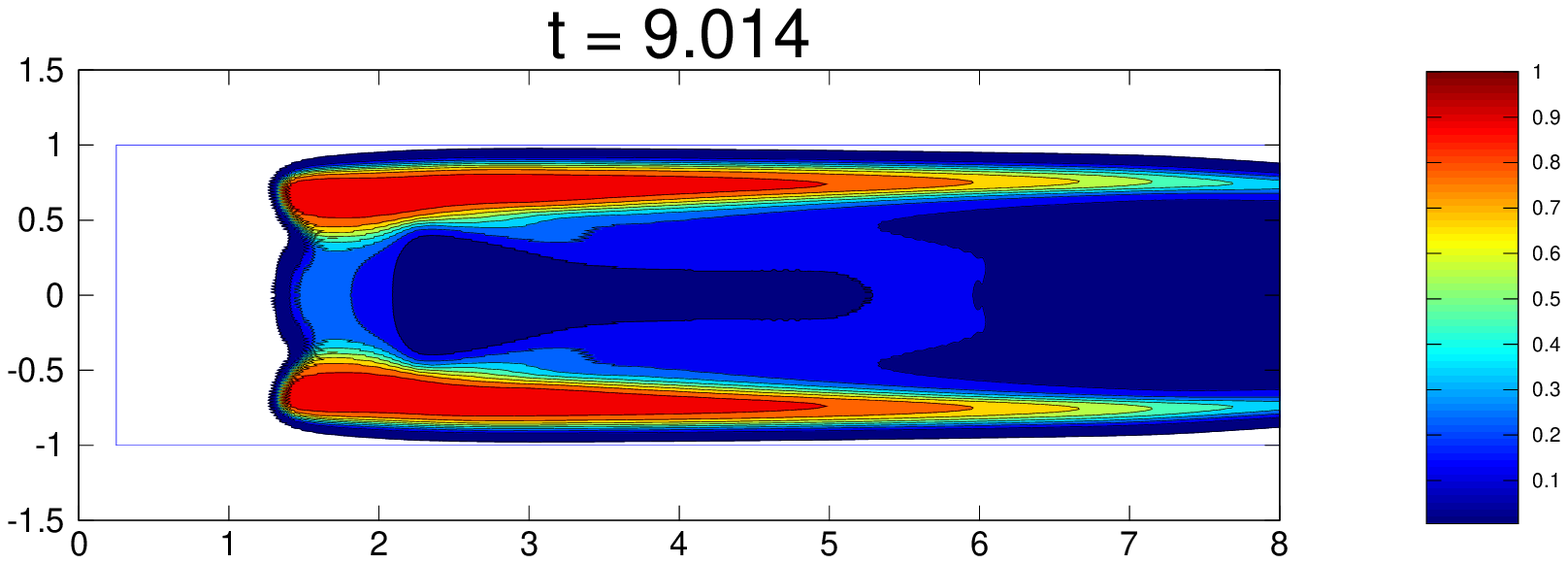}
  \caption{Solution
    to~\eqref{eq:General}--\eqref{eq:IGood}--\eqref{eq:nu}--\eqref{eq:lanes}
    computed at time $t=9.014$ and with $\spt\eta$ with radius
    $r=0.5,\, 0.6,\, 0.8,\, 0.9,\, 1.0,\, 
    1.4$. Note that as $r$ increases, the number of lanes diminishes.}
  \label{fig:DifferentLanes}
\end{figure}

The formation of lanes is a rather stable phenomenon. Indeed,
Figure~\ref{fig:stab} shows the result of the integration
of~\eqref{eq:General}--\eqref{eq:IGood}--\eqref{eq:nu}--\eqref{eq:lanes}
computed at time $t=0,\,5,\,10$ with $r=3/5$ (above) and $r=9/10$
(below) with initial data different from that in~\eqref{eq:lanes}.
\begin{figure}[htpb]
  \centering
  \includegraphics[width=0.32\textwidth
]{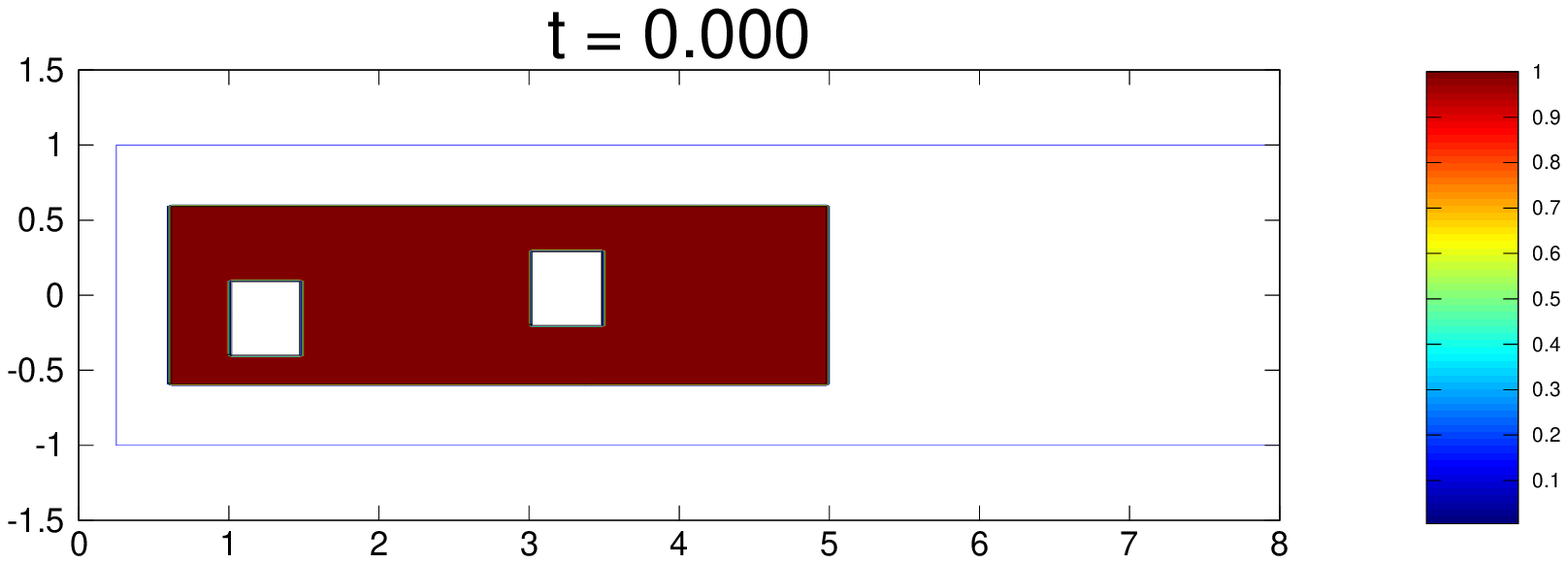}%
  \includegraphics[width=0.32\textwidth
]{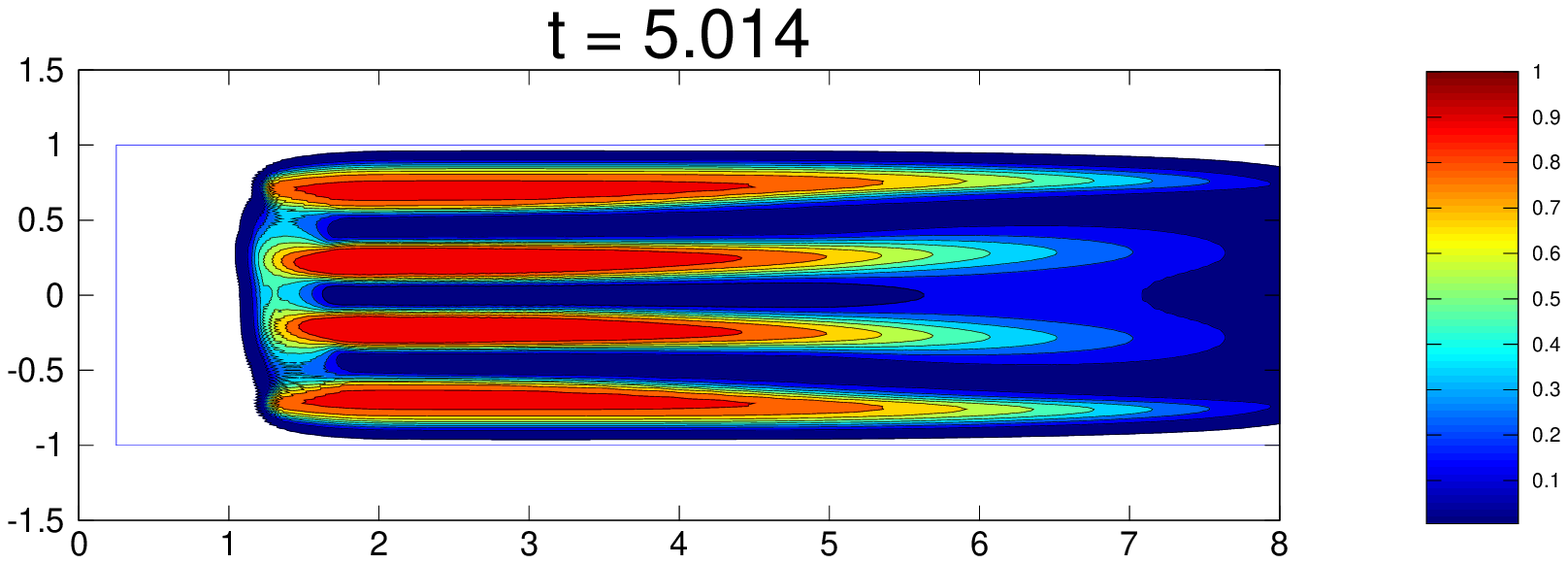}%
  \includegraphics[width=0.32\textwidth
]{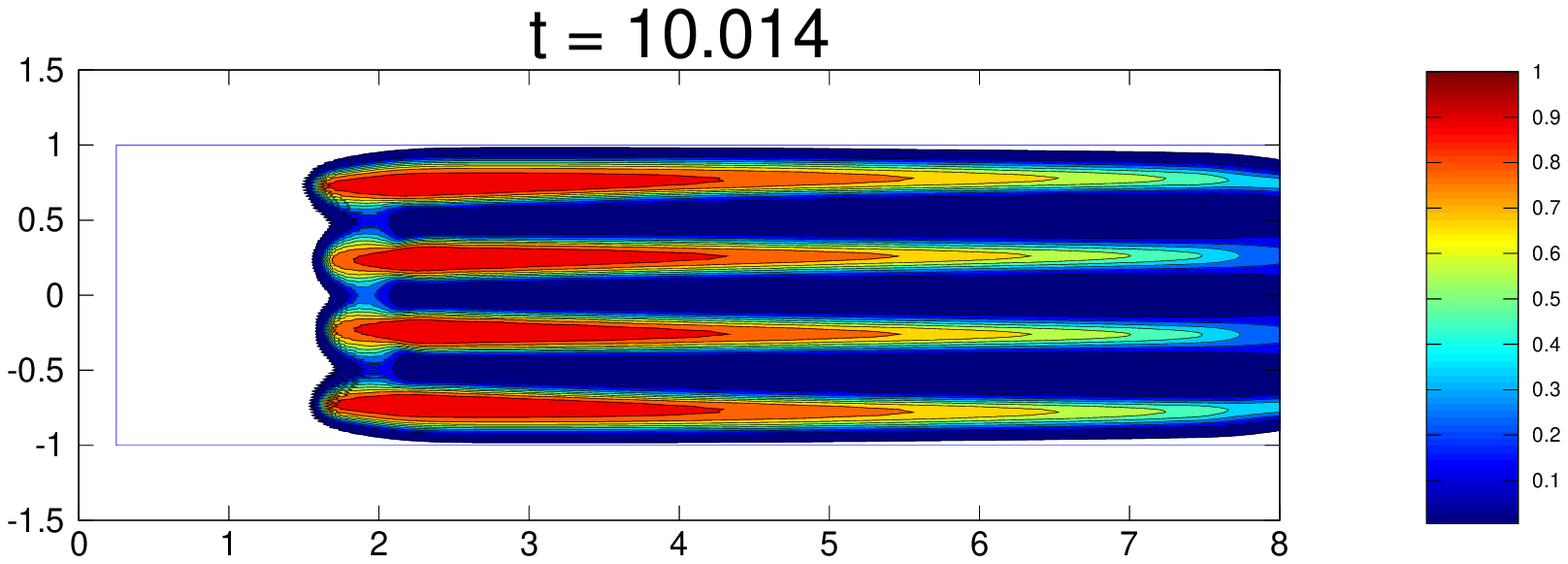}\\
  \includegraphics[width=0.32\textwidth
]{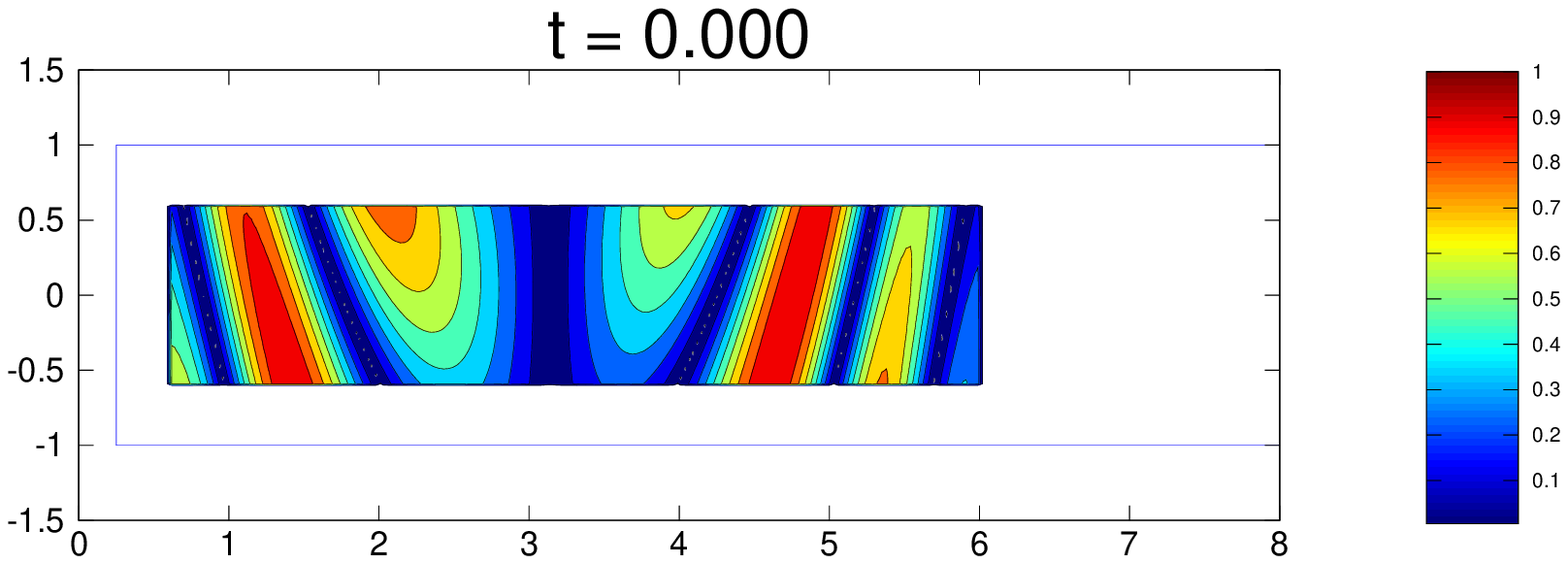}%
  \includegraphics[width=0.32\textwidth
]{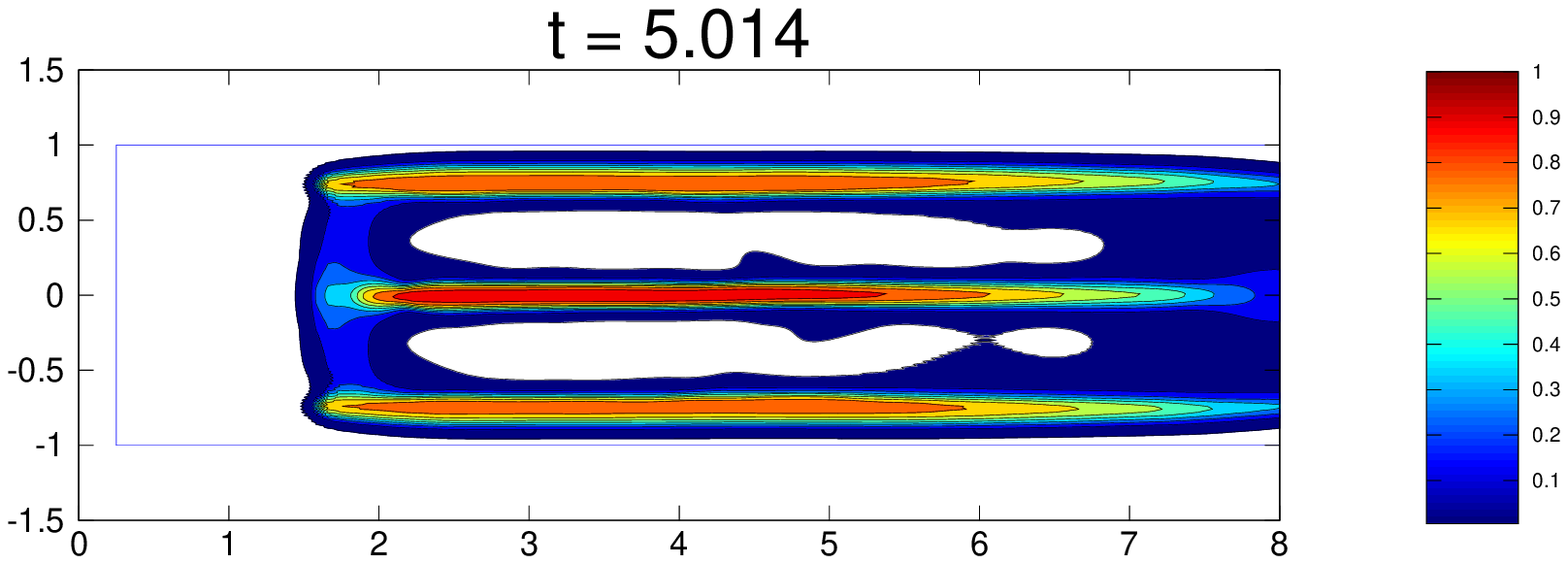}%
  \includegraphics[width=0.32\textwidth
]{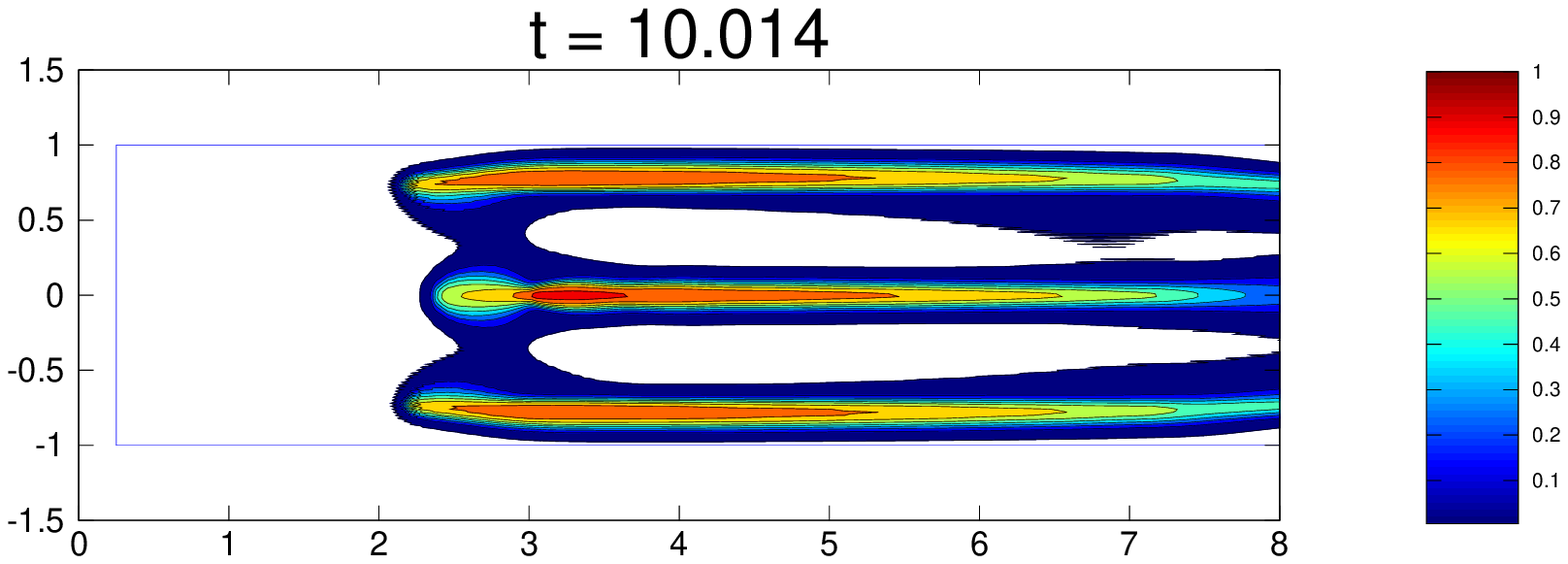}
  \caption{Solution
    to~\eqref{eq:General}--\eqref{eq:IGood}--\eqref{eq:nu}--\eqref{eq:lanes}
    with different initial data at time $t = 0,\, 5.014,\, 10.014$ and
    above with $r = 0.6$, below with $r = 0.9$. Note that above $4$
    lanes form and below $5$, similarly to what obtained in
    Figure~\ref{fig:lanes}.}
  \label{fig:stab}
\end{figure}
In both cases, lanes are formed similar to the corresponding
situations in Figure~\ref{fig:DifferentLanes}.

We also note that in the present framework, using the terms
in~\cite[Section~5.4]{PiccoliTosin2009}, lanes form also in an
\emph{isotropic} setting. Indeed, the integrations in
figures~\ref{fig:lanes}--\ref{fig:DifferentLanes} were obtained with
individuals able to see both forward \emph{and} behind.

\subsection{Evacuation of a Room}
\label{subs:Evacuation}

A standard application of macroscopic models for crowd dynamics is the
minimization of evacuation times. The present setting applies to
general geometries, see the
assumption~\textbf{($\boldsymbol{\Omega}$)}. Here we show
that~\eqref{eq:General} captures reasonable features of the escape
dynamics.

We consider a room with an exit, as in Figure~\ref{fig:Initial}. The
vector $\nu = \nu (x)$ is chosen as the unit vector tangent at $x$ to
the geodesic connecting $x$ to the exit. The discomfort $d = d (x)$ is
a vector normal to the walls, pointing inward, with intensity $1$
along the walls, decreasing linearly to $0$ at a distance $1/2$ from
the walls. The other quantities are in~\eqref{eq:Evacuation}.
\begin{figure}[htpb]
  \begin{minipage}[c]{0.37\linewidth}
    \includegraphics[width=\textwidth
]{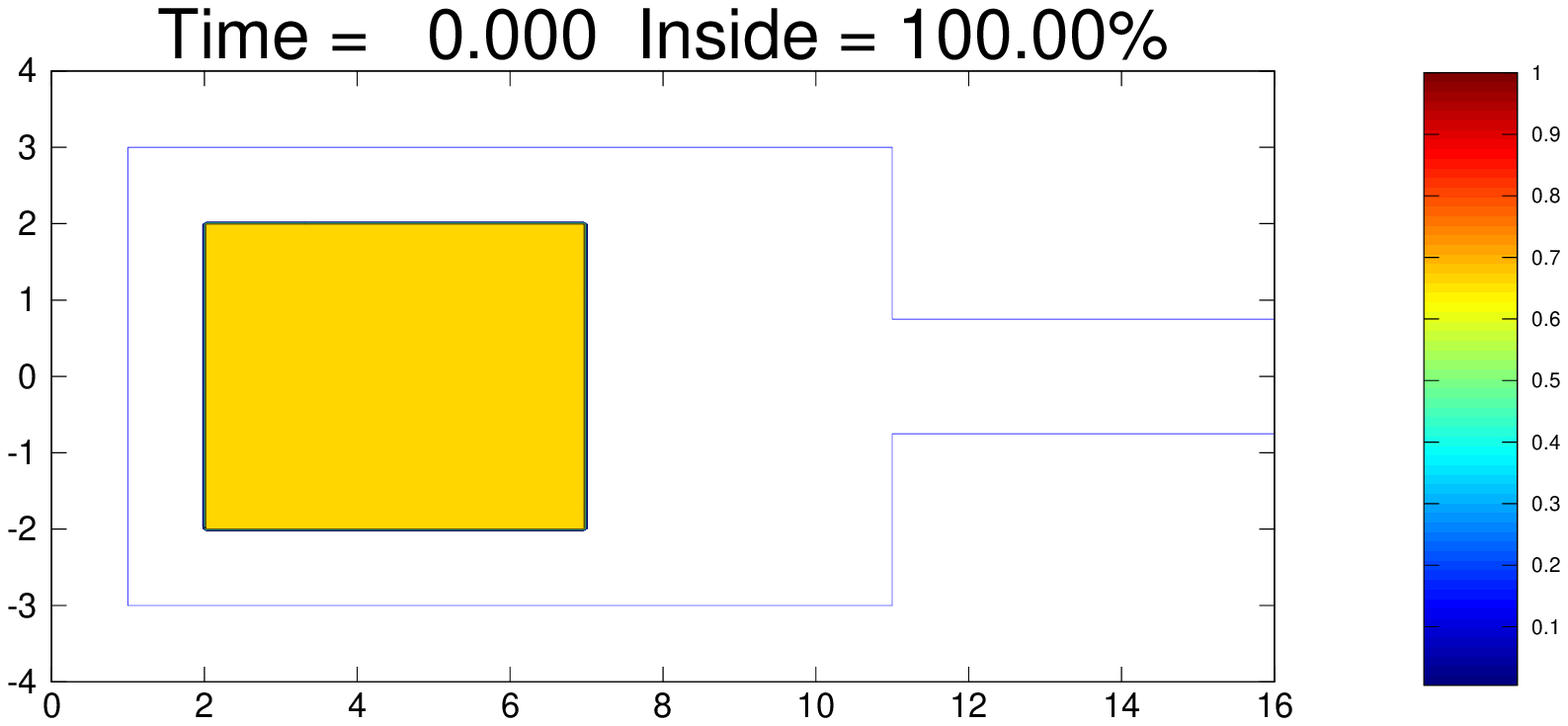}
    \caption{Initial datum and room geometry considered in
      \S~\ref{subs:Evacuation}.}
    \label{fig:Initial}
  \end{minipage}%
  \begin{minipage}[c]{0.63\linewidth}
    \begin{equation}
      \label{eq:Evacuation}
      \begin{array}{@{}rcl@{}}
        v (\rho) & = & 6 (1-\rho)
        \\
        \eta (x)
        & = &
        \left[1-\left(\frac{x_1}{r}\right)^2\right]^3
        \left[1-\left(\frac{x_2}{r}\right)^2\right]^3 \,
        \caratt{[-r,r]^2} (x)\,,
        \\
        r & = & 0.6
        \\
        \rho_0 (x) & = & 0.75 \, \caratt{[2,7] \times [-2,2]} (x)
        \\
        \epsilon & = & 0.4
        \\
        \\
      \end{array}
    \end{equation}
  \end{minipage}
\end{figure}
Keeping the above parameters fixed, as well as the outer walls of the room, we insert various obstacles (columns) to direct the movement of the crowd.

First, Figure~\ref{fig:Evacuation1}, first line, shows an integration
of the case with two columns that direct people towards the exit.
\begin{figure}[htpb]
  \centering
  \includegraphics[width=0.32\textwidth
]{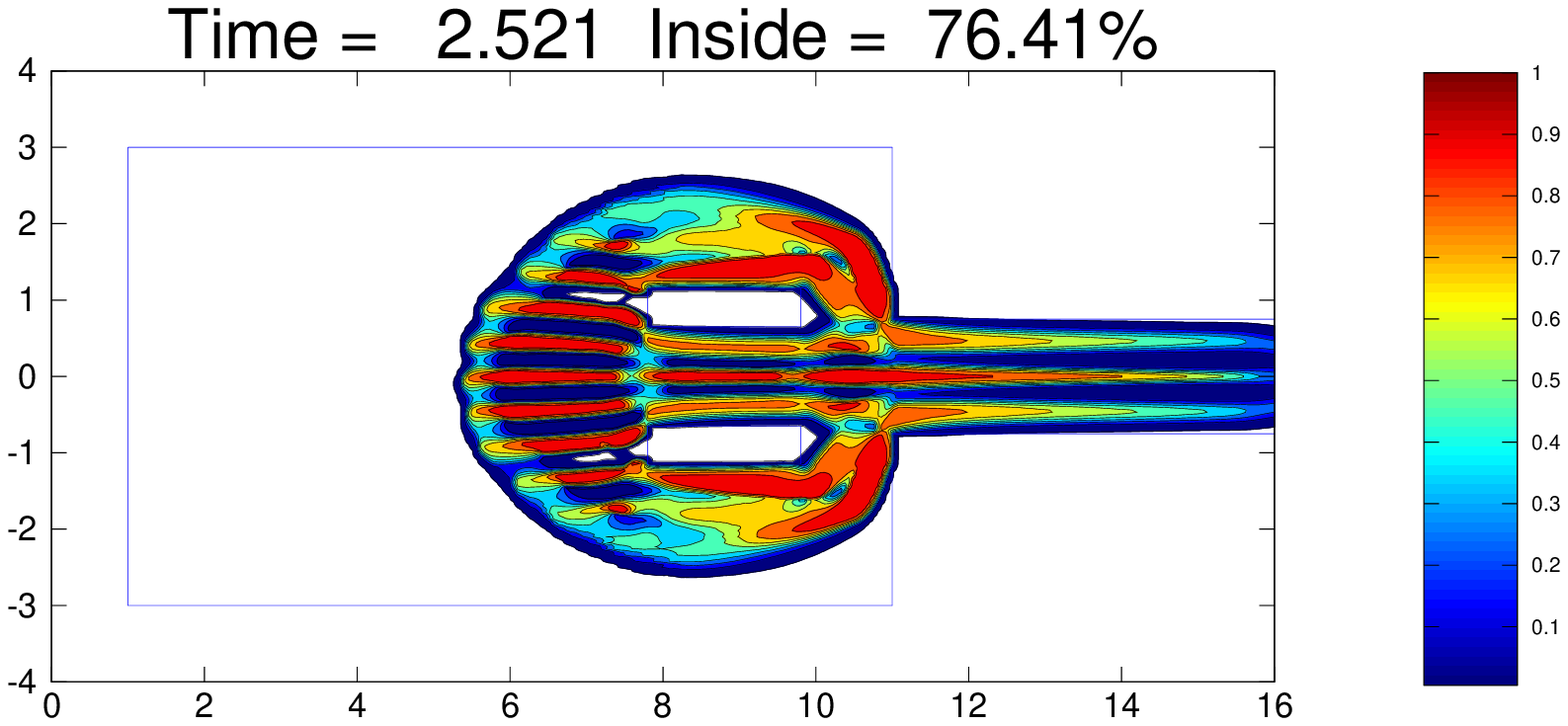}%
  \includegraphics[width=0.32\textwidth
]{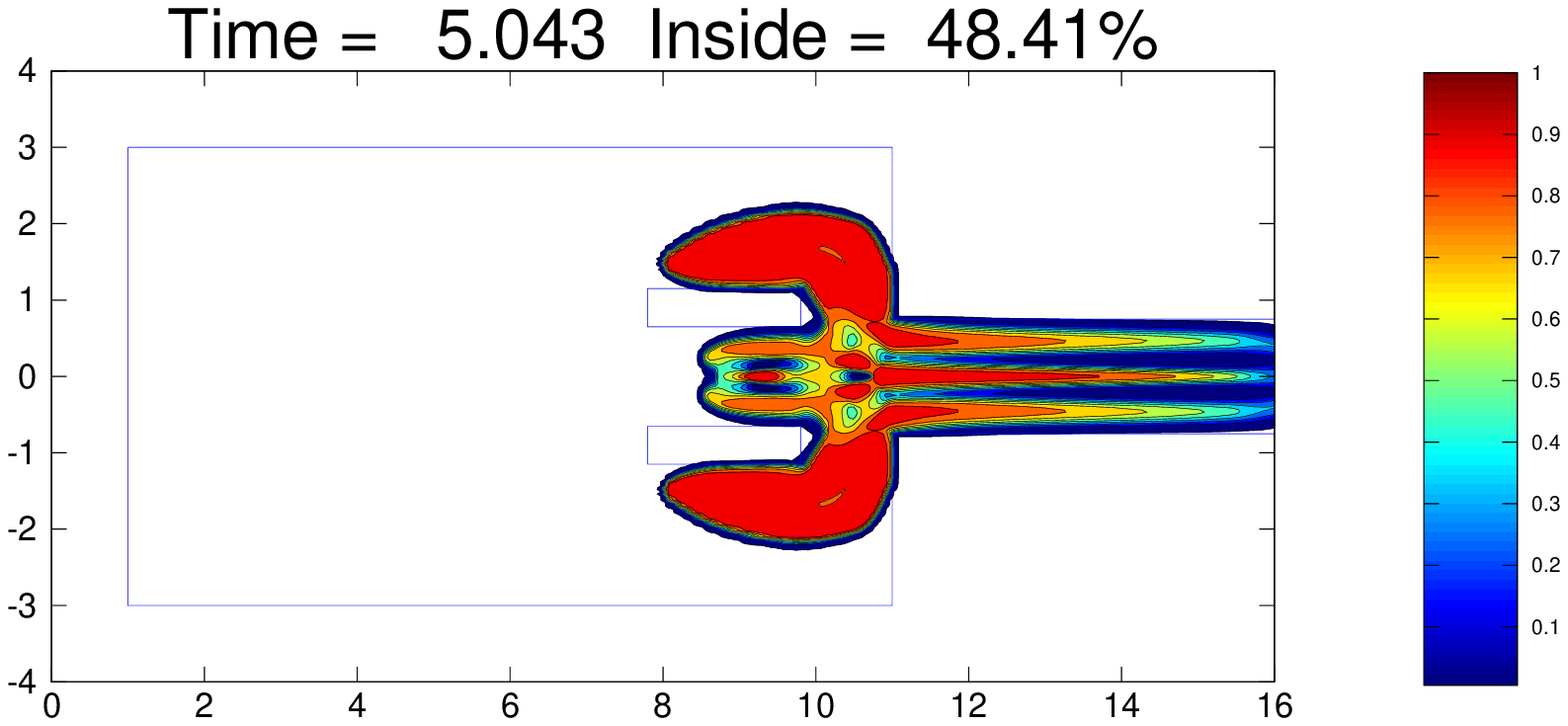}%
  \includegraphics[width=0.32\textwidth
]{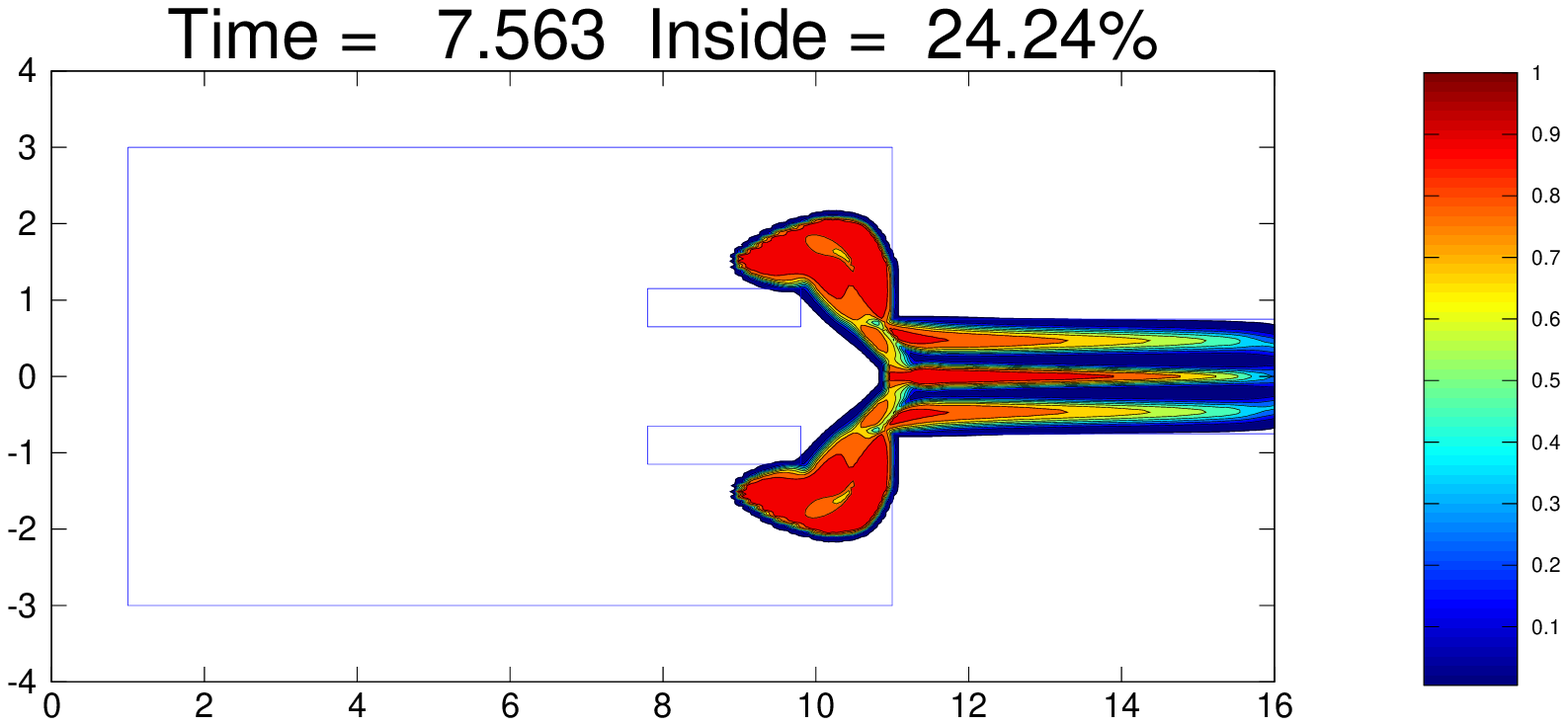}\\
  \includegraphics[width=0.32\textwidth
]{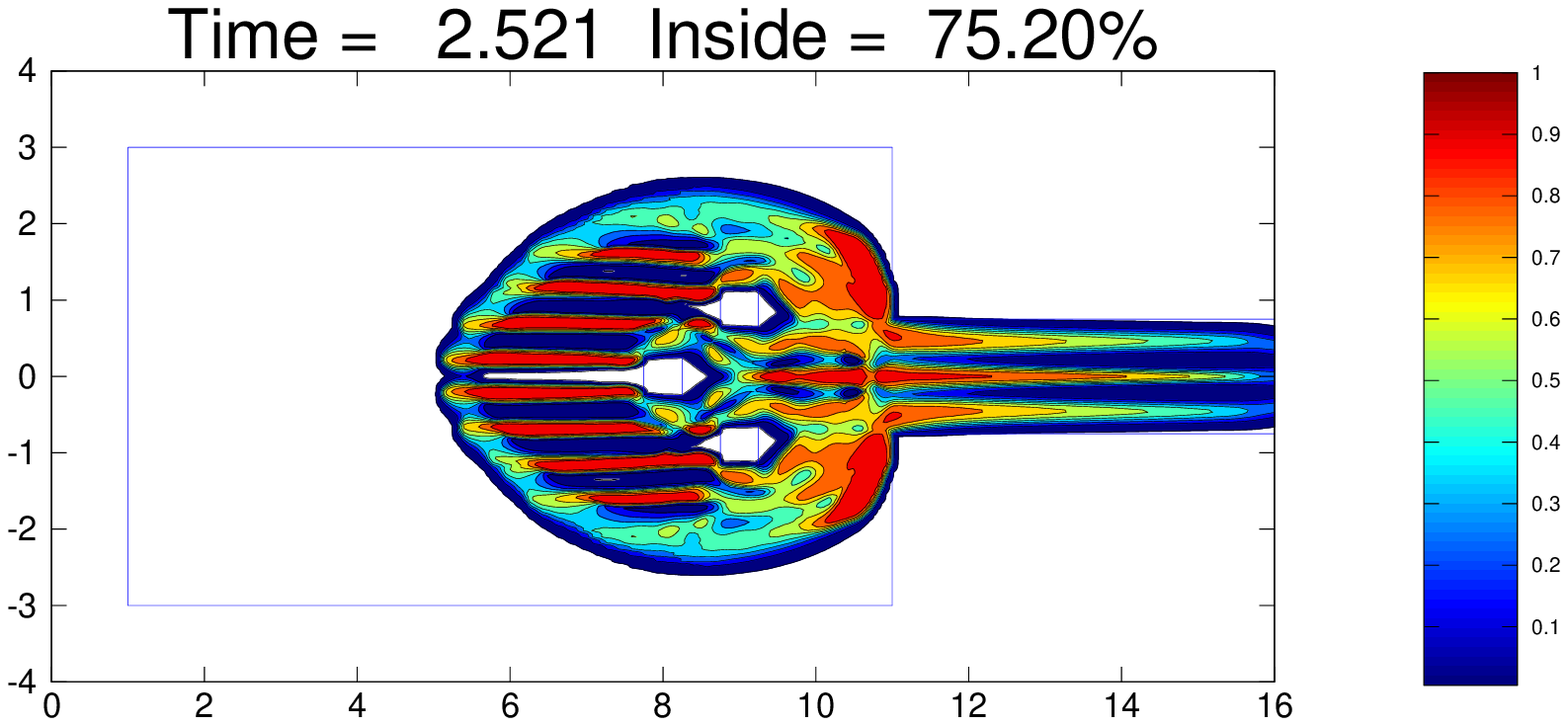}%
  \includegraphics[width=0.32\textwidth
]{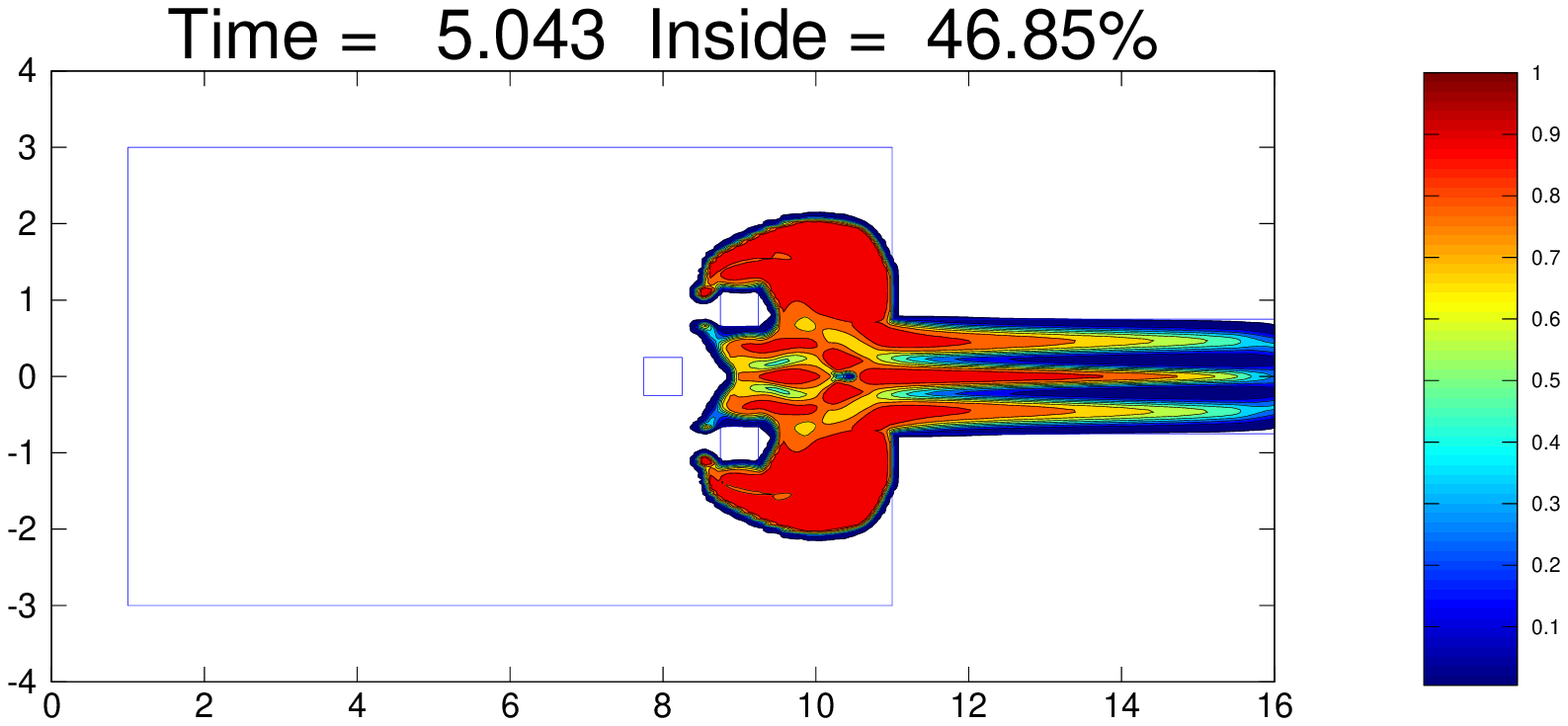}%
  \includegraphics[width=0.32\textwidth
]{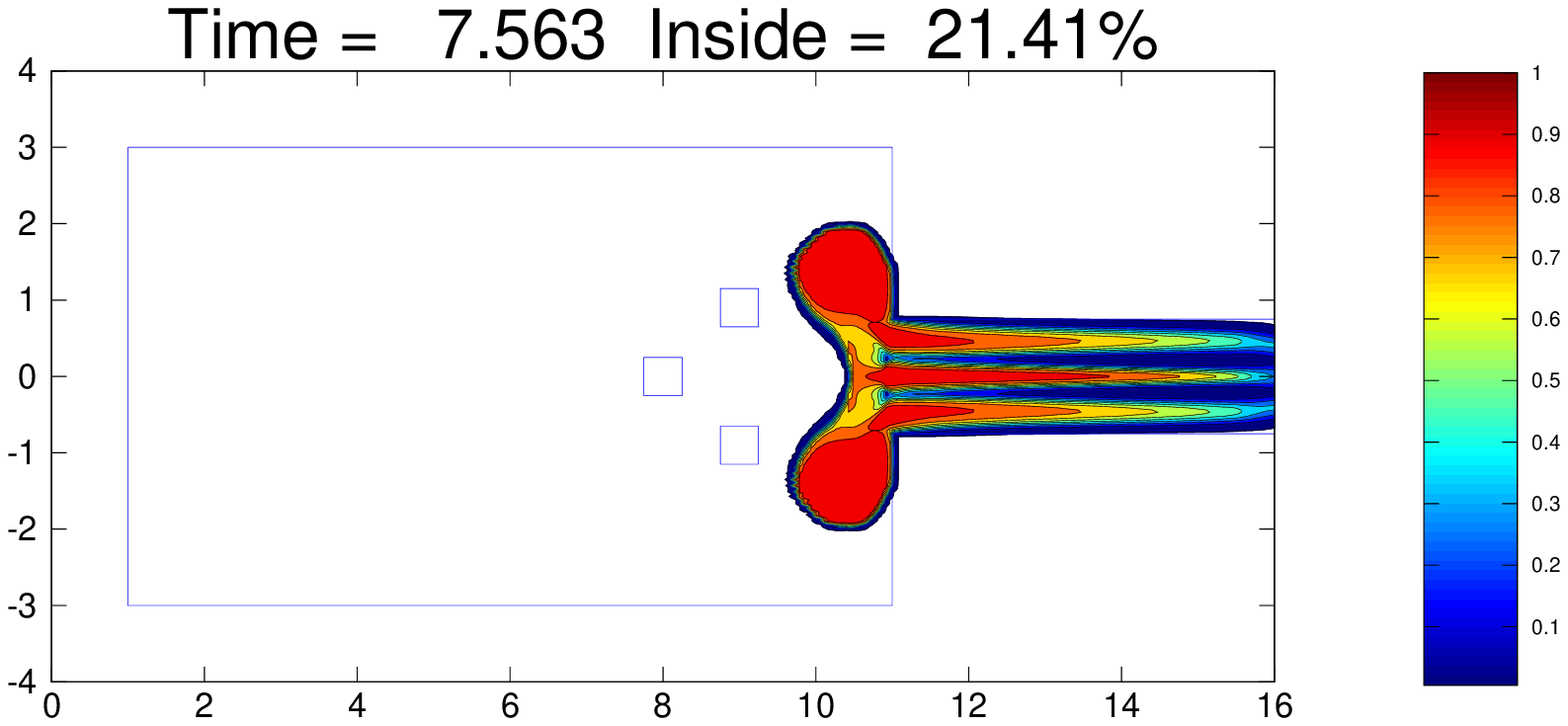}%
  \caption{Solution
 to~\eqref{eq:General}--\eqref{eq:IGood}--\eqref{eq:nu}--\eqref{eq:Evacuation}
    with different geometries, computed at time $t = 2.521$, $5.043$
    and $7.563$.}
  \label{fig:Evacuation1}
\end{figure}
The number of lanes self adapts to the available space, with three
lanes merging into one before the bottleneck. On the second line of
Figure~\ref{fig:Evacuation1}, the insertion of three columns in these
positions delays but does not avoid the congestion at the exit. Those
individuals that pass through the bottleneck are favored in exiting
the room.

In Figure~\ref{fig:Evacuation2}, first line, the insertion of four columns is more successful. Note, in the first two diagrams, that the number of lanes changes from $7$ before the bottleneck, to $4$ in between it. The individuals that do not pass through the bottleneck are penalized by the high density they meet near the door jambs.
\begin{figure}[htpb]
  \centering
  \includegraphics[width=0.32\textwidth
]{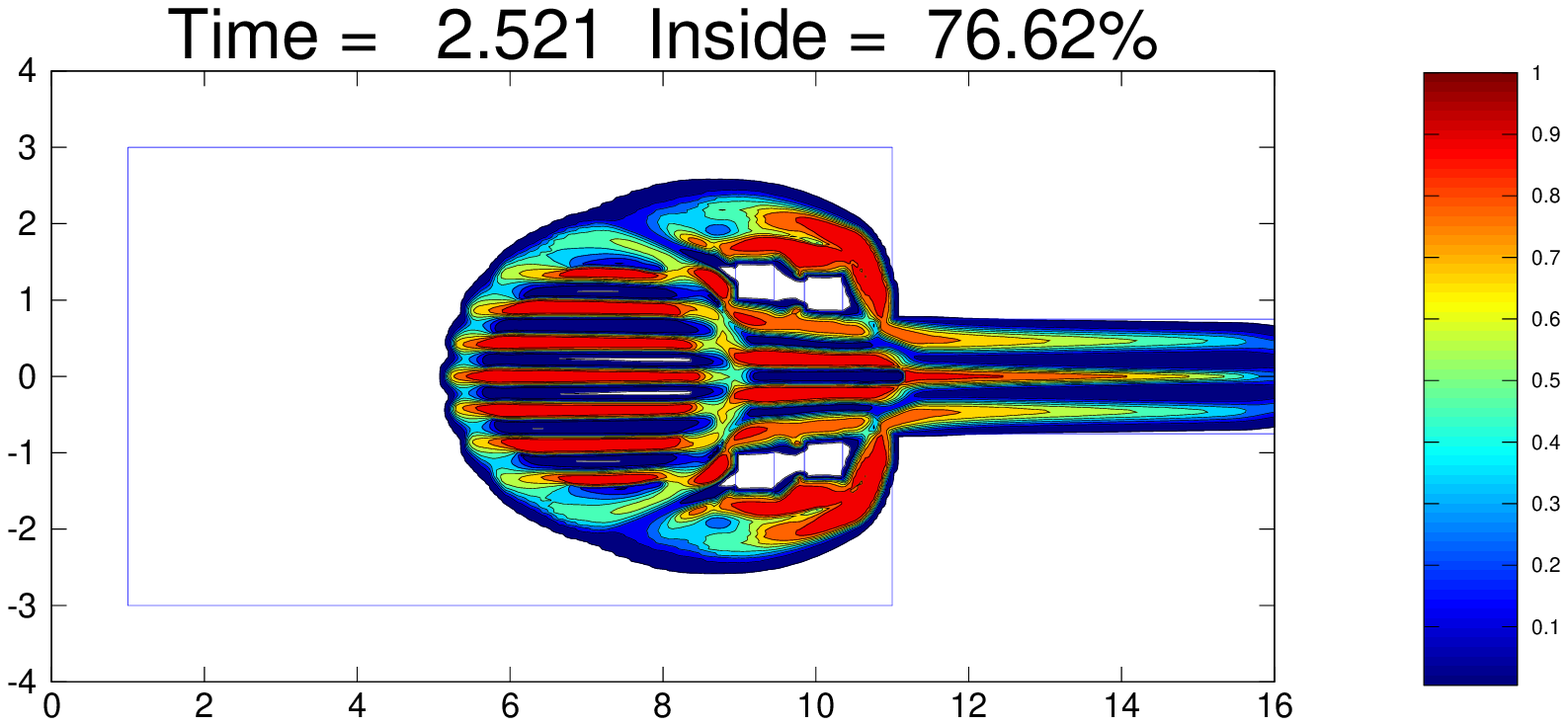}%
  \includegraphics[width=0.32\textwidth
]{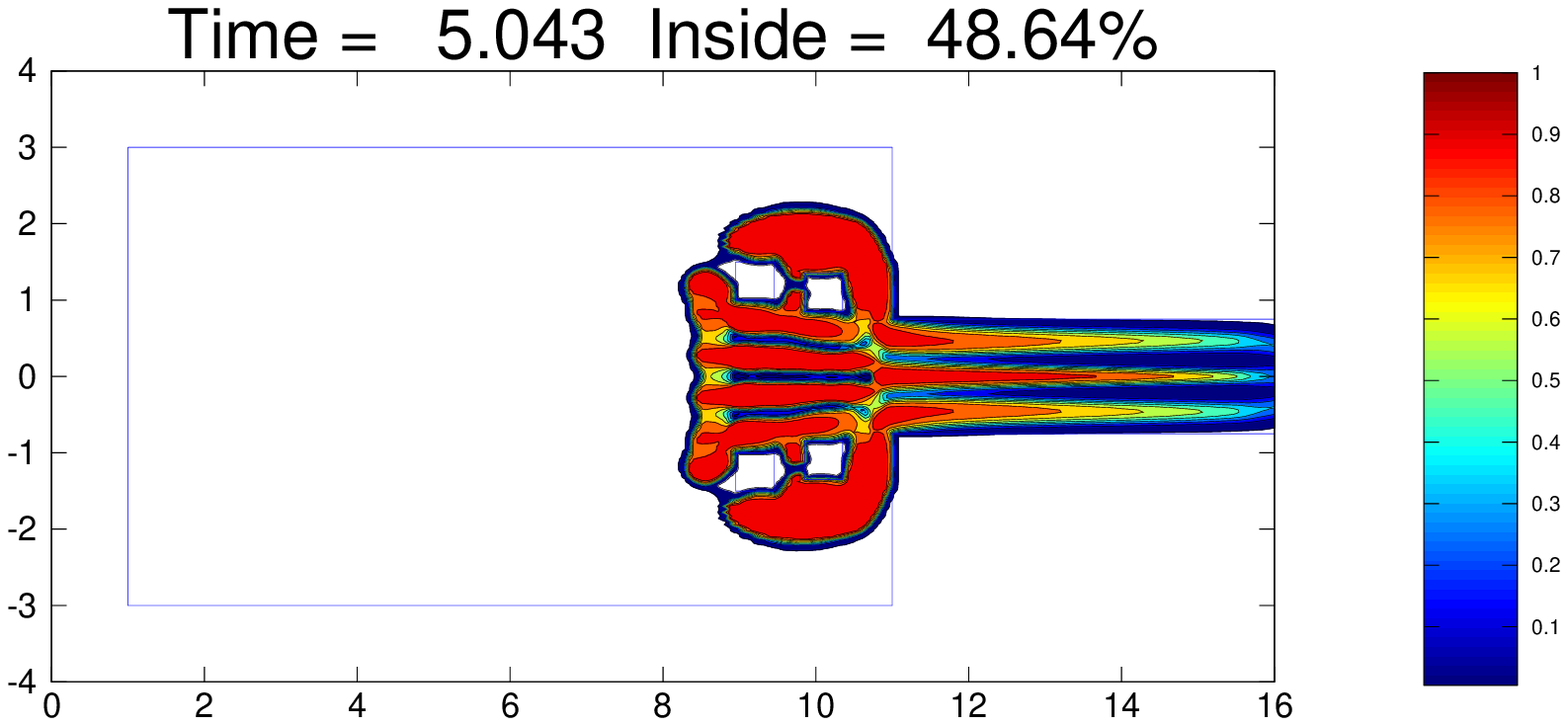}%
  \includegraphics[width=0.32\textwidth
]{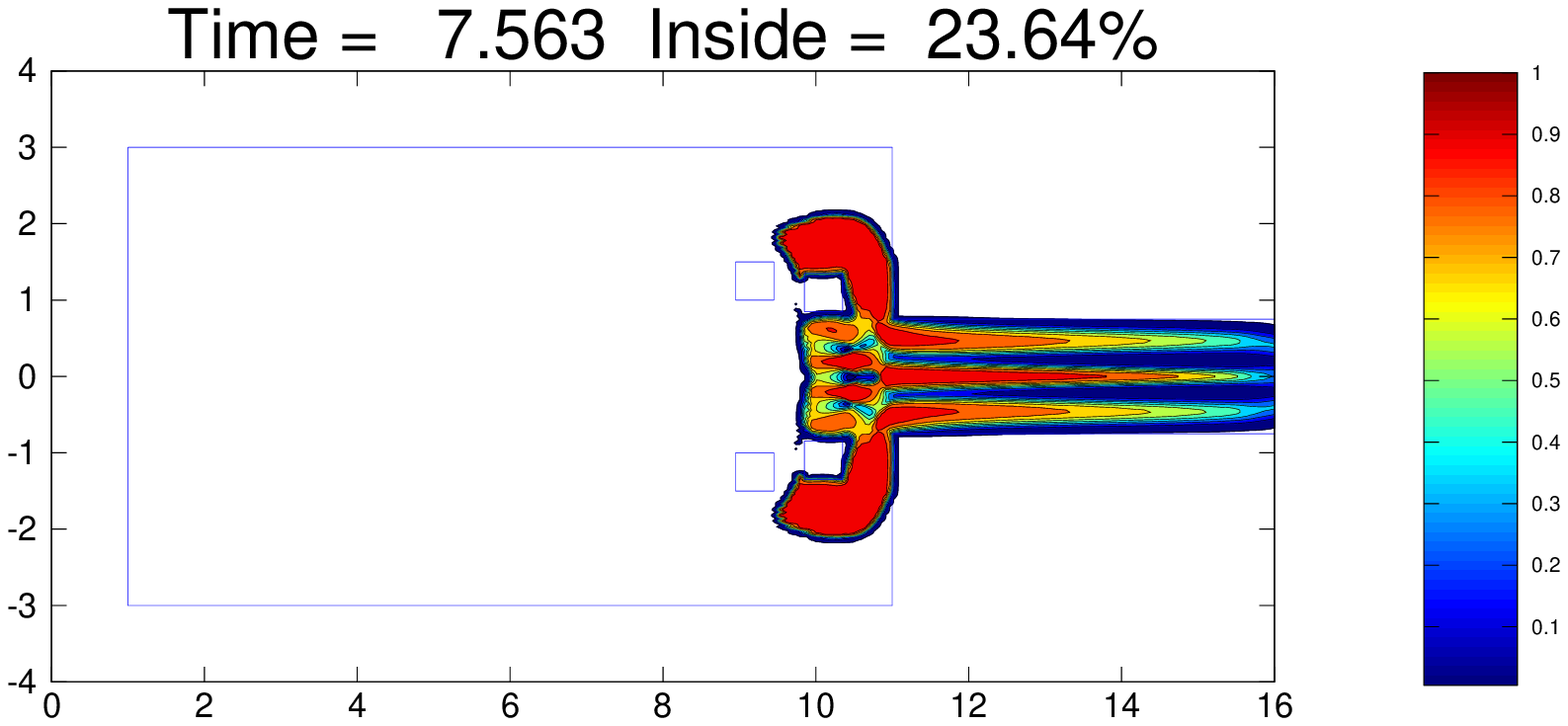}\\
  \includegraphics[width=0.32\textwidth
]{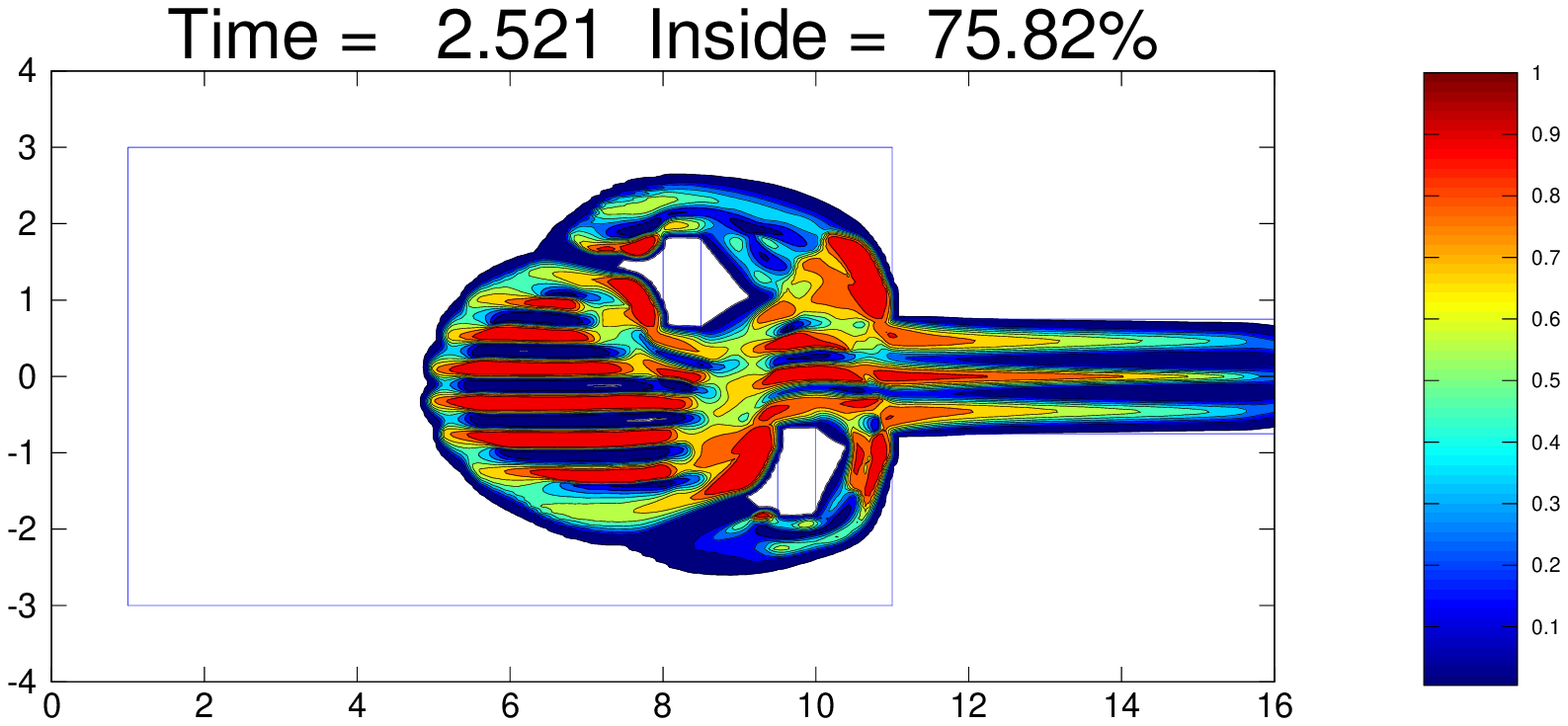}%
  \includegraphics[width=0.32\textwidth
]{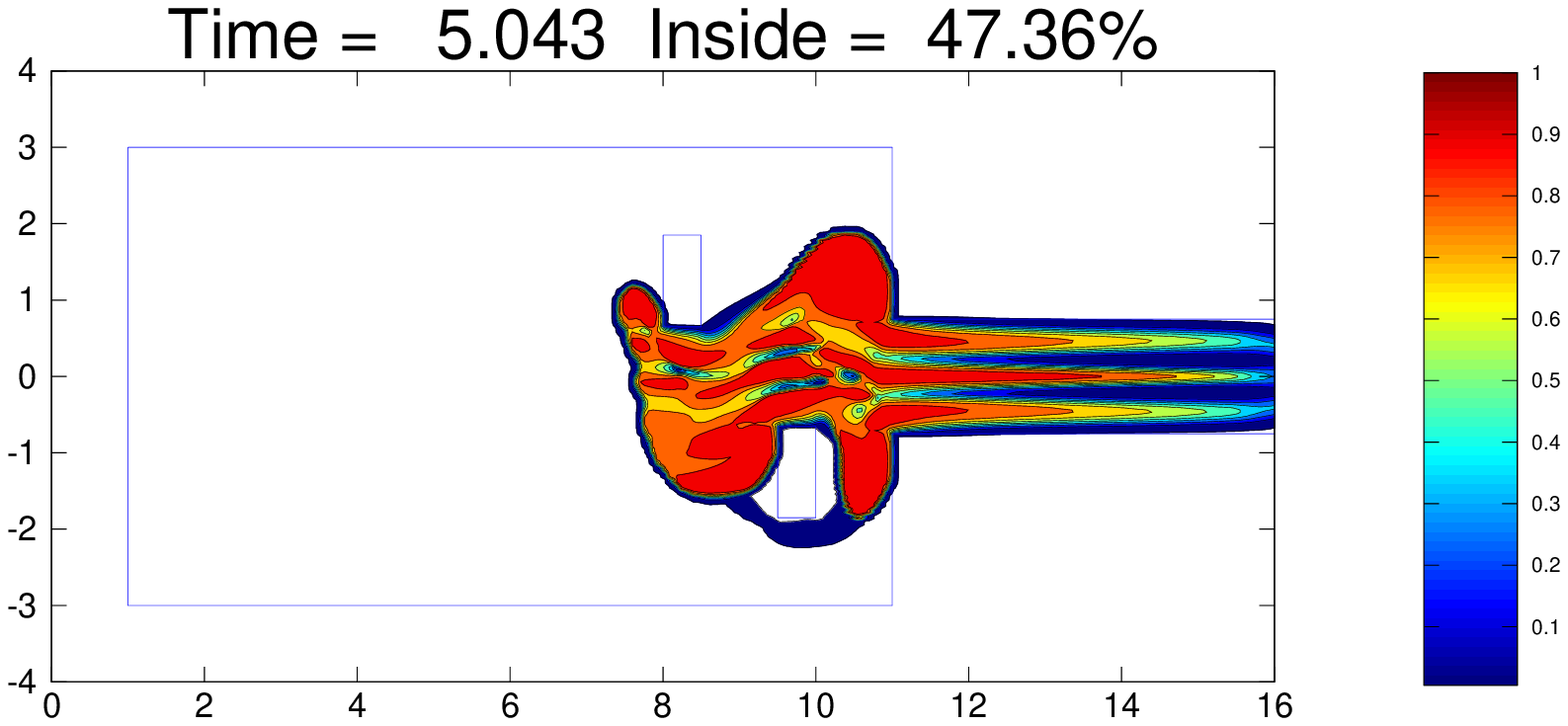}%
  \includegraphics[width=0.32\textwidth
]{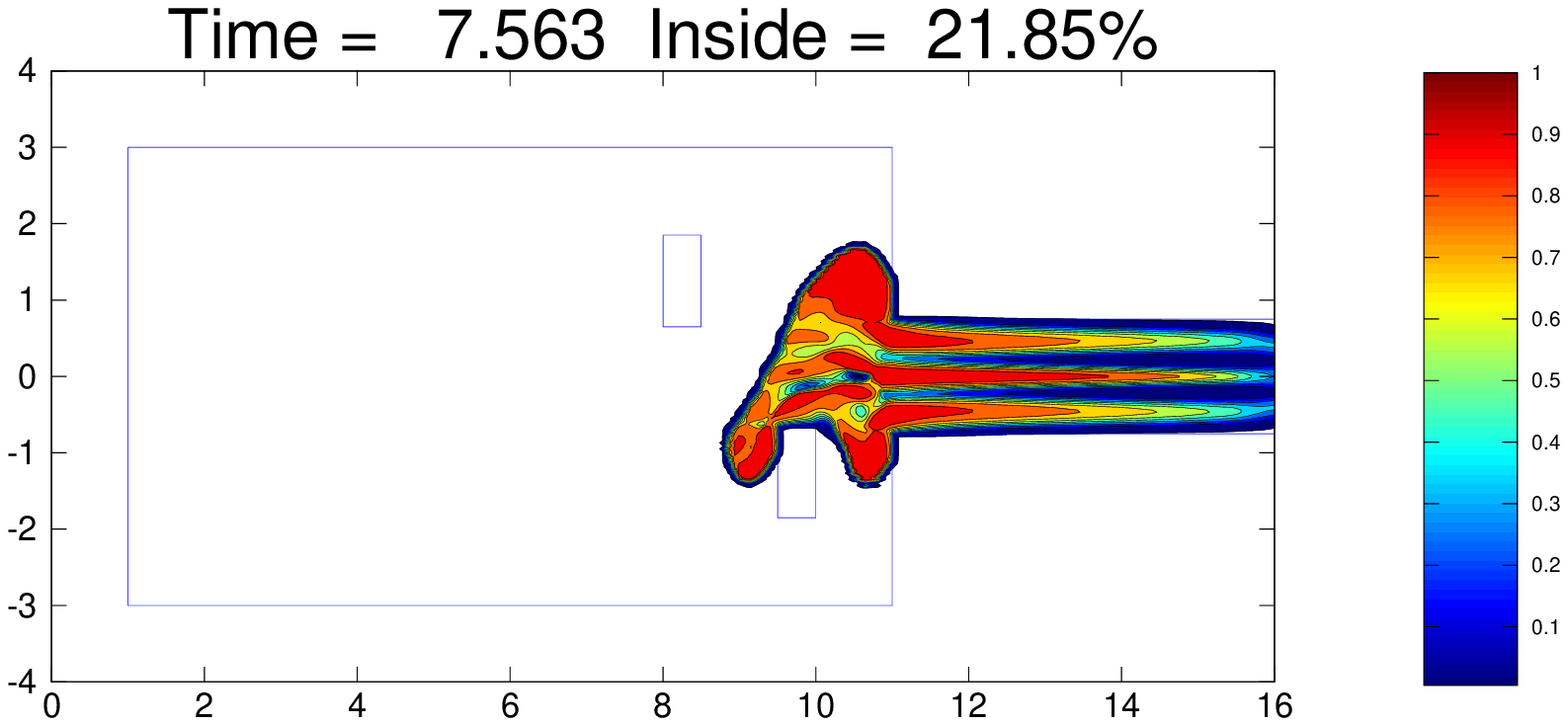}%
  \caption{Solution
 to~\eqref{eq:General}--\eqref{eq:IGood}--\eqref{eq:nu}--\eqref{eq:Evacuation}
    with different geometries, computed at time $t = 2.521$, $5.043$
    and $7.563$.}
  \label{fig:Evacuation2}
\end{figure}
Finally, on the second line in Figure~\ref{fig:Evacuation2}, an asymmetric layout hinders the lanes' pattern.

\subsection{On the Rise of Singularities}
\label{subs:Rise}

A typical feature of conservation laws is the possible rise of
singularities, see for
instance~\cite[Example~1.4]{BressanLectureNotes}. The nonlocal
equation~\eqref{eq:General} shares this characteristic. Indeed, assume
$\rho = \rho (t,x)$ is a given solution to~\eqref{eq:General}, smooth
up to time $T>0$. Then, setting $w (t,x) = \left(\nu (x) +
  \left(\mathcal{I} (t) \right) (x) \right)$, simple computations lead
to the following equation for the space derivative $\rho_j$ of $\rho$
in the direction $x_j$, for $j=1, \ldots, N$:
\begin{displaymath}
  \partial_t \rho_j
  +
  (\div \rho_j) \, q' \, w
  =
  (\rho_j)^2 \, q''\, w
  +
  \rho_j \, q'' \, \sum_{i\neq j} (\partial_i \rho) \, w
  +
  \rho_j \, q' \, \div w
  +
  q' \, \sum (\partial_i \rho) \, \partial_j w
  +
  q \, \partial_j \div w \,.
\end{displaymath}
The first term in the right hand side is quadratic in $\rho_j$,
showing that a blow up of $\rho_j$ may take place in finite time.

\begin{figure}[htpb]
  \centering
  \includegraphics[width=75mm]{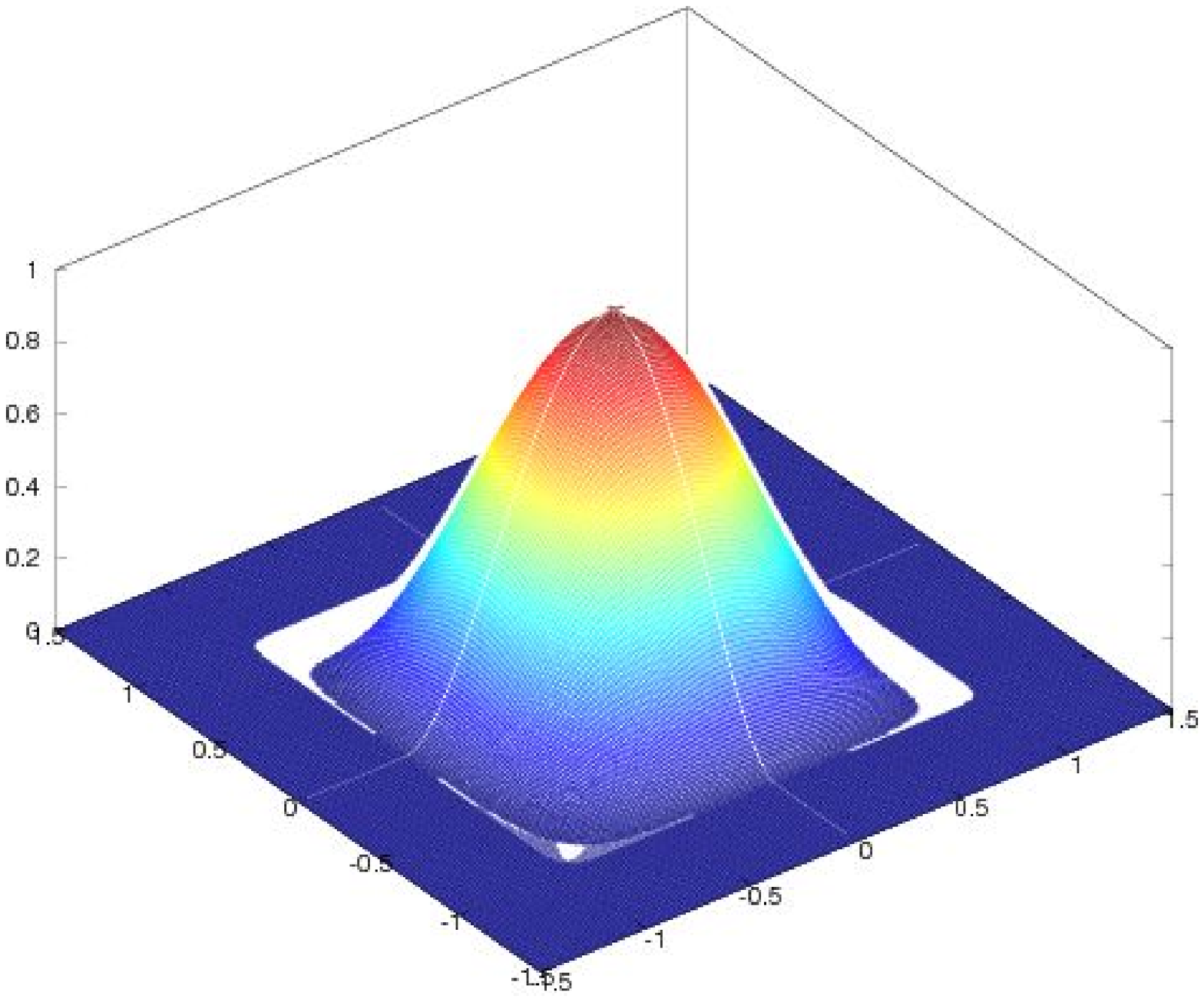}\quad
  \includegraphics[width=75mm]{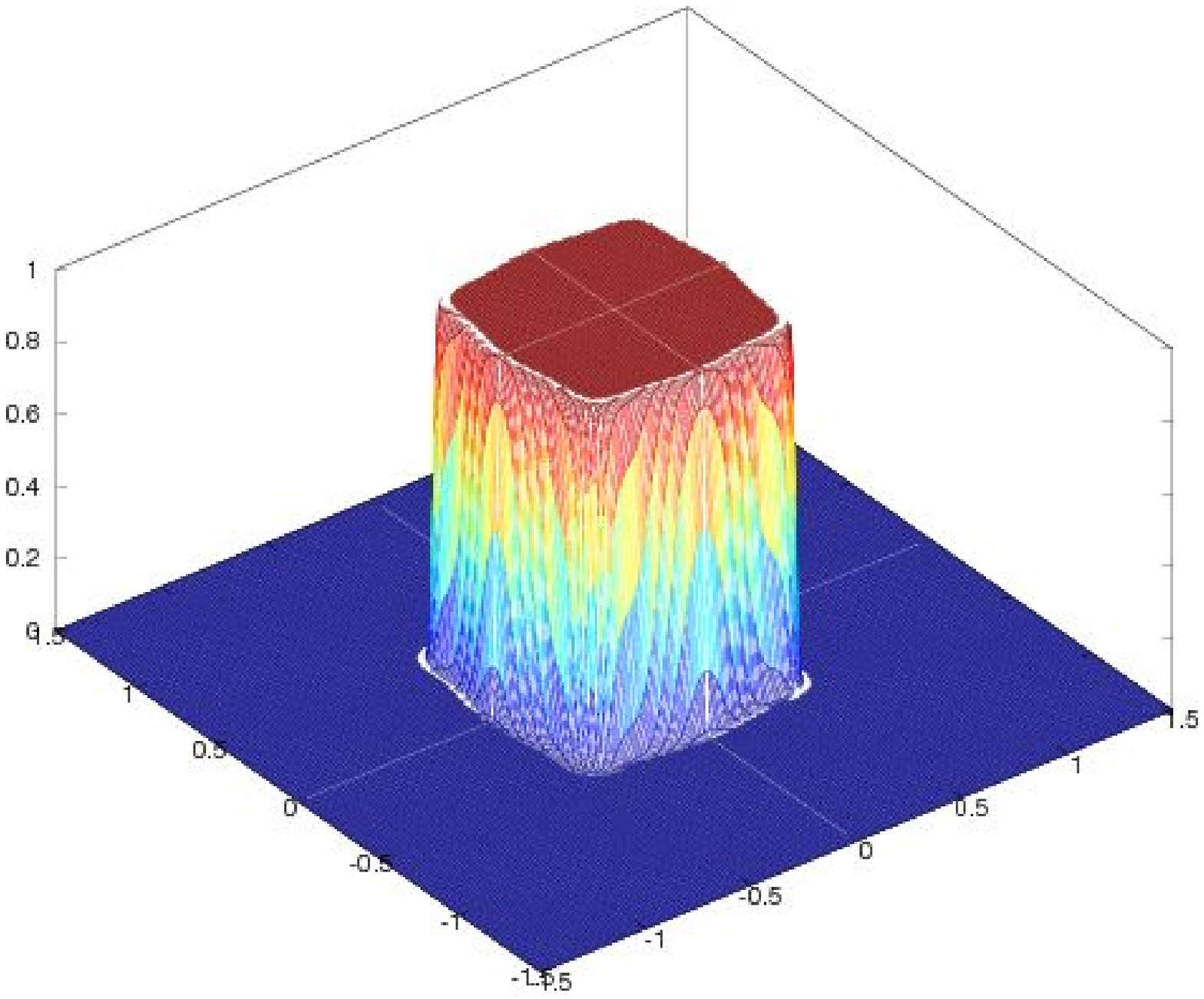}\vspace{-\baselineskip}
  \caption{Left, the initial datum and, right, the solution
    to~\eqref{eq:General}--\eqref{eq:IGood}--\eqref{eq:singo} at time
    $t=1$. Note the formation of vertical faces in the initially
    smooth distribution.}
  \label{fig:singo}
\end{figure}

Moreover, we consider~\eqref{eq:General} with $\mathcal{I}$ as
in~\eqref{eq:IGood} and
\begin{equation}
  \label{eq:singo}
  \begin{array}{rcl@{\quad}rcl}
    v (\rho)
    & = &
    1-\rho\,,
    &
    \eta (x,y)
    & = &
    (1-16x^2)^3 \, (1-16y^2)^3 \, \caratt{[-1/4,1/4]^2} (x,y)\,,
    \\
    \nu (x)
    & = &
    0\,,
    &
    \rho_0 (x,y)
    & = &
    (1-4x^2/9)^2 \, (1-4y^2/9)^2 \, \caratt{[-3/2,3/2]^2} (x,y)\,,
  \end{array}
  \quad
  \epsilon=-1
\end{equation}
and we obtain the solution in Figure~\ref{fig:singo}.

\section{Technical Details}
\label{sec:TD}

\subsection{Existence and uniqueness}
\begin{proofof}{Lemma~\ref{lem:Kruzkov}}
  Let $q (\rho) = \rho \, v (\rho)$.  Thanks
  to~\textbf{($\boldsymbol{\nu}$)}, the assumptions on $r$
  and~\textbf{(I.1)}, we have for all $M>0$
  \begin{displaymath}
    \begin{array}{r@{\;}c@{\;}lr@{\;}c@{\;}l}
      \partial_\rho f (t,x,\rho)
      & = &
      q'(\rho) \,
      \left(\nu (x) + \mathcal{I}\left(r (t)\right) (x)\right);
      &
      \partial_\rho f
      & \in &
      \L\infty([0,T]\times\reali^N\times [-M,M]; \reali^N)\,;
      \\
      \div f(t,x,\rho)
      & = &
      q(\rho) \,
      \div \left(\nu (x) + \mathcal{I}\left(r (t)\right) (x)\right);
      &
      \div f
      & \in &
      \L\infty([0,T]\times\reali^N\times [-M,M]; \reali)\,.
    \end{array}
  \end{displaymath}
  Thus, we can apply Kru\v zkov Theorem~\cite[Theorem~5 \&
  §5.4]{Kruzkov} and ensure that~\eqref{eq:K} admits a unique Kru\v
  zkov solution $\rho \in \L\infty \left( \rpic; \L1(\reali^N; \reali)
  \right)$, which is continuous from the right in time.
\end{proofof}

\begin{lemma}
  \label{lem:loc}
  Let $\rho_0 \in \L1(\reali^N; [0,R])$, $r \in
  \C0\left(\rpic; \L1(\reali^N; [0,R])\right)$. Under hypotheses
  \textbf{(v)}--\textbf{($\boldsymbol{\nu}$)}--\textbf{(I.1)}, the
  Cauchy problem~\eqref{eq:K}
  admits a unique weak entropy solution $\rho$, with $\rho \in \C0
  \left( \rpic; \L1(\reali^N; [0,R]) \right)$ satisfying, for all $t\in
  \rpic$,
  \begin{equation}
    \label{est:L1}
    \norma{\rho(t)}_{\L1} = \norma{\rho_0}_{\L1}\,.
  \end{equation}
  Assume, in addition, that~\textbf{(I.2)} is satisfied. Then, $\rho_0
  \in \BV(\reali^N; [0,R])$ implies $\rho(t) \in \BV(\reali^N;[0,R])$
  for all time $t\geq 0$. Moreover, the following bound is satisfied
  \begin{equation}
    \label{est:bv}
    \tv\left(\rho(t)\right)
    \leq
    \left[
      \tv(\rho_0)
      +
      t \, N W_N\norma{q}_{\L\infty([0,R])}
      \left(
        \norma{\nabla\div \nu}_{\L1} + C_I(\norma{r}_{\L\infty([0,T];\L1)})
      \right)
    \right] e^{\kappa_0^* t},
  \end{equation}
  where $W_N=\int_0^{\pi/2}(\cos \theta)^N\d{\theta}$ and the constant
  $\kappa_0^*$ is bounded above as follows:
  \begin{displaymath}
    \kappa_0^*
    \leq
    (2N+1)
    \norma{q'}_{\L\infty([0,R])}
    \left(
      \norma{\nabla \nu}_{\L\infty}
      +
      C_I(\norma{r}_{\L\infty([0,T];\L1)})
    \right)\,.
  \end{displaymath}
\end{lemma}

\begin{proof}
  Below we denote $q(\rho) = \rho \, v (\rho) $, so that $f(t,x,\rho)
  = q(\rho)\, \left(\nu (x) + \left(\mathcal{I} \left(r (t) \right)
    \right) (x) \right)$. The existence of a solution follows from
  Lemma~\ref{lem:Kruzkov}.  The rest of the proof is obtained through
  the following steps.

  \paragraph{Estimates in $\L1$ and $\L\infty$.}
  Since $\rho$ is a solution to a conservation law, (\ref{est:L1}) is
  immediately satisfied.

  Besides, $\rho\equiv 0$ and $\rho\equiv R$ are solutions
  to~(\ref{eq:K}) associated respectively to the constant initial
  conditions $\rho_0 \equiv 0$ and $\rho_0 \equiv R$. Hence, we can
  apply the comparison Theorem~\cite[Theorem~3]{Kruzkov} and obtain
  that if the initial condition takes value in $[0,R]$, then the
  solution takes values in $[0,R]$, so that $\rho \in \L\infty \left(
    \rpic; \L1(\reali^N; [0,R]) \right)$.

  \paragraph{Continuity in time.}
  Thanks to~\cite[Remark~2.4]{ColomboMercierRosini}
  or~\cite[Corollary~2.4]{MercierStability}, $\rho$ is continuous in
  time if, for any $T>0$,
  \begin{equation}
    \label{eq:cond}
    \norma{q}_{\L\infty ([0,R])}
    \int_0^T\int_{\reali^N}
    \modulo{\div \left(\nu (x) + \left(\mathcal{I} \left(r (t) \right)
        \right) (x) \right)} \d{x}\d{t}
    <
    \infty\,.
  \end{equation}
  We have, thanks to~\textbf{(I.1)}, for all $t\geq 0$
  \begin{displaymath}
    \norma{\div \left(\nu + \mathcal{I} \left(r (t) \right) \right)}_{\L1}
    \leq
    \norma{\div \nu}_{\L1}
    +
    C_I(\norma{r (t)}_{\L1})\,.
  \end{displaymath}
  Hence, condition~(\ref{eq:cond}) is satisfied under
  hypotheses~\textbf{(v)}--\textbf{($\boldsymbol{\nu}$)}--\textbf{(I.1)}
  and we also obtain $\rho\in \C0(\rpic; \L1(\reali^N; [0,R]))$.

  \paragraph{Estimate in $\BV$.}
  To prove the bound on the $\tv$ norm we
  use~\cite[Theorem~2.2]{MercierStability}. To this aim, we have to
  check that for any $T>0$,
  \begin{displaymath}
    \nabla\partial_\rho f \in \L\infty([0,T]\times \reali^N\times [0,R];
    \reali^{N\times N})
    \quad \mbox{ and } \quad
    \int_0^T \!\!\! \int_{\reali^N}
    \norma{\nabla \div f (t,x, \cdot)}_{\L\infty([0,R])} \d{x} \d{t} < \infty\,.
  \end{displaymath}
  Note first that $\nabla\partial_\rho f (t,x,\rho) = q'(\rho) \, \nabla
  \left(\nu + \mathcal{I} \left(r (t) \right) \right)$. Hence, thanks
  to \textbf{(v)}--\textbf{($\boldsymbol{\nu}$)} and~\textbf{(I.1)},
  we have
  \begin{displaymath}
    \norma{\nabla\partial_\rho f (t,x,\rho)}_{\L\infty}
    \leq
    \norma{q'}_{\L\infty([0,R])}
    \left(
      \norma{\nabla \nu}_{\L\infty}
      +
      C_I(\norma{r (t)}_{\L1})
    \right)\,,
  \end{displaymath}
  Thanks to \textbf{(v)}--\textbf{($\boldsymbol{\nu}$)}
  and~\textbf{(I.2)}, we have, for all $t\geq 0$
  \begin{displaymath}
    \norma{\nabla\div
      \left(\nu + \mathcal{I} \left(r (t) \right) \right)}_{\L1}
    \leq
    \norma{\nabla\div \nu}_{\L1}
    +
    C_I(\norma{r (t)}_{\L1}) \,.
  \end{displaymath}
  Applying~\cite[Theorem~2.2]{MercierStability}, we obtain
  \begin{eqnarray*}
    \kappa_0^*
    & = &
    (2N+1)\norma{\nabla \partial_\rho f}_{\L\infty}
    \\
    & \leq &
    (2N+1) \, \norma{q'}_{\L\infty([0,R])}
    \left(
      \norma{\nabla \nu}_{\L\infty}
      +
      C_I(\norma{r}_{\L\infty([0,T];\L1)})
    \right)\,,
    \\
    \tv(\rho(t))
    & \leq &
    \tv(\rho_0) \, e^{\kappa_0^* t}
    \\
    & &
    +
    t e^{\kappa_0^* t} N W_N \norma{q}_{\L\infty([0,R])}
    \left(
      \norma{\nabla\div \nu}_{\L1}
      +
      C_I(\norma{r}_{\L\infty([0,T];\L1)})
    \right) \,,
  \end{eqnarray*}
  completing the proof.
\end{proof}

\begin{proofof}{Theorem~\ref{thm:existence}}
  Thanks to Lemma~\ref{lem:loc}, under
  hypotheses~\textbf{(v)}--\textbf{($\boldsymbol{\nu}$)}
  and~\textbf{(I)}, for any $r \in \C0 \left([0,T]; \L1
    (\reali^N;[0,R])\right)$ there exists a unique solution $\rho \in
  \C0 \left(\rpic; \L1(\reali^N; [0,R]) \right)$ to~(\ref{eq:K})
  satisfying~(\ref{est:L1})--(\ref{est:bv}).  For any time $T > 0$,
  introduce the space
  \begin{displaymath}
    X(T) =
    \left\{
      r\in \C0 \left([0,T]; \L1(\reali^N; [0,R]) \right) \colon
      \forall\, t\in [0,T],
      \norma{r(t)}_{\L1} \leq \norma{\rho_0}_{\L1}
    \right\} \,.
  \end{displaymath}
  Note that $X (T)$ is a Banach space with respect to the norm
  $\norma{r}_{X (T)} = \sup_{t \in [0,T]} \norma{r(t)}_{\L1}$. Let
  $\mathcal{Q}$ be the map that associates to $r\in X(T)$ the solution
  $\rho \in X(T)$ to~(\ref{eq:K}). We want to find a condition on $T$
  so that $\mathcal{Q}$ is a contraction.

  Let $r_1, r_2 \in X(T)$. To
  apply~\cite[Theorem~2.6]{MercierStability}, we have to check that,
  for all $t \in [0,T]$,
  \begin{eqnarray*}
    & &
    \partial_\rho (f_1-f_2)
    \in
    \L\infty([0,T]\times \reali^N\times [0,R]; \reali^N)
    \,, \mbox{ and}
    \\
    & &
    \int_0^T \!\! \int_{\reali^N}
    \norma{\div (f_1-f_2) (t,x,\cdot)}_{\L\infty([0,R])} \d{x} \d{t}
    <
    +\infty\,.
  \end{eqnarray*}
  Thanks to~\textbf{(I.3)} we have
  \begin{eqnarray*}
    & &
    \norma{\partial_\rho (f_1-f_2)}_{\L\infty}
    =
    \norma{\left(\mathcal{I}(r_1) - \mathcal{I}(r_2)\right) q'}_{\L\infty}
    \leq
    K_I \, \norma{q'}_{\L\infty ([0,R])} \,
    \norma{r_1-r_2}_{\L\infty([0,T];\L1)}\, ,
    \\
    & &
    \int_0^T \!\! \int_{\reali^N}
    \norma{\div (f_1-f_2) (t,x,\cdot)}_{\L\infty([0,R])} \d{x} \d{t}
    \leq
    T K_I \,
    \norma{q}_{\L\infty([0,R])} \,
    \norma{r_1-r_2}_{\L\infty([0,T];\L1)}\,.
  \end{eqnarray*}
  Hence, we get for all $t\geq 0$
  \begin{eqnarray*}
    & &
    \norma{(\rho_1-\rho_2)(t)}_{\L1}
    \\
    &\leq &
    t \,K_I  \norma{q}_{\W1\infty([0,R])} \norma{r_1-r_2}_{\L\infty([0,T];\L1)}
    \\
    & &
    \times
    \big[1+ e^{\kappa_0^* t}
    \left(
      \tv(\rho_0)
      +
      \frac{N W_N}{2}\,t\, \norma{q}_{\L\infty([0,R])}
      \left(\norma{\nabla\div \nu}_{\L1}
        +
        C_I(\norma{r}_{\L\infty([0,T];\L1)})
      \right)
    \right)
    \big]\,.
  \end{eqnarray*}
  Denoting
  \begin{eqnarray}
    k
    & = &
    (2N+1) \norma{q'}_{\L\infty([0,R])}
    \left(
      \norma{\nabla \nu}_{\L\infty}
      +
      C_I(\norma{\rho_0}_{\L1})
    \right)\,, \label{eq:k}
    \\
    C_1
    & = &
    K_I \,  \norma{q}_{\W1\infty ([0,R])} \,, \nonumber
    \\
    C_2
    & = &
    \frac{N W_N}{2} \norma{q}_{\L\infty([0,R])}
    \left(
      \norma{\nabla\div \nu}_{\L1}
      +
      C_I(\norma{\rho_0}_{\L1})
    \right)\,.\nonumber
  \end{eqnarray}
  The above estimate can be written
  \begin{displaymath}
    \norma{(\rho_1-\rho_2)(t)}_{\L1}
    \leq
    C_1 \, t \, \norma{r_1-r_2}_{\L\infty([0,T];\L1)}
    \left(
      1
      +
      e^{k t} \left(\tv(\rho_0) + C_2t
      \right)
    \right)\,,
  \end{displaymath}
  We choose now $T$ so that
  \begin{displaymath}
    C_1 \, T \, \left( 1 + e^{k T} \left(\tv(\rho_0) + C_2\, T \right) \right)
    =
    \frac{1}{2}
  \end{displaymath}
  and, applying Banach Fixed Point Theorem, we obtain a unique fixed
  point for $\mathcal{Q}$ on $X(T)$.

  We now iterate the procedure above. To this aim, we use the total
  variation estimate~\eqref{est:bv}
  \begin{displaymath}
    \tv\left(\rho(t)\right)
    \leq
    \left(
      \tv(\rho_0)
      +
      2\, C_2 t
    \right) e^{k t},
  \end{displaymath}
  and, given $T_n$, we recursively define $T_{n+1}$ so that
  \begin{displaymath}
    C_1 \, (T_{n+1} - T_n)
    \left(
      1
      +
      e^{k T_{n+1}} \tv (\rho_0)
      +
      C_2 \, e^{k (T_{n+1}-T_n)} (2T_n e^{kT_n} +T_{n+1} - T_n)
    \right)
    =
    \frac{1}{2}
  \end{displaymath}
  and the above procedure ensures the existence of a fixed point on
  the interval $[T_n,T_{n+1}]$.

  The sequence $T_n$ is strictly increasing. If it is bounded, then
  the latter relation above yields $0=1/2$. Hence, $T_n\to \infty$
  when $n\to \infty$, completing the proof.
\end{proofof}

\subsection{Stability}

\begin{proofof}{Theorem~\ref{thm:stability}}
  Here, we apply~\cite[Corollary 2.8]{MercierStability} to compare the
  solutions $\rho_1$ and $\rho_2$ of~(\ref{eq:2eq}). Let $k$ be as
  in~(\ref{eq:k}).  As in the proof of Theorem~\ref{thm:existence},
  \textbf{(v)}--\textbf{($\boldsymbol{\nu}$)}--\textbf{(I)} ensure
  that the set of
  hypotheses~\cite[\textbf{(H1*)}--\textbf{(H2*)}--\textbf{(H3*)}]{MercierStability}
  are satisfied. Thus, for any $t\geq 0$, we obtain
  \begin{eqnarray*}
    & &
    \norma{\rho_1 (t) -\rho_2(t)}_{\L1}
    \\
    & \leq &
    \norma{\rho_{0,1}-\rho_{0,2}}_{\L1}
    \\
    & &
    +
    \big[
    e^{k t } \tv(\rho_{0,1})
    +
    NW_N \norma{q_1}_{\L\infty([0,R])}
    \int_0^t e^{k(t-\tau)} \int_{\reali^N}
    \norma{\nabla\div(\nu_1(x)+\mathcal{I}_1(\rho_1))}
    \d{x}\d{\tau}\big]
    \\
    & &
    \qquad
    \times
    \int_0^t
    \norma{
      q_1'(\rho) \left(\nu_1(x) +\mathcal{I}_1(\rho_1)\right)
      -
      q_2'(\rho) \left(\nu_2(x) +\mathcal{I}_2(\rho_2)\right)
    }_{\L\infty(\reali^N)}
    \d{\tau}
    \\
    & &
    +
    \int_0^t \int_{\reali^N}
    \norma{
      q_1(\rho) \div\left(\nu_1(x) + \mathcal{I}_1(\rho_1)\right)
      -
      q_2(\rho) \div\left(\nu_2(x) + \mathcal{I}_2(\rho_2)\right)
    }_{\L\infty ([0,R])}
    \d{x} \d{\tau}
    \\
    & \leq &
    \norma{\rho_{0,1}-\rho_{0,2}}_{\L1}
    +
    \big[
    \tv(\rho_{0,1})
    +
    t\,NW_N  \norma{q_1}_{\L\infty([0,R])}
    \left(
      \norma{\nabla\div \nu_1}_{\L1}
      +
      C_I(\norma{\rho_{0,1}}_{\L1} ) \right)
    \big]
    \\
    & &
    \qquad
    \times
    \big(
    t \, \norma{q_1'-q_2'}_{\L\infty([0,R])}
    \norma{\nu_1+\mathcal{I}_1(\rho_1)}_{\L\infty}
    \\
    & &
    \qquad\qquad
    +
    t\,\norma{q_2'}_{\L\infty([0,R])}
    \norma{
      \nu_1 - \nu_2 + \mathcal{I}_1(\rho_1) -\mathcal{I}_2(\rho_1)
    }_{\L\infty}
    \\
    & &
    \qquad\qquad
    +
    \norma{q_2'}_{\L\infty([0,R])}
    \int_0^t \norma{
      \mathcal{I}_2\left(\rho_1(\tau)\right)
      -
      \mathcal{I}_2\left(\rho_2(\tau)\right)
    }_{\L\infty}
    \d{\tau}
    \big)
    \\
    & &
    +
    t \norma{q_2}_{\L\infty([0,R])} \norma{
      \div\left(
        \nu_1 - \nu_2 + \mathcal{I}_1(\rho_1) - \mathcal{I}_2(\rho_1)
      \right)
    }_{\L1}
    \\
    & &  + t\norma{q_1-q_2}_{\L\infty([0,R])}\norma{\div(\nu_1+\mathcal{I}_1(\rho_1))}_{\L1}
    +
    \norma{q_2}_{\L\infty([0,R])} \int_0^t
    \modulo{\div\left(\mathcal{I}_2(\rho_1)-\mathcal{I}_2(\rho_2)\right)}
    \d\tau.
  \end{eqnarray*}
  Hence, denoting
  \begin{eqnarray*}
    a(t)
    & = &
    \tv(\rho_{0,1})
    +
    t\,NW_N \, \norma{q_1}_{\L\infty([0,R])}
    \left(
      \norma{\nabla\div \nu_1}_{\L1}
      +
      C_I(\norma{\rho_{0,1}}_{\L1} )
    \right)
    \\
    b(t)
    & = &
    K_I \left(
      \norma{q_2'}_{\L\infty([0,R])} e^{k t} a(t)
      +
      \norma{q_2}_{\L\infty([0,R])}
    \right)
    \\
    \gamma_1
    & = &
    \norma{q_1'-q_2'}_{\L\infty([0,R])}
    \left(
      \norma{\nu_1}_{\L\infty}
      +
      C_I(\norma{\rho_{0,1}}_{\L1} )
    \right)
    +
    \norma{q_2'}_{\L\infty([0,R])} \norma{\nu_1-\nu_2}_{\L\infty}
    \\
    & &
    \qquad
    +
    \norma{q_2'}_{\L\infty([0,R])} \sup_{\rho\in\L1(\reali^N;[0,R])}
    \norma{\mathcal{I}_1(\rho)-\mathcal{I}_2(\rho)}_{\L\infty}
    \\
    \gamma_2
    & = &
    \norma{q_1-q_2}_{\L\infty([0,R])}
    \left(
      \norma{\div\nu_1}_{\L1}
      +
      C_I(\norma{\rho_{0,1}}_{\L1} )
    \right)
    \\
    & &
    +
    \norma{q_2}_{\L\infty([0,R]) }
    \left(
      \norma{\div(\nu_1-\nu_2)}_{\L1}
      +\sup_{\rho\in\L1(\reali^N;[0,R])}
      \norma{\div\mathcal{I}_1(\rho)-\div\mathcal{I}_2(\rho)}_{\L1}
    \right)
  \end{eqnarray*}
  we get
  \begin{displaymath}
    \norma{\rho_1-\rho_2(t)}_{\L1}
    \leq
    \norma{\rho_{0,1}-\rho_{0,2}}_{\L1}
    +
    \gamma_1 \, t \, e^{k t} \, a(t)
    +
    t \, \gamma_2
    +
    b(t) \int_0^t \norma{\rho_1-\rho_2(\tau)} \, \d{\tau}\,.
  \end{displaymath}
  By Gronwall Lemma, we obtain
  \begin{eqnarray*}
    \norma{\rho_1-\rho_2(t)}_{\L1}
    &\leq&
    \norma{\rho_{0,1}-\rho_{0,2}}_{\L1}
    +
    \gamma_1 \, t \, e^{k t} \, a(t)
    +
    t \, \gamma_2\\
    && +b(t)e^{\int_0^t b(u)\d{u}} \int_0^t e^{-\int_0^\tau b(u)\d{u}}
    (\norma{\rho_{0,1}-\rho_{0,2}}_{\L1}
    +
    \gamma_1  \tau e^{k \tau} \, a(\tau)
    +
    \tau \gamma_2)\d{\tau}\,.
  \end{eqnarray*}
  Noting that
  \[
  e^{-\int_0^t b(u)\d{u}}+b(t)\int_0^t e^{-\int_0^\tau
    b(u)\d{u}}\d{\tau}\leq \frac{b(t)}{b(0)}
  \]
  we get
  \begin{eqnarray*}
    \norma{\rho_1-\rho_2(t)}_{\L1}
    & \leq&
    \left(    \norma{\rho_{0,1}-\rho_{0,2}}_{\L1}
      +
      \gamma_1 \, t \, e^{k t} \, a(t)
      +
      t \, \gamma_2\right)\frac{b(t)}{b(0)}e^{t b(t)}\\
  \end{eqnarray*}
  completing the proof.
\end{proofof}

\subsection{Driving examples}

\begin{proofof}{Proposition~\ref{prop:invariance}}
  Let $\rho$ solve~\eqref{eq:General} in the sense of
  Definition~\ref{def:WS}. Define $w (t,x) = \nu (x) +
  \left(\mathcal{I} \left( \rho (t) \right) \right) (x)$.  Then,
  $\rho\in \C0([0,T]; \L1(\reali^N; [0,R]))$ is also a weak solution
  to~(\ref{eq:Kruzkov}). Then, set as above $q (\rho) = \rho \, v
  (\rho)$ and, for any $\phi\in \Cc\infty (\reali \times \reali^N;
  \reali)$, we have
  \begin{displaymath}
    \int_\reali\int_{\reali^N}
    \left(
      \rho \, \partial_t \phi
      +
      q(\rho) \, w(t,x)  \cdot \nabla\phi(t,x)
    \right)
    \d{x}\, \d{t} = 0\,.
  \end{displaymath}
  Fix positive $T$ and $M$. For $\alpha > 0$ define $Y_\alpha \in
  \C\infty(\reali; \reali)$ so that $Y_\alpha(s)=0$ for $s\leq 0$,
  $Y_\alpha (s)=1$ for $s\geq \alpha$ and $Y_\alpha' (s)\geq 0$ for
  all $s$. Let $\zeta(t) = Y_\alpha(t-\alpha)-Y_\alpha(t-T)$.

  Call $K_M = B (0,M) \setminus \Omega$. Set $\mu \in \Cc\infty
  ([-1,1]; [0,1])$ such that $\int_\reali \mu(r) \, \d{r}=1$. For
  $\theta > 0$, let $\mu_\theta(x) = \frac{1}{\theta^N} \,
  \mu\left(\norma{x} / \theta\right)$ and define $\psi = \caratt{K_M}
  * \mu_\theta$, the convolution being over all of $\reali^N$. Note
  that $\nabla\psi = 0$ on $\reali^N\setminus B(\partial K_M, \theta)$. If
  $x\in B(\partial K_M, \theta)$, we have
  \begin{displaymath}
    \nabla\psi(x)
    =
    \int_{K_M} \nabla\mu_\theta (x-y) \, \d{y}
    =
    -\int_{\partial K_M} \mu_\theta (x-y)\, \tilde n(y) \, \d{\sigma(y)} \,,
  \end{displaymath}
  where $\tilde n (x)$ is the outer normal to $K_M$ at any $x
  \in \partial K_M$.  Then, we choose $\phi(t,x) = \zeta(t) \,
  \psi(x)$. Clearly, $\phi \in \Cc\infty (\reali \times \reali^N;
  \reali)$. The definition of weak solution gives
  \begin{displaymath}
    0
    =
    \int_{\reali} \int_{\reali^N}
    \left(
      \rho(t,x)\, \zeta'(t) \, \psi(x)
      +
      q\left(\rho(t,x)\right) \, w(t,x) \, \zeta(t) \, \nabla\psi(x)
    \right)
    \d{x} \d{t}\,.
  \end{displaymath}
  Letting $\alpha \to 0$, we obtain
  \begin{eqnarray*}
    0
    & = &
    \int_{\reali^N}
    \left(\rho(0,x)-\rho(T,x) \right) \psi(x) \,\d{x}
    +
    \int_0^T\int_{\reali^N }
    q\left(\rho(t,x)\right) \, w(t,x) \, \nabla\psi(x) \, \d{x} \, \d{t}
    \\
    & = &
    \int_{\reali^N} \left(\rho(0,x)-\rho(T,x)\right)  \psi(x) \, \d{x}
    \\
    & &
    \qquad -
    \int_0^T\int_{\reali^N }\int_{\partial K_M}
    q\left(\rho(t,x)\right) w(t,x)\cdot \tilde n(y) \, \mu_\theta (x-y)\,
    \d{\sigma(y)} \, \d{x}\, \d{t}
    \\
    & = &
    \int_{\reali^N} \left(\rho(0,x)-\rho(T,x)\right)  \psi(x) \, \d{x}
    \\
    & &
    \qquad -
    \int_0^T \int_{\partial K_M} \int_{ B(0,1) }
    q\left(\rho(t,y+\theta z)\right) \, w(t,y+\theta z)\cdot \tilde n(y) \,
    \mu (\norma{z}) \,
    \d{z} \, \d{\sigma(y)} \, \d{t}\,.
  \end{eqnarray*}
  That is to say, setting $H_\theta (t) = \{(y,z):\; w(t,y+\theta z)
  \cdot \tilde n (y)\leq 0\}$ and $\partial\Omega_M = \partial\Omega \cap B
  (0,M)$
  \begin{eqnarray*}
    \!\!\!\!\!
    & &
    \int_{\reali^N} \left( \rho(0,x) - \rho(T,x) \right)  \psi(x) \, \d{x}
    \\
    \!\!\!\!\!\!\!
    & = &
    \int_0^T \int_{\partial K_M} \int_{ B(0,1) }
    q\left(\rho(t,y+\theta z)\right) \, w(t,y+\theta z)\cdot \tilde n (y) \,
    \mu (\norma{z})  \, \d{z} \, \d{\sigma(y)} \, \d{t}
    \\
    \!\!\!\!\!\!\!
    & \geq &
    \int_0^T \int_{\partial K_M} \int_{ B(0,1) }
    \caratt{H_\theta (t)} (y,z) \,
    q\left(\rho(t,y+\theta z)\right) \, w(t,y+\theta z) \cdot \tilde n (y) \,
    \mu (\norma{z}) \,
    \d{z} \, \d{\sigma(y)} \, \d{t}
    \\
    \!\!\!\!\!\!\!
    & = &
    \int_0^T \int_{\partial \Omega_M} \int_{ B(0,1) }
    \caratt{H_\theta (t)} (y,z) \,
    q\left(\rho(t,y+\theta z)\right) \, w(t,y+\theta z) \cdot \tilde n (y) \,
    \mu (\norma{z}) \,
    \d{z} \, \d{\sigma(y)} \, \d{t}
    \\
    \!\!\!\!\!\!\!
    & &
    +
    \int_0^T \!\!\! \int_{\partial K_M\setminus \partial \Omega_M} \!\! \int_{B(0,1)}
    \caratt{H_\theta (t)} (y,z) \,
    q\left(\rho(t,y+\theta z)\right) \, w(t,y+\theta z) \cdot \tilde n (y) \,
    \mu (\norma{z})
    \d{z} \d{\sigma(y)}  \d{t}
    \\
    \!\!\!\!\!\!\!
    & \geq &
    \int_0^T \int_{\partial \Omega_M} \int_{ B(0,1)}
    \caratt{H_\theta (t)} (y,z) \,
    \norma{q}_{\L\infty} \, w(t,y+\theta z) \cdot \tilde n (y) \, \mu (\norma{z})
    \, \d{z} \, \d{\sigma(y)} \, \d{t}
    +
    m
  \end{eqnarray*}
  where $m$ can be arbitrarily small provided $M$ is sufficiently
  large, since $\rho \in \L1 ([0,T] \times \reali^N; [0,R])$. Letting
  $\theta \to 0$, $M \to \infty$ and using~\eqref{eq:Nagumo}, by the
  Dominated Convergence Theorem we get
  \begin{displaymath}
    0
    \leq
    \int_{{}^c\Omega} \left(\rho(0,x)-\rho(T,x) \right) \d{x}
    =
    -\int_{{}^c\Omega} \rho(T,x) \, \d{x}\,,
  \end{displaymath}
  since by hypothesis $\spt \rho_0\subset\Omega$.  Besides, the
  positivity of $\rho$ gives us $ -\int_{{}^c\Omega} \rho(t,x)
  \d{x}\leq 0$.  Finally, we have $\int_{{}^c\Omega}\rho(T,x)=0$ and
  $\spt \rho(T)\subset \Omega$.
\end{proofof}

\begin{proofof}{Lemma~\ref{lem:Good}}
  Let $\rho\in \L1(\reali^N; [0,R])$. Note that $\mathcal{I}$ can be
  rewritten as $\mathcal{I}(\rho) = - \epsilon\, \rho * (\nabla\eta)
  \left/ \sqrt{1+\norma{\rho*\nabla \eta}^2}\right.$. By the
  assumptions on $\eta$ and the properties of the convolution product,
  we have
  \begin{eqnarray*}
    \norma{\mathcal{I}(\rho)}_{\L\infty}
    & \leq &
    \epsilon \, \norma{\rho}_{\L1} \norma{\nabla\eta}_{\L\infty}\,,
    \\
    \norma{\mathcal{I}(\rho)}_{\L1}
    & \leq &
    \epsilon \, \norma{\rho}_{\L1}\norma{\nabla\eta}_{\L1}\,,
    \\
    \norma{\nabla\mathcal{I}(\rho)}_{\L\infty}
    & \leq &
    \epsilon \,
    \norma{\rho}_{\L1}  \norma{\nabla \eta}_{\W1\infty}
    \left(
      1+
      R^2 \norma{\nabla\eta}_{\L1}
      \norma{\nabla^2\eta}_{\L1}
    \right) \, ,
    \\
    \norma{\nabla\mathcal{I}(\rho)}_{\L1}
    & \leq &
    \epsilon \, \norma{\rho}_{\L1}
    \norma{\nabla \eta}_{\W1\infty}
    \left(
      1+
      R^2 \norma{\nabla\eta}_{\L1}
      \norma{\nabla^2\eta}_{\L1}
    \right)\,.
  \end{eqnarray*}
  Hence~\textbf{(I.1)} is satisfied. Let us check now~\textbf{(I.2)}
  \begin{displaymath}
    \norma{\nabla^2 \mathcal{I}(\rho)}_{\L1}
    \leq
    \epsilon \, R^2 \, \norma{\rho}_{\L1}\norma{\nabla\eta}_{\W21}
    \norma{\nabla\eta}_{\W11}^2
    (1+ 4
    +
    3 R^2 \norma{\nabla\eta}_{\W11}^2
    )\,.
  \end{displaymath}
  Let $r_1, r_2\in \L1(\reali^N; [0,R])$. We have:
  \begin{eqnarray*}
    \norma{\mathcal{I}(r_1)-\mathcal{I}(r_2)}_{\L\infty}
    & \leq &
    \epsilon \, \left(
      1
      +
      R^2 \norma{\nabla \eta}_{\L1}^2
    \right)  \norma{\nabla\eta}_{\L\infty}
    \norma{r_1-r_2}_{\L1}\,,
    \\
    \norma{\mathcal{I}(r_1)-\mathcal{I}(r_2) }_{\L1}
    & \leq &
    \epsilon \, \left(
      1
      + R^2 \norma{\nabla \eta}_{\L1}^2
    \right)   \norma{\nabla\eta}_{\L1}
    \norma{r_1-r_2}_{\L1}\,,
    \\
    \norma{\div(\mathcal{I}(r_1)-\mathcal{I}(r_2) )}_{\L1}
    & \leq &
    \epsilon \, \norma{r_1-r_2}_{\L1} \norma{\nabla\eta}_{\W11}
    (1 + 8 R^2 \norma{\nabla \eta}_{\W11}^2 + 3R^4\norma{\nabla\eta}_{\W11}^4) \,,
  \end{eqnarray*}
  completing the proof.
\end{proofof}

\begin{proofof}{Proposition~\ref{prop:GoodP}}
  The fact that~\textbf{($\boldsymbol{\nu}$)} holds follows through
  simple computations from~\textbf{($\boldsymbol{\Omega}$)},
  \textbf{(g)}, \textbf{($\boldsymbol{\eta}$)},
  \textbf{($\boldsymbol{\alpha}$)} and the assumptions on
  $\partial\Omega$ or
  $\alpha,d_{\partial\Omega}$. Condition~\textbf{(I)} follows from
  Lemma~\ref{lem:Good}.

  To prove the invariance property \textbf{(P)}, we verify that
  \begin{displaymath}
    \left(  g (x) + \delta (x) + \left(\mathcal{I} (\rho)\right) (x)\right)
    \cdot
    n (x) \geq 0
  \end{displaymath}
  for all $x \in \partial\Omega$. By~\textbf{(g)}, \eqref{eq:delta}
  and~\eqref{eq:IGood}, it is sufficient to prove that
  \begin{displaymath}
    \lambda
    \geq
    \epsilon
    \sup_{x \in \partial\Omega}
    \frac{(\rho * \nabla \eta) (x) \cdot n (x)}{\sqrt{1 + \norma{(\rho*\nabla\eta) (x)}^2}} \,.
  \end{displaymath}
  Note that
  \begin{eqnarray*}
    \sup_{x \in \partial\Omega}
    \frac{(\rho * \nabla \eta) (x) \cdot n (x)}{\sqrt{1 + \norma{(\rho*\nabla\eta) (x)}^2}}
    & \leq &
    \sup_{x \in \partial\Omega}
    (\rho * \nabla \eta) (x) \cdot n (x)
    \leq
    \norma{\rho}_{\L\infty} \, \norma{\nabla\eta}_{\L1} \, \norma{n}_{\L\infty}
    \leq
    R \, \norma{\nabla\eta}_{\L1}
  \end{eqnarray*}
  completing the proof, by Proposition~\ref{prop:invariance}.
\end{proofof}

\begin{proofof}{Lemma~\ref{lem:PT}}
  Note that $\mathcal{I}$ as $\mathcal{I}(\rho) = \epsilon
  \int_{\reali^N} \rho(y) \, \nabla\left( \eta(x-y)\,
    g\left((y-x)\cdot \nu(x)\right) \right) \d{y}$.  Let $r_\eta$ be
  such that $B (0,r_\eta) \supset \spt \eta$. Fix $\rho\in
  \L1(\reali^N; [0,R])$. Let us check~\textbf{(I.1)}:
  \begin{eqnarray*}
    \norma{ \mathcal{I} (\rho) }_{\L\infty}
    & \leq &
    \epsilon \, \norma{\rho}_{\L1} \norma{\eta}_{\W1\infty}
    \norma{g}_{\W1\infty}
    \left(1+ (1+r_\eta) \norma{\nu}_{\W1\infty} \right)\,,
    \\
    \norma{ \mathcal{I} (\rho) }_{\L1}
    & \leq &
    \epsilon \, \norma{\rho}_{\L1} \norma{\eta}_{\W11}
    \norma{g}_{\W1\infty}
    \left(1+ (1+r_\eta) \norma{\nu}_{\W1\infty} \right) \,,
    \\
    \norma{ \nabla \mathcal{I} (\rho) }_{\L\infty}
    & \leq &
    \epsilon \, \norma{\rho}_{\L1}
    \norma{\eta}_{\W2\infty} \norma{g}_{\W2\infty}
    \\
    & &
    \times \left(
      1
      +
      2 (r_\eta+1) \norma{\nu}_{\W1\infty}
      +
      (1+r_\eta)^2 \norma{\nu}_{\W1\infty}^2
      +
      (2+r_\eta) \norma{\nabla\nu}_{\W1\infty}
    \right) ,
    \\
    \norma{ \nabla \mathcal{I} (\rho) }_{\L1}
    & \leq &
    \epsilon \,
    \norma{\rho}_{\L1} \norma{\eta}_{\W21} \norma{g}_{\W2\infty}
    \\
    & &
    \times
    \left(
      1
      +
      2 (r_\eta+1) \norma{\nu}_{\W1\infty}
      +
      (1+r_\eta)^2 \norma{\nu}_{\W1\infty}^2
      +
      (2+r_\eta) \norma{\nabla\nu}_{\W1\infty}
    \right).
  \end{eqnarray*}
  Passing to~\textbf{(I.2)}:
  \begin{eqnarray*}
    \norma{\nabla^2 \mathcal{I}(\rho)}_{\L1}
    & \leq &
    \epsilon \,
    \norma{\rho}_{\L1}
    \norma{\eta}_{\W31} \norma{g}_{\W3\infty}
    \\
    & &
    \times
    \Big[
    1
    +
    3 (1+r_\eta) \norma{\nu}_{\W1\infty}
    +
    3 (1+r_\eta)^2 \norma{\nu}_{\W1\infty}^2
    +
    3 (2+r_\eta) \norma{\nabla\nu}_{\W1\infty}
    \\
    & &
    \qquad
    +
    (1+r_\eta)^3 \norma{\nu}_{\W1\infty}^3
    +
    3 (2+r_\eta)^2 \norma{\nu}_{\W2\infty}^2
    +
    (3+r_\eta) \norma{\nabla^2\nu}_{\W1\infty}
    \Big] \,.
  \end{eqnarray*}
  Let $r_1, r_2 \in \L1(\reali^N; [0,R])$. The operator $\mathcal{I}$
  is linear in $\rho$, hence
  \begin{eqnarray*}
    \norma{\mathcal{I}(r_1)-\mathcal{I}(r_2)}_{\L\infty}
    & \leq &
    \epsilon \,\norma{r_1-r_2}_{\L1} \norma{\eta}_{\W1\infty} \norma{g}_{\W1\infty}
    (1+(1+ r_\eta) \norma{\nu}_{\W1\infty})\,,
    \\
    \norma{\mathcal{I}(r_1)-\mathcal{I}(r_2)}_{\L\infty}
    & \leq &
    \epsilon \, \norma{r_1-r_2}_{\L1} \norma{\eta}_{\W11} \norma{g}_{\W1\infty}
    (1+ (1+r_\eta) \norma{\nu}_{\W1\infty})\,,
    \\
    \norma{\div(\mathcal{I}(r_1)-\mathcal{I}(r_2))}_{\L\infty}
    & \leq &
    \epsilon \, \norma{r_1-r_2}_{\L1} \norma{\eta}_{\W21}
    \norma{g}_{\W2\infty}
    \\
    & &
    \times
    \Big[
    1
    +
    2 (1+r_\eta) \norma{\nu}_{\W1\infty}
    +
    (1+r_\eta)^2 \norma{\nu}_{\W1\infty}^2
    \\
    & &
    \qquad \qquad
    +
    (2+r_\eta)\norma{\nabla\nu}_{\W1\infty}
    \Big]
  \end{eqnarray*}
  showing that~\textbf{(I.3)} is satisfied.
\end{proofof}

\begin{proofof}{Proposition~\ref{prop:GoodPT}}
  Fix $\rho\in \L1(\reali^N; [0,R])$. Let us check~\textbf{(I.1)}
  and~\textbf{(I.2)}:
  \begin{eqnarray*}
    \norma{ \mathcal{I} (\rho) }_{\L\infty}
    & \leq &
    \epsilon \, \norma{\rho}_{\L1} \norma{\eta}_{\W1\infty}
    \norma{\phi}_{\W1\infty}
    \left(1+ (1+\ell) \norma{g}_{\W1\infty} \right)\,,
    \\
    \norma{ \mathcal{I} (\rho) }_{\L1}
    & \leq &
    \epsilon \, \norma{\rho}_{\L1} \norma{\eta}_{\W11}
    \norma{\phi}_{\W1\infty}
    \left(1+ (1+\ell) \norma{g}_{\W1\infty} \right) \,,
    \\
    \norma{ \nabla \mathcal{I} (\rho) }_{\L\infty}
    & \leq &
    \epsilon \, \norma{\rho}_{\L1}
    \norma{\eta}_{\W2\infty} \norma{\phi}_{\W2\infty}
    \\
    & &
    \times \left(
      1
      +
      2 (\ell+1) \norma{g}_{\W1\infty}
      +
      (1+\ell)^2 \norma{g}_{\W1\infty}^2
      +
      (2+\ell) \norma{\nabla g}_{\W1\infty}
    \right)
    \\
    & &
    \times
    \left(1 +
      R^2 \, \norma{\eta}_{\W11}^2
      \norma{\phi}_{\W1\infty}^2
      \left(1+ (1+\ell) \norma{g}_{\W1\infty} \right)^2
    \right) \,,
    \\
    \norma{ \nabla \mathcal{I} (\rho) }_{\L1}
    & \leq &
    \epsilon \,
    \norma{\rho}_{\L1} \norma{\eta}_{\W21} \norma{\phi}_{\W2\infty}
    \\
    & &
    \times
    \left(
      1
      +
      2 (\ell+1) \norma{g}_{\W1\infty}
      +
      (1+\ell)^2 \norma{g}_{\W1\infty}^2
      +
      (2+\ell) \norma{\nabla g}_{\W1\infty}
    \right)
    \\
    & &
    \times
    \left(1 +
      R^2 \, \norma{\eta}_{\W11}^2
      \norma{\phi}_{\W1\infty}^2
      \left(1+ (1+\ell) \norma{g}_{\W1\infty} \right)^2
    \right) \,.
  \end{eqnarray*}
  Passing to~\textbf{(I.2)}:
  \begin{eqnarray*}
    \norma{\nabla^2 \mathcal{I}(\rho)}_{\L1}
    & \leq &
    \epsilon \,
    \norma{\rho}_{\L1}
    \norma{\eta}_{\W31} \norma{\phi}_{\W3\infty}
    \\
    & &
    \times
    \Big[
    1
    +
    3 (1+\ell) \norma{g}_{\W1\infty}
    +
    3 (1+\ell)^2 \norma{g}_{\W1\infty}^2
    +
    3 (2+\ell) \norma{\nabla g}_{\W1\infty}
    \\
    & &
    \qquad
    +
    (1+\ell)^3 \norma{g}_{\W1\infty}^3
    +
    3 (2+\ell)^2 \norma{g}_{\W2\infty}^2
    +
    (3+\ell) \norma{\nabla^2g}_{\W1\infty}
    \Big]
    \\
    & &
    \times
    \left(1 +
      R^2 \, \norma{\eta}_{\W11}^2
      \norma{\phi}_{\W1\infty}^2
      \left(1+ (1+\ell) \norma{g}_{\W1\infty} \right)^2
    \right)
    \\
    & &
    +
    3 \,
    \epsilon \, R
    \norma{\eta}_{\W11}
    \norma{\phi}_{\W1\infty}
    \left(1+ (1+\ell) \norma{g}_{\W1\infty} \right)
    \\
    & &
    \times
    \norma{\rho}_{\L1}
    \norma{\eta}_{\W21} \norma{\phi}_{\W2\infty}
    \\
    & &
    \times \left(
      1
      +
      2 (\ell+1) \norma{g}_{\W1\infty}
      +
      (1+\ell)^2 \norma{g}_{\W1\infty}^2
      +
      (2+\ell) \norma{\nabla g}_{\W1\infty}
    \right)
    \\
    & &
    \times
    \left(1 +
      R^2 \, \norma{\eta}_{\W11}^2
      \norma{\phi}_{\W1\infty}^2
      \left(1+ (1+\ell) \norma{g}_{\W1\infty} \right)^2
    \right)
    \\
    & \leq &
    \epsilon \,
    \norma{\rho}_{\L1}
    \norma{\eta}_{\W31} \norma{\phi}_{\W3\infty}
    \\
    & &
    \times
    \Big[
    1
    +
    3 (1+\ell) \norma{g}_{\W1\infty}
    +
    3 (1+\ell)^2 \norma{g}_{\W1\infty}^2
    +
    3 (2+\ell) \norma{\nabla g}_{\W1\infty}
    \\
    & &
    \qquad
    +
    (1+\ell)^3 \norma{g}_{\W1\infty}^3
    +
    3 (2+\ell)^2 \norma{g}_{\W2\infty}^2
    +
    (3+\ell) \norma{\nabla^2g}_{\W1\infty}
    \\
    & &
    \qquad
    +
    3 \,
    R
    \norma{\eta}_{\W11}
    \norma{\phi}_{\W1\infty}
    \left(1+ (1+\ell) \norma{g}_{\W1\infty} \right)
    \\
    & &
    \qquad
    \times \left(
      1
      +
      2 (\ell+1) \norma{g}_{\W1\infty}
      +
      (1+\ell)^2 \norma{g}_{\W1\infty}^2
      +
      (2+\ell) \norma{\nabla g}_{\W1\infty}
    \right)
    \Big]
    \\
    & &
    \times
    \left(1 +
      R^2 \, \norma{\eta}_{\W11}^2
      \norma{\phi}_{\W1\infty}^2
      \left(1+ (1+\ell) \norma{g}_{\W1\infty} \right)^2
    \right)
    \,.
  \end{eqnarray*}
  Finally, \textbf{(I.3)} is proved exactly as in
  Lemma~\ref{lem:Good}.

  To prove the invariance property \textbf{(P)}, following the same
  procedure as in the proof of Proposition~\ref{prop:GoodP}, we check
  that $\left(\nu (x) + \left(\mathcal{I} (\rho)\right) (x)\right)
  \cdot n (x) \geq 0$ for all $x \in \partial\Omega$. In fact, if $x
  \in
  \partial\Omega$,
  \begin{eqnarray*}
    & &  \left(\nu (x) + \left(\mathcal{I} (\rho)\right) (x)\right) \cdot n (x)
    \\
    & = &
    \left(g (x) + \lambda \, n (x)
      -
      \epsilon \,
      \frac{
        \nabla \!\!
        \int_{\Omega} \rho(y) \, \eta(x-y)  \,
        \phi \!\left((y-x) \cdot g(x)\right)\d{y}
      }{\sqrt{1+
          \norma{
            \nabla \!\!
            \int_{\Omega} \rho(y) \, \eta(x-y)  \,
            \phi \!\left((y-x) \cdot g(x)\right)\d{y}}^2}
      }\right)
    \cdot n (x)
    \\
    & \geq &
    \lambda (x)
    +
    g (x) \cdot n (x)
    -
    \epsilon\, \norma{\phi}_{\L\infty}
    \int_\Omega \rho (y)
    \modulo{\nabla\left(\eta (x-y)\right) \cdot n (x)} \d{y}
    \\
    & &
    \quad
    -
    \epsilon
    \int_\Omega \rho (y) \, \eta (x-y) \, \phi' \left((y-x)\cdot g (x)\right)
    \left(
      (y-x) \, \nabla g (x) \, n (x)
      -
      g (x) \cdot n (x)
    \right)
    \d{y}
    \\
    & \geq &
    \lambda (x)
    -
    \epsilon\, R \, \norma{\phi}_{\L\infty} \, \norma{\nabla\eta}_{\L1}
    -
    \epsilon\, R\, \norma{\phi'}_{\L\infty}
    \int_\Omega  \eta (x-y) \,
    \modulo{(y-x) \, \nabla g (x) \, n (x)}
    \d{y} \,.
    \\
    & \geq &
    \lambda (x)
    -
    \epsilon\, R
    \left(
      \norma{\phi}_{\L\infty} \, \norma{\nabla\eta}_{\L1}
      +
      \norma{\phi'}_{\L\infty} \, \norma{\eta}_{\L1} \, \ell \,
      \norma{\nabla g}_{\L\infty}
    \right)
  \end{eqnarray*}
  completing the proof, by Proposition~\ref{prop:invariance}.
\end{proofof}

\begin{proofof}{Proposition~\ref{prop:PT}}
  Denote $\ell = \mbox{diam} \spt \eta$.  Straightforward computations
  give:
  \begin{eqnarray*}
    \norma{\mathcal{I} (\rho)}_{\L\infty}
    & \leq &
    \epsilon \,
    \norma{\rho}_{\L1} \,
    \norma{\nabla\eta}_{\L\infty} \,
    \norma{\phi}_{\L\infty} \,,
    \\
    \norma{\nabla\mathcal{I} (\rho)}_{\L\infty}
    & \leq &
    \epsilon \,
    \norma{\rho}_{\L1} \,
    \left(
      \norma{\nabla^2\eta}_{\L\infty} \,
      \norma{\phi}_{\L\infty}
      +
      \norma{\nabla\eta}_{\L\infty}\,
      \norma{\phi'}_{\L\infty}
      \left(
        \norma{g}_{\L\infty}
        +
        \ell \,
        \norma{\nabla g}_{\L\infty}
      \right)
    \right) \,,
    \\
    \norma{\mathcal{I} (\rho)}_{\L1}
    & \leq &
    \epsilon \,
    \norma{\rho}_{\L1} \,
    \norma{\nabla\eta}_{\L1} \,
    \norma{\phi}_{\L\infty} \,,
    \\
    \norma{\nabla\mathcal{I} (\rho)}_{\L1}
    & \leq &
    \epsilon \,
    \norma{\rho}_{\L1} \,
    \left(
      \norma{\nabla^2\eta}_{\L1} \,
      \norma{\phi}_{\L\infty}
      +
      \norma{\nabla\eta}_{\L1}\,
      \norma{\phi'}_{\L\infty}
      \left(
        \norma{g}_{\L\infty}
        +
        \ell \,
        \norma{\nabla g}_{\L\infty}
      \right)
    \right) \,,
    \\
    \norma{\nabla^2 \mathcal{I} (\rho)}_{\L1}
    & \leq &
    \epsilon \,
    \norma{\rho}_{\L1} \,
    \norma{\nabla\eta}_{\W21}
    \norma{\phi}_{\W2\infty}
    \left(
      1
      +
      3 \norma{g}_{\L\infty}
      +
      \norma{\nabla g}_{\L\infty}
      (3\ell +2)
      +
      \ell \norma{\nabla^2g}_{\L\infty}
    \right)
  \end{eqnarray*}
  giving~\textbf{(I.1)} and~\textbf{(I.2)}. The proof
  of~\textbf{(I.3)} is immediate by the linearity of $\mathcal{I}$ in
  $\rho$.
\end{proofof}

\appendix
\section{Appendix: Geometrical Issues Related to $\Omega$}
\label{app:G}

The framework presented in Section~\ref{sec:Des} can be adapted to
various real situations. For instance, the \emph{``walls''}
$\partial\Omega$ may stop the visibility of the individuals. Then,
from a modeling point of view, it can be reasonable to introduce the
set
\begin{equation}
  \label{eq:OmegaX}
  \Omega_x
  =
  \left\{
    y \in \Omega \colon
    x + \sigma (y-x) \in \Omega\,,
    \forall\, \sigma \in [0,1]
  \right\}
\end{equation}
of the points in $\Omega$ visible from $x$. Correspondingly, the
nonlocal operator $\mathcal{I}$ in~\eqref{eq:IGood} can be modified
intending the convolution as follows:
\begin{equation}
  \label{eq:ModifiedConvolution}
  (\rho*\eta) (x)
  =
  \int_{\Omega_x} \rho (y) \, \eta (x-y) \d{y} \,.
\end{equation}
The above relation means that the individual at $x$ evaluates an
average of the densities $\rho (y)$ at all values $y$ that are visible
from $x$. With these choices, the validity of condition~\textbf{(I)}
essentially depends on the geometry of $\Omega$. In particular, if
$\Omega$ is convex, then $\Omega_x = \Omega$ and~\textbf{(I)} holds by
Proposition~\ref{prop:GoodP}.

Here we only show how to choose a discomfort so that~\textbf{(P)}
holds
for~\eqref{eq:General}--\eqref{eq:IGood}--\eqref{eq:ModifiedConvolution}. The
case of the nonlocal operator~\eqref{eq:IPT} is entirely similar.

To this aim, introduce the set
\begin{equation}
  \label{eq:Hx}
  H_x
  =
  \left\{
    y \in \Omega \colon \nabla\eta (x-y) \cdot n (x) \geq 0
  \right\} \,.
\end{equation}

\begin{proposition}
  \label{prop:cond2}
  Assume that~\textbf{(\/$\boldsymbol{\Omega}$)}
  and~\textbf{($\boldsymbol{\eta}$)} hold.  Let $\mathcal{I}$ be as
  in~(\ref{eq:IGood}), where the convolution is intended as
  in~\eqref{eq:ModifiedConvolution}. If
  \begin{equation}
    \label{eq:DeltaA}
    \delta(x)
    =
    \epsilon \, R
    \int_{\Omega_x \cap H_x } \nabla \eta(x-y) \, \d{y}
    \quad
    \mbox{ for } x \in \partial\Omega
  \end{equation}
  then $\left(\delta(x)+\mathcal{I}(\rho)(x) \right)\cdot n(x)\geq 0$
  for all $x\in \partial\Omega$ and $\rho \in \L1(\reali^N; [0,R])$,
  so that~\textbf{(P)} is satisfied.

  If moreover $\eta (x) = \bar \eta (\norma{x})$ with $\bar\eta \in
  \Cc3 (\left[0, +\infty\right[;\reali^+)$ and $\bar\eta' \leq 0$,
  then $\Omega_x \subseteq H_x$ along $\partial\Omega$ and the
  discomfort can be defined on all $\Omega$ by
  \begin{displaymath}
    \delta(x) = \epsilon \, R \int_{\Omega_x} \nabla \eta(x-y) \, \d{y}
  \end{displaymath}
  satisfying to property~\textbf{(P)}.
\end{proposition}

\begin{proof}
  We have $\mathcal{I}(\rho)(x) = -\epsilon
  \int_{\Omega_x}\rho(y)\nabla\eta(x-y)\d{y} \left/
    \sqrt{1+\norma{\nabla \rho*\eta}^2} \right.$. Hence, to ensure
  that $(\delta(x)+\mathcal{I}(\rho)(x))\cdot n(x)\geq 0$ for all
  $x\in \partial\Omega$ and $\rho\in \L1(\reali^N; [0,R])$ we require
  \begin{eqnarray*}
    \delta(x)\cdot  n(x)
    &\geq &
    \epsilon
    \sup_{\rho\in \L1(\reali^N;[0,R])}
    \int_{\Omega_x} \rho(y) \, \nabla\eta(x-y)\cdot n(x) \, \d{y} \,.
  \end{eqnarray*}
  In the latter expression, the supremum is attained for $\rho = R \,
  \caratt{\Omega_x \cap H_x}$, since $ \Omega_x\cap H_x = \{y\in
  \Omega_x\,;\; \nabla\eta(x-y)\cdot n(x)\geq 0\}$. It is thus
  sufficient to define $\delta $ so that
  \begin{eqnarray*}
    \delta(x)\cdot  n(x)
    & \geq &
    \epsilon \, R
    \int_{ \Omega_x \cap H_x } \nabla \eta(x-y)\cdot n(x) \, \d{y}\,,
  \end{eqnarray*}
  proving that~\eqref{eq:DeltaA} implies~\textbf{(P)}, by
  Proposition~\ref{prop:invariance}.

  If $\eta(x) = \bar\eta(\norma{x})$, then $\nabla\eta(x) =
  \bar\eta'(\norma{x}) \, \frac{x}{\norma{x}}$. Since $\bar \eta' \leq
  0$,
  \begin{eqnarray*}
    H_x
    & \supseteq &
    \left\{
      y\in \Omega_x \colon (x-y)\cdot n(x)\leq 0
    \right\}
    \\
    & = &
    \left\{
      y\in \Omega_x \colon
      y = x + a \, n(x) + \omega \mbox{ with }
      a \geq0 \mbox{ and } \omega\in n(x)^\perp
    \right\}
    \\
    & \supseteq &
    \Omega_x \,.
  \end{eqnarray*}
  The latter inclusion holds since $\Omega_x$ is a convex set
  contained in $\Omega$ and containing $x$.
\end{proof}

The choice~\eqref{eq:IGood}--\eqref{eq:ModifiedConvolution} is
appealing from the modeling point of view, but not easily tractable
from both the numerical and the analytical points of view, without
major restrictions on the geometry of $\Omega$ and on
$\eta$. Therefore, we consider also the following choice.

\begin{proposition}
  \label{prop:cond}
  Assume that~\textbf{(\/$\boldsymbol{\Omega}$)}
  and~\textbf{($\boldsymbol{\eta}$)} hold.  Let $\mathcal{I}$ be as
  in~(\ref{eq:IGood}), with the usual
  convolution~\eqref{eq:convolution}. If
  \begin{displaymath}
    \delta(x)
    =
    \epsilon \, R
    \int_{H_x} \nabla \eta(x-y) \, \d{y}
    \qquad \mbox{ for } x \in \partial\Omega
  \end{displaymath}
  then $\left(\delta(x)+\mathcal{I}(\rho)(x) \right) \cdot n(x)\geq 0$
  for all $x\in \partial\Omega$ and $\rho \in \L1(\reali^N; [0,R])$, so
  that~\textbf{(P)} holds.
\end{proposition}

With this choice, a way to define $\delta$ on all $\Omega$ could be
\begin{displaymath}
  \delta (x)
  =
  \epsilon \, R \, \alpha(x)
  \int_{H_x} \nabla \eta(x-y) \, \d{y} \,,
\end{displaymath}
where $\alpha$ is as in~\textbf{($\boldsymbol{\alpha}$)}.

The proof is entirely similar to that of Proposition~\ref{prop:cond2}.

\small{

  \bibliography{pietons}

  \bibliographystyle{abbrv} }

\end{document}